\documentclass[11pt,a4paper]{article}
\usepackage[utf8]{inputenc}
\usepackage[T1]{fontenc}
\usepackage[francais]{babel}
\usepackage{amsmath, amscd, geometry, amsthm, amstext, amssymb, verbatim, version, mathrsfs, bbm, graphicx}
\usepackage[all, cmtip]{xy}
\usepackage{hyperref}\hypersetup{colorlinks=true, citecolor=black, linkcolor=black} 
\title{Invariants topologiques des Espaces non-commutatifs}

\theoremstyle{plain}
\newtheorem{prop}{Proposition}[section]
\newtheorem{cor}[prop]{Corollaire}
\newtheorem{lem}[prop]{Lemme}
\newtheorem{theo}[prop]{Théorème}
\newtheorem{df}[prop]{Définition}
\newtheorem{conj}[prop]{Conjecture}

\theoremstyle{definition}
\newtheorem{rema}[prop]{Remarque}
\newtheorem{ex}[prop]{Exemple}
\newtheorem{nota}[prop]{Notation}

\newenvironment{dedication}
    {\vspace{6ex}\begin{quotation}\begin{center}\begin{em}}
    {\par\end{em}\end{center}\end{quotation}}

\addto\captionsfrench{}

\newcommand{\Z}{\mathbb{Z}}
\newcommand{\N}{\mathbb{N}}
\newcommand{\C}{\mathbb{C}}
\newcommand{\s}{\mathbb{S}}

\newcommand{\lef}{\mathbb{L}}
\newcommand{\R}{\mathbb{R}}
\newcommand{\M}{\mathcal{M}}
\newcommand{\K}{\mathcal{K}}
\newcommand{\A}{\mathbf{A}}
\newcommand{\ocal}{\mathcal{O}}
\newcommand{\lmo}{\longrightarrow}
\newcommand{\lmos}[1]{\stackrel{#1} {\longrightarrow}}
\newcommand{\mo}{\rightarrow}

\newcommand{\moi}{\hookrightarrow}
\def\dar[#1]{\ar@<2pt>[#1]\ar@<-2pt>[#1]}
\newcommand{\mbb}[1]{\mathbb{#1}}
\newcommand{\mbf}[1]{\mathbf{#1}}
\newcommand{\mcal}[1]{\mathcal{#1}}

\newcommand{\mrm}[1]{\mathrm{#1}}

\newcommand{\tel}{\otimes^{\mathbb{L}}}
\newcommand{\sml}{\wedge^{\mathbb{L}}}
\newcommand{\te}{\otimes}
\newcommand{\sm}{\wedge}

\newcommand{\uh}{\underline{\mrm{h}}}
\newcommand{\uk}{\underline{k}}
\newcommand{\und}[1]{\underline{#1}}
\newcommand{\map}{\mathrm{Map}}
\newcommand{\Hom}{\mathrm{Hom}}
\newcommand{\End}{\mathrm{End}}
\newcommand{\uaut}{\underline{\mathrm{Aut}}}
\newcommand{\ugl}{\underline{\mathrm{Gl}}}
\newcommand{\uend}{\underline{\mathrm{End}}}
\newcommand{\uma}{\underline{\mathrm{M}}}

\newcommand{\homi}{\underline{\mathrm{Hom}}}
\newcommand{\rhomi}{\mathbb{R}\underline{\mathrm{Hom}}}
\newcommand{\U}{\mbb{U}}
\newcommand{\V}{\mbb{V}}
\newcommand{\W}{\mbb{W}}
\newcommand{\sps}{Sp_{S^1}}
\newcommand{\spss}{Sp_{S^2}}
\newcommand{\spcon}{Sp^{\mathrm{con}}}
\newcommand{\sinf}{\Sigma^{\infty}}

\newcommand{\mota}{\mbb{M}_{a}}
\newcommand{\mloc}{\mbb{M}_{loc}(k)}

\newcommand{\unit}{\mathbbm{1}}
\newcommand{\dgcat}{dgCat_{k}}
\newcommand{\dgcattf}{dgCat^{\mrm{tf}}_{k}}
\newcommand{\dgcatc}{dgCat_{\C}}

\newcommand{\dgmor}{dgMor_{k}}
\newcommand{\dgmorc}{dgMor_{\C}}
\newcommand{\affk}{\mathrm{Aff}_k}
\newcommand{\schk}{\mrm{Sch}_k}
\newcommand{\affc}{\mathrm{Aff}_{\C}}
\newcommand{\afflissc}{\mathrm{Aff^{liss}_{\C}}}
\newcommand{\calgc}{\mrm{CAlg}_{\C}}
\newcommand{\schc}{\mrm{Sch}_{\C}}

\newcommand{\lissc}{\mathrm{Sch^{liss}_{\C}}}

\newcommand{\spec}{\mathrm{Spec}}

\newcommand{\hlamb}{H\Lambda-Mod_\s}

\newcommand{\kc}{\tilde{\mathbf{K}}}
\newcommand{\kn}{\mathbf{K}}
\newcommand{\ka}{\mathbf{k}}
\newcommand{\ukc}{\und{\tilde{\mathbf{K}}}}
\newcommand{\ukn}{\und{\mathbf{K}}}

\newcommand{\kh}{KH}
\newcommand{\ukh}{\und{KH}}

\newcommand{\kcst}{\tilde{\mathbf{K}}^{\mathrm{st}} }
\newcommand{\kst}{\mathbf{K}^{\mathrm{st}} }
\newcommand{\ktop}{\mathbf{K}^{\mathrm{top}} }
\newcommand{\kctop}{\tilde{\mathbf{K}}^{\mathrm{top}} }
\newcommand{\ktopu}{K_{\mathrm{top}} }

\newcommand{\hc}{\mathrm{HC}}
\newcommand{\hcn}{\mathrm{HC^-}}
\newcommand{\hh}{\mathrm{HH}}
\newcommand{\hp}{\mathrm{HP}}
\newcommand{\hpa}{\mathrm{HP^{alg}}}
\newcommand{\uhp}{\und{\mathrm{HP}}}
\newcommand{\uhh}{\und{\mathrm{HH}}}
\newcommand{\uhcn}{\und{\mathrm{HC}}^-}
\newcommand{\hdrna}{\mathrm{H_{DR}^{naive}}}
\newcommand{\hdran}{\mathrm{H_{DR}^{naive,an}}}

\newcommand{\hb}{\mathrm{H_{B}}}
\newcommand{\hbuu}{\mathrm{H_{B}}[u, u^{-1}]}
\newcommand{\hbs}{\mathrm{H_{\s,B}}}

\newcommand{\ao}{\mathbf{A}^1}
\newcommand{\po}{\mathbf{P}^1}
\newcommand{\gm}{\mathbf{G}_m}
\newcommand{\cuu}{\C[u^{\pm 1}]}

\newcommand{\re}[1]{\vert #1\vert}
\newcommand{\ret}[1]{\vert #1\vert^{\mrm{top}}}
\newcommand{\resp}[1]{\vert #1 \vert_{\s}}
\newcommand{\redel}[1]{\vert #1 \vert_{\del}}
\newcommand{\regam}[1]{\vert #1 \vert_{\gam}}
\newcommand{\resh}[1]{\vert #1 \vert_{S^2}}

\newcommand{\del}{\Delta}
\newcommand{\gam}{\Gamma}
\newcommand{\sh}{\stackrel{h}{\amalg}}
\newcommand{\ph}{\stackrel{h}{\times}}
\newcommand{\wh}{\stackrel{h}{\wedge}}
\newcommand{\cone}{\mathrm{cone}}
\newcommand{\cocone}{\mathrm{cocone}}
\newcommand{\chc}{\mathrm{Ch_c}}
\newcommand{\ch}{\mathrm{Ch}}
\newcommand{\chst}{\mathrm{Ch^{st}}}

\newcommand{\chutop}{\mathrm{Ch_{utop}}}
\newcommand{\chtop}{\mathrm{Ch^{top}}}
\newcommand{\chctop}{\mathrm{Ch^{top}_c}}
\newcommand{\bul}{\bullet}
\newcommand{\bu}{\mathbf{bu}}
\newcommand{\BU}{\mathbf{BU}}

\newcommand{\spr}{SPr(\affc)}
\newcommand{\sprliss}{SPr(\afflissc)}
\newcommand{\sprs}{SPr(\schc)}

\newcommand{\spret}{SPr(\affc)^{\mathrm{\acute{e}t}}}
\newcommand{\sprlisset}{SPr(\afflissc)^{\mathrm{\acute{e}t}}}
\newcommand{\sprset}{SPr(\schc)^{\mathrm{\acute{e}t}}}
\newcommand{\sprslisset}{SPr(\lissc)^{\mathrm{\acute{e}t}}}

\newcommand{\shpro}{Sh_{\mrm{pro}}(\schc)}
\newcommand{\shproliss}{Sh_{\mrm{pro}}(\lissc)}
\newcommand{\sprpro}{SPr(\schc)^{\mathrm{pro}}}
\newcommand{\sprproliss}{SPr(\lissc)^{\mathrm{pro}}}
\newcommand{\sppro}{Sp^{\mathrm{pro}}}

\newcommand{\spafk}{Sp(\affk)}
\newcommand{\spaf}{Sp(\affc)}
\newcommand{\spafliss}{Sp(\afflissc)}
\newcommand{\spretao}{SPr^{\mathrm{\acute{e}t}, \ao}}
\newcommand{\spetao}{Sp^{\mathrm{\acute{e}t}, \ao}}

\newcommand{\splissna}{Sp^{\mathrm{Nis, \ao}}_{\mathrm{Liss}}}

\newcommand{\shc}{\mathcal{SH}_{\C}} 

\newcommand{\m}[1]{\widehat{#1}} 
\newcommand{\mpa}[1]{\widehat{#1}_{pe}}
\newcommand{\lpe}{\mathrm{L_{pe}}}
\newcommand{\parf}{\mathrm{Parf}}
\newcommand{\uparf}{\und{\mathrm{Parf}}}
\newcommand{\pspa}{\mathrm{PsParf}}
\newcommand{\upspa}{\und{\mathrm{PsParf}}}
\newcommand{\proj}{\mathrm{Proj}}
\newcommand{\uproj}{\und{\mathrm{Proj}}}
\newcommand{\psproj}{\mathrm{PsProj}}
\newcommand{\upsproj}{\und{\mathrm{PsProj}}}
\newcommand{\vect}{\mathrm{Vect}}

\newcommand{\sspsp}{ssp_{\s}}

\newcommand{\n}{\und{n}}
\newcommand{\tc}{\mbb{T}}
\newcommand{\bcl}{\mathbf{B}}

\begin{document}

\thispagestyle{empty}

\renewcommand*{\thefootnote}{\fnsymbol{footnote}}

\vspace*{-3cm}

\begin{tabular}{p{8cm}p{5cm}}
\hspace*{2cm}\includegraphics[scale=.7]{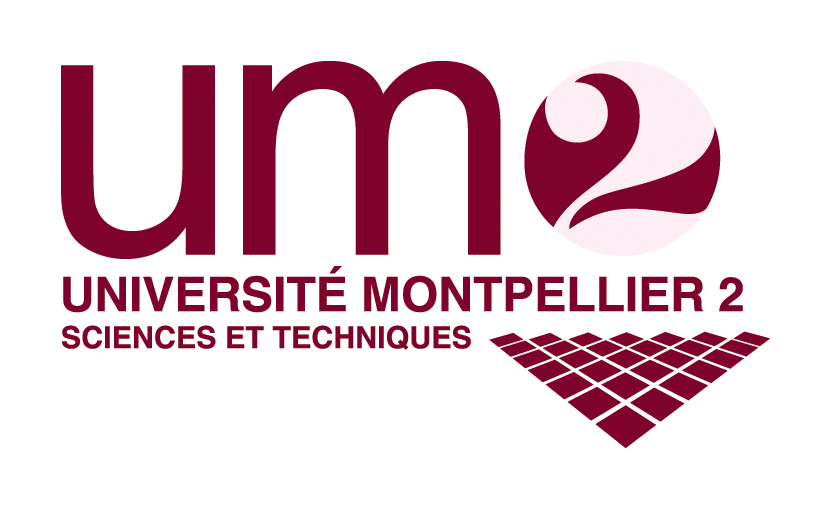}  & 
\vspace*{-2.7cm} \includegraphics[scale=.6]{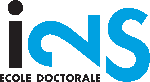} 
\end{tabular}

\vspace*{-.5cm}

\begin{center}
{\small
UNIVERSIT\'E MONTPELLIER II -- Sciences et Techniques du Languedoc.\\
\'Ecole doctorale I2S : Informations, Structures, Systèmes.}
\vspace*{0.3cm}

\rule{\textwidth}{.5pt}

\vspace*{0.3cm}

{\Large \textbf{TH\`ESE}} \\

\vspace*{0.5cm}

Présentée en vue d'obtenir le Grade de
 
\vspace*{0.5cm}

DOCTEUR DE L'UNIVERSIT\'E MONTPELLIER II \\
ET DE L'\'ECOLE DOCTORALE I2S\\
SPECIALIT\'E : MATH\'EMATIQUES

\vspace*{0.5cm}

Par \\
\vspace*{0.4cm}
\textbf{{\Large Anthony Blanc}} \\
\vspace*{0.4cm}
préparée à l'Institut de Mathématiques et de Modélisation de Montpellier.\\

\vspace*{0.4cm}

\rule{\textwidth}{.5pt}

\vspace*{.8cm}

\textbf{\LARGE{Invariants topologiques des Espaces non commutatifs.}}\\

\vspace*{.8cm}

Soutenue publiquement le 5 Juillet 2013, devant le Jury composé de

\vspace*{.4cm}

\end{center}

{\flushleft
\hspace*{-.5cm}
\begin{tabular}{lll}
M. Denis-Charles {\sc Cisinski} & \small{Professeur (UPS, Toulouse)} & Président du jury \\
M. Bernhard {\sc Keller} & \small{Professeur (Paris Diderot)} & Examinateur \\
M. Carlos {\sc Simpson}& \small{Directeur de Recherche (Nice Sophia-Antipolis)} & Examinateur \\
M. Bertrand {\sc Toën} & \small{Directeur de Recherche (UM2, Montpellier)} & Directeur de Thèse  \\
M. Michel {\sc Vaquié} & \small{Chargé de Recherche (UPS, Toulouse)} & Examinateur  \\
\end{tabular}
}
\begin{center}
Au vu des Rapports de
\end{center}

{\flushleft
\hspace*{-.5cm}
\begin{tabular}{ll}
M. Denis-Charles {\sc Cisinski} & \small{Professeur (UPS, Toulouse)} \\
M. Boris {\sc Tsygan} & \small{Professeur (Northwestern University, Evanston IL, USA)}  \\
\end{tabular}
}

\vspace*{0.5cm}

\rule{\textwidth}{.5pt}

\vspace*{.5cm}

{\small
\begin{tabular}{l}
Anthony Blanc\\
\texttt{anthony.blanc@math.univ-montp2.fr} \\
 I3M, Université Montpellier 2\\
Case Courrier 51\\
Place Eugène Bataillon\\
34095 Montpellier Cedex 5, France.
\end{tabular}
}

\renewcommand*{\thefootnote}{\alph{footnote}}

\newpage
\thispagestyle{empty}

\begin{abstract}
Dans cette thèse, on donne une définition de la K-théorie topologique des espaces non commutatifs de Kontsevich (c'est-à-dire des dg-catégories) définis sur les nombres complexes. L'introduction de ce nouvel invariant initie la recherche des invariants de nature topologique des espaces non-commutatifs, comme "simplifications" des invariants algébriques (K-théorie algébrique, homologie cyclique, périodique comme étudiés dans les travaux de Tsygan, Keller). La motivation principale vient de la théorie de Hodge non-commutative au sens de Katzarkov--Kontsevich--Pantev. 

En géométrie algébrique, la partie rationnelle de la structure de Hodge est donnée par la cohomologie de Betti rationnelle, qui est la cohomologie rationnelle de l'espace des points complexes du schéma. La recherche d'un espace associé à une dg-catégorie trouve une première réponse avec le champ (défini par Toën--Vaquié) classifiant les dg-modules parfaits sur cette dg-catégorie. La définition de la K-théorie topologique a pour ingrédient essentiel le foncteur de réalisation topologique des préfaisceaux en spectres sur le site des schémas de type fini sur les complexes. La partie connective de la K-théorie semi-topologique peut être définie comme la réalisation topologique du champ en monoïdes commutatifs des dg-modules parfaits. Cependant pour atteindre la K-théorie négative, on réalise le préfaisceau donné par la K-théorie algébrique non-connective. Un de nos résultats principaux énonce l'existence d'une équivalence naturelle entre ces deux définitions dans le cas connectif. On montre que la réalisation topologique du préfaisceau de K-théorie algébrique connective pour la dg-catégorie unité donne le spectre de K-théorie topologique usuel. Puis que c'est aussi vrai pour la K-théorie algébrique non-connective, en utilisant la propriété de restriction aux lisses de la réalisation topologique. En outre, cette propriété de restriction aux schémas lisses nécessite de montrer une généralisation de la descente propre cohomologique de Deligne, dans le cadre homotopique non-abélien.
La K-théorie topologique est alors définie en localisant par rapport à l'élément de Bott. Cette définition repose donc sur des résultats non-triviaux. On montre alors que le caractère de Chern de la K-théorie algébrique vers l'homologie périodique se factorise par la K-théorie topologique, donnant un candidat naturel pour la partie rationnelle d'une structure de Hodge non-commutative sur l'homologie périodique, ceci étant énoncé sous la forme de la conjecture du réseau. 

Notre premier résultat de comparaison concerne le cas d'un schéma lisse de type fini sur les complexes -- la conjecture du réseau est alors vraie pour les dg-catégories de complexes parfaits sur de tels schémas. On montre ensuite que cette conjecture est vraie dans le cas des algèbres associatives de dimension finie. 
\end{abstract}

\vspace*{1cm}

Mots clés : DG-catégories, K-théorie algébrique, K-théorie topologique, Théorie homotopique des schémas.

\newpage
\thispagestyle{empty}

\renewcommand{\abstractname}{Abstract}

\begin{center}
\rule{\textwidth}{.5pt}
\vspace*{.2cm}

\textbf{{\Large Topological Invariants of noncommutative Spaces.}}
\vspace*{.2cm}

\rule{\textwidth}{.5pt}
\end{center}

\vspace*{1cm}

\begin{abstract}
In this thesis, we give a definition of a topological K-theory of Kontsevich's non-commutative spaces (i.e. of dg-categories) defined over complex numbers. The introduction of this invariant initiates the quest for topological invariants of non-commutative spaces, which are considered as "simplifications" of algebraic ones like algebraic K-theory, cyclic homology, periodic homology as studied by Tsygan, Keller. The main motivation comes from non-commutative Hodge theory in the sense of Katzarkov--Kontsevich--Pantev. 

In algebraic geometry, the rational part of the Hodge structure is given by rational Betti cohomology, which is the rational cohomology of the underlying space of complex points. The existence of a space associated to a dg-category admits a first answer given by the stack (defined by Toën--Vaquié) classifying perfect dg-modules over this dg-category. The essential ingredient in the definition of the topological K-theory is the topological realization functor of spectral presheaves on the site of complex schemes of finite type. The connective part of the semi-topological K-theory can then be definied as the topological realization of the stack of perfect dg-modules over the space, together with its commutative monoid structure up to homotopy. But to deal with negative K-groups, we realize the presehaf given by non-connective algebraic K-theory. One of our main results relies the two previous definition in the connective case. We show that the topological realization of the presheaf of connective algebraic K-theory for the unit dg-category is equivalent to the usual topological K-theory spectrum. We show this is also true in the non-connective case, using a property of restriction to smooth schemes. This last property leads us to show a generalization of Deligne's proper cohomological descent to the homotopical non-abelian setting. This enables us to define topological K-theory by inverting the Bott element. We point out that the process of the definition involves non-trivial results. We then show that the Chern character from algebraic K-theory to periodic homology factorizes through topological K-theory, giving a natural candidate for the rational part of a non-commutative Hodge structure on the periodic homology of a smooth and proper dg-category. This last claim is written in the form of a conjecture : the lattice conjecture. 
Our first comparison result deals with the case of a smooth scheme of finite type over complex numbers -- we show the lattice conjecture holds for dg-categories of perfect complexes. We also show this conjecture is true in the case of finite dimensional associative algebras. 
\end{abstract}

\vspace*{1cm}

Key Words : DG-categories, Algebraic K-theory, Topological K-theory, Homotopy theory of schemes.

\newpage

\begin{dedication}
\large{à mon frère Nicolas !...}
\end{dedication}

\vspace*{5cm}

\begin{verse}
\textit{Où vont les courses folles \\
Où vont les courses folles \\
Contre la lumière \\
Soleil\\
Soleil\\
Creuset des larmes d'or\\
Dans la plaine la plaie pour la nuit brille encore \\
Pareil\\
Pareil\\
Aux chants des oiseaux morts\\
Aux sons des astres oubliés\\
A tous les coeurs dehors...\\
Soleil ! \\
Soleil !}\\

Bertrand Cantat, Wajdi Mouawad, \textit{Dithyrambe au Soleil ; Choeurs}.
\end{verse}

\newpage
\thispagestyle{empty}

\begin{center}
\rule{\textwidth}{.5pt}
\vspace*{.2cm}

\textbf{{\LARGE Invariants topologiques des Espaces non commutatifs.}}
\vspace*{.2cm}

\rule{\textwidth}{.5pt}
\end{center}

\newpage
\thispagestyle{empty}

\tableofcontents

\newpage
\thispagestyle{empty}

\setcounter{section}{-1}
\section{Introduction}
Étant donné un objet d'origine géométrique (variété $\mathscr{C}^\infty$ ou algébrique, espace topologique...), une des attitudes fécondes que les mathématiciens ont adoptées est de lui associer un objet de nature purement  <<algébrique>> comme un groupe, un anneau ou une catégorie. Après l'apparition des invariants numériques (la dimension, la caractéristique d'Euler, l'invariant de Hopf...), c'est au cours du développement de la topologie algébrique des années 30 qu'apparaissent les premières structures algébriques attachées à un espace avec le groupe fondamental de Poincaré, l'anneau de cohomologie, etc. Cette dualité espace/algèbre a été l'idée directrice d'une repensée de la géométrie et de la topologie au XXième siècle, notamment avec les travaux de Gelfand en topologie et de Grothendieck en géométrie algébrique. Ce dernier repense alors la notion d'espace en regardant la catégorie des faisceaux comme faisant office d'espace elle-même : les topos font leurs apparitions, étendant la dualité espace/algèbre en une dualité espace/catégorie. Ce point de vue est alors très fructueux, permettant l'étude systématique de toutes les variantes de catégories dérivées d'un schéma et des théories cohomologiques de saveurs possiblement arithmétiques apparaissant dans ce domaine. Mais les catégories dérivées apparaissent aussi dans d'autres types de géométries, analytique, différentielle, et même l'analyse microlocale et la théorie des $D$-modules. Le point de vue de la géométrie non-commutative \emph{à la} Connes prend le contrepied de la dualité de Gelfand : étant donnée une $C^*$-algèbre non nécessairement commutative, étudions-la comme l'algèbre des fonctions continues sur un espace <<non commutatif>> hypothétique. Le point de vue de Kontsevich prend le contrepied de l'étude d'une variété par le biais de sa catégorie dérivée : étant donné une catégorie triangulée, étudions-la comme la catégorie dérivée d'un espace <<non commutatif>> hypothétique. C'est dans ce cadre de la géométrie algébrique dîte <<non-commutative>> ou <<catégorique>> de Kontsevich que se place ce texte. Pour nous un espace non commutatif n'est rien d'autre qu'une catégorie triangulée (ou mieux : une catégorie différentielle graduée). Plusieurs invariants algébriques des espaces non commutatifs ont été très largement étudiés comme la K-théorie algébrique, les (co)homologies de Hochschild, cycliques, l'algèbre de Hall, (voir \cite{kedg}). Ce travail de thèse répond à la problématique de la recherche d'invariants de nature topologique associés aux espaces non commutatifs et notamment d'une K-théorie topologique dans ce contexte. 

En géométrie algébrique, la puissance de la généralisation apportée par la théorie des schémas offre la liberté de considérer des objets de nature a priori distincte sur la même scène, comme les anneaux de fonctions et les anneaux de nombres, et d'étudier ceux-ci avec les mêmes techniques géométriques. 
Le point de vue de Kontsevich avec la géométrie des catégories déploie une telle puissance d'unification : on peut considérer sur la même scène des objets provenant de géométries de nature a priori distincte, parmi lesquels on cite les exemples suivants.
\begin{itemize}
\item En géométrie algébrique, le cas des variétés algébriques semble être le plus étudié jusqu'à présent avec la catégorie dérivée quasi-cohérente (voir par exemple \cite{bondalorlovicm}, \cite{orlovdereq}). 
\item En géométrie symplectique, on peut citer la catégorie de Fukaya associée à une variété symplectique (voir par exemple \cite{fooo}, \cite{seidelfuk}, \cite{kontssymp}, \cite{abouwrap}). 
\item En théorie des représentations, la catégorie triangulée des complexes de représentations d'un carquois (voir par exemple \cite{kellerclus}, \cite{krauserep}, \cite{hrstilt}). 
\item En théorie des singularités, la catégorie des factorisations matricielles associée à une singularité isolée d'hypersurface apparues dans \cite{eisenbudmf}, et étudiées ensuite dans une autre perspective par Kontsevich pour décrire les catégories de $D$-branes de la théorie des cordes (\cite{orlovderivedcatsing}, \cite{tobimf}, \cite{efimovmf}). 
\item En analyse algébrique, les catégories de DQ-modules intervenant dans la théorie de la déformation par quantification des variétés de Poisson (\cite{kontsdefq}, \cite{shapiradef}, \cite{poleshapidefq}, \cite{petitdgaff}). 
\item En théorie de Hodge non-abélienne, les catégories dérivées de fibrés vectoriels plats, de fibrés vectoriels holomorphes et de système locaux sur une variété complexe (voir par exemple \cite{ldgcat}). 
\item En topologie algébrique, l'algèbre des chaînes sur un groupe de lacets (\cite{sat}), et on peut citer aussi en théorie de l'homotopie rationnelle les algèbres (commutatives) qui codent le type d'homotopie rationnel d'un espace simplement connexe. 
\end{itemize}
De ce point de vue, la géométrie des catégories est donc le point de rencontre de toutes les géométries. En particulier le point de rencontre de la géométrie algébrique et symplectique, ce qui permet de formuler des conjectures pour tenter de comprendre d'un point de vue catégorique le phénomène de \emph{symétrie miroir} jouant un rôle important en théorie des cordes (voir \cite{kontshom}, \cite{kontssoi}, \cite{kkp}). 
\\

L'étude des catégories dérivées s'est développée à travers leur structure triangulée. Il est maintenant compris que la structure de catégorie triangulée se comporte de manière trop pathologique dans différents contextes, et que celles-ci souffrent d'un défaut de stabilité par construction catégorique telle que le hom interne. Dans cette thèse nous travaillons avec la théorie homotopique des dg-catégories\footnote{a.k.a. catégories différentielles graduées.} qui a été initié par Tabuada (\cite{tabth}) et développée aussi par Toën (\cite{dgmor}) et Toën--Vaquié (\cite{modob}). Notons qu'il existe des variantes pour étudier les catégorie dérivées en les considérant comme des $A_\infty$-catégories ou des $\infty$-catégories stables linéaires. Ces trois approches sont supportées par l'existence d'une bonne théorie homotopique (à équivalence de Morita près) des dg-catégories, des $A_\infty$-catégories, et des $\infty$-catégories stables linéaires. Il est communément conjecturé qu'elles sont deux à deux équivalentes\footnote{On pourra consulter \cite[§6.2]{marco1} concernant le lien entre dg-catégories et $\infty$-catégories stables linéaires à équivalences Morita près.}. 

Dans le langage des dg-catégories, le passage du monde géométrique commutatif dans le monde catégorique est donné par un foncteur
$$
\lpe : \mrm{Vari\acute{e}t\acute{e}s}/k\lmo \mrm{dg-Cat\acute{e}gories}/k^{op}
$$
tel que pour toute variété algébrique $X$, la dg-catégorie $\lpe(X)$ est équivalente à la dg-catégorie formée des complexes parfaits de faisceaux quasi-cohérents. Cette situation soulève un ensemble de questions sur la possibilité d'étendre les propriétés et les invariants classiques que l'on connaît pour les variétés, aux espaces non commutatifs. Un certain nombre de ces questions a déjà trouvé des réponses complètes tandis que d'autres semblent être d'une difficulté bien supérieure. Les notions de lissité, de propreté, de type fini ont leur analogue non commutatif (tel que défini dans \cite{kontssoi}, \cite{modob}).
Les dg-catégories du type $\lpe(X)$ pour $X$ une variété vérifient une condition forte : elles ont un générateur compact (\cite[Thm.3.3.1]{bvgen}), chacune d'elles est donc équivalente à une dg-algèbre. De plus $\lpe(X)$ partage les mêmes propriétés que la variété $X$ (être propre, lisse...). Une idée directrice dans l'étude des espaces non commutatifs (dans le travail de Kontsevich) et dans leur théorie motivique peut s'exprimer par le slogan suivant. 
\\

 Slogan : \emph{Les espaces non commutatifs propres et lisses partagent toutes les bonnes propriétés des variétés algébriques propres et lisses}. 
\\

En d'autres termes, ce slogan reflète la vision qu'ont les experts du comportement des espaces non commutatifs à travers les exemples qui abondent et les calculs que l'ont sait mettre en œuvre. Par le théorème de Hochschild--Konstant--Rosenberg, la cohomologie de de Rham d'une variété lisse a pour analogue non commutatif l'homologie cyclique périodique (tel que défini par Tsygan \cite{tsyhom}) et la cohomologie de Hodge se traduit en l'homologie de Hochschild. Une structure importante que supporte la cohomologie de de Rham d'une variété propre et lisse est sa structure de Hodge, qui s'est révélé être un outil puissant en géométrie algébrique, intervenant dans la théorie motivique des variétés algébriques et dans la formulation mathématique de la symétrie miroir. Une définition de structure de Hodge non-commutative a été proposé dans l'article fondateur \cite[Déf.2.5]{kkp}, mais son existence dans le cas général d'un espace non commutatif propre et lisse fait l'objet des conjectures \cite[Conj.2.24]{kkp}, \cite[Conj.8.6]{kal} qui semblent relativement difficiles à l'heure actuelle. Son existence est supportée par la symétrie miroir dans des cas particuliers (voir \cite[§3]{kkp}) et bien sûr par le cas (commutatif) d'une variété algébrique lisse. 

Nous ne rentrons pas dans les détails techniques de la définition d'une structure de Hodge non-commutative (abrégé en shnc) (\cite[Déf.2.5]{kkp}). Contentons nous de dire de manière heuristique, qu'une telle structure peut être décrite par deux types de données : la <<partie de Rham>> (ou partie <<fibré vectoriel>>) et la <<partie Betti>>  qui correspondent à la filtration de Hodge et à la partie rationnelle respectivement dans le cas classique d'une variété algébrique lisse. La conjecture de dégénérescence (\cite[§2.2.4]{kkp}) donne une condition suffisante pour l'existence de la partie de Rham d'une shnc sur l'homologie périodique d'un espace non commutatif propre et lisse. Cette conjecture est supportée par l'existence d'une connexion de Gauss--Manin sur l'homologie cyclique périodique (\cite{tsygm}). Elle a été prouvé en imposant des hypothèses de finitude sur la dg-algèbre (\cite{kal}, \cite{shkly}) et dans le cas des dg-catégories de factorisations matricielles (\cite{tobimf}) qui est en lien direct avec le côté $B$ de la symétrie miroir. 

Ce travail de thèse concerne la partie Betti de la shnc hypothétique d'un espace non commutatif propre et lisse. Bondal et Toën ont proposé d'atteindre cette partie rationnelle en construisant une \emph{K-théorie topologique des espaces non commutatifs} \cite[§2.2.6]{kkp}.

\subsection{Description des résultats principaux}

L'objectif premier de cette thèse est de proposer une définition de K-théorie topologique des espaces non commutatifs définis sur $\C$ et de prouver que le caractère de Chern de la K-théorie algébrique vers l'homologie périodique se factorise de façon naturelle par la K-théorie topologique. La conséquence principale de ce résultat est l'existence d'un candidat naturel pour la structure rationnelle sur l'homologie périodique d'un espace non commutatif propre et lisse.

Si $T$ est un espace non commutatif et qu'on note $\kn$ le foncteur de K-théorie algébrique des espaces non commutatifs dans les spectres, alors $T$ détermine un préfaisceau en spectre
$$\ukn(T) : \spec(A)\longmapsto \kn(T\tel A).$$
La définition de la K-théorie topologique se construit en deux étapes : appliquer un certain foncteur de réalisation topologique au préfaisceau $\ukn(T)$ et inverser le générateur de Bott. Pour être en mesure d'inverser l'élément de Bott, nous devons d'abord calculer la K-théorie topologique du point (que l'on montre être équivalente au spectre $\bu$), et considérer la K-théorie semi-topologique comme un module sur l'anneau en spectre $\bu$.  

Plus précisément, si $\spr$ désigne la catégorie des préfaisceaux simpliciaux sur les schémas affines de type fini sur $\C$, et $SSet$ la catégorie des ensembles simpliciaux, la réalisation topologique est un foncteur 
\begin{align*}
\re{-} : \spr & \lmo SSet \\
 E & \longmapsto \re{E} 
\end{align*}
qui est l'extension naturelle du foncteur <<espace topologique des points complexes>> associé à un schéma de type fini sur $\C$. Ce foncteur admet une extension naturelle aux préfaisceaux en spectres et une version dérivée relativement à la théorie homotopique étale locale et $\ao$-localisé sur $\spr$. La notation $\re{-}$ réfère alors au foncteur dérivé. La réalisation topologique apparait dans les papiers \cite{simpsontop} et \cite[§3.3]{mv} au niveau de la théorie $\ao$-homotopique des schémas. Le foncteur de réalisation topologique (et ses propriétés de descente) est l'ingrédient clé dans la définition de la K-théorie topologique. 

\begin{df} \emph{(voir déf.\ref{defkst})} ---
Pour tout espace non commutatif $T$ sur $\C$, la K-théorie semi-topologique de $T$ est le spectre 
$$\kst(T)=\re{\ukn(T)}.$$
\end{df} 

La terminologie <<semi-topologique>> est empruntée à Friedlander--Walker, voir \cite{fwcomp} dans lequel les auteurs définissent une théorie équivalente pour les variétés algébriques complexes quasi-projectives. Un résultat de comparaison sera énoncé plus bas. Il existe en outre une version <<connective>> $\kcst(T)$ de la définition précédente obtenue en réalisant le préfaisceau de K-théorie algébrique connective. Ce spectre est digne d'intérêt puisque l'on ne sait pas a priori s'il est équivalent au revêtement connectif de $\kst(T)$. Cependant sous des hypothèses de lissité raisonnable sur $T$, on s'attend à ce que qu'ils soient équivalents (par exemple en montrant directement l'annulation de la K-théorie algébrique négative pour les espaces non commutatifs propres et lisses, voir la partie applications). 

Il existe un morphisme naturel $\kc(T)\lmo \kcst(T)$ donné par adjonction (où la notation tilde désigne la K-théorie connective). En particulier sur le $\kc_0$ et dans le cas d'une variété $X$, on donne une formule pour $\kcst_0(X)$ qui identifie ce groupe comme le quotient de $\kc_0(X)$ par la relation d'équivalence algébrique, c'est-à-dire que deux fibrés sont équivalents s'ils sont reliés par une courbe algébrique connexe, formule analogue à celle de Friedlander--Walker. 

Nos premiers résultats concernent le cas de la dg-catégorie ponctuelle $T=\unit$  (avec un seul objet et $\C$ comme algèbre d'endomorphismes). 

\begin{theo} \emph{(voir thm.\ref{bu})} --- Le spectre $\kcst(\unit)$ est équivalent au spectre $\bu$ de K-théorie topologique connective usuelle. 
\end{theo}

\begin{theo} \emph{(voir thm.\ref{annupoint})} --- Pour toute $\C$-algèbre commutative lisse $B$, on a $\kst_{-i}(B)=0$ pour tout $i>0$. En particulier on a une équivalence $\kst(\unit)\simeq \bu$. 
\end{theo}

Comme suggéré par Kontsevich, les espaces non commutatifs ont une théorie motivique, comme développé dans le travail de Tabuada (\cite{tabhkt}, \cite{tabgui}) et on doit être en mesure de comparer les motifs de variétés algébriques avec ceux associés aux espaces non commutatifs (voir aussi le travail récent de Robalo \cite{marco1}). En utilisant le théorème de Cisinski--Tabuada sur la coreprésentabilité de la K-théorie algébrique dans la catégorie motivique de Tabuada (\cite[Thm.7.16]{nck}), on donne un modèle de $\kst(T)$ qui est un module sur l'anneau en spectre $\bu$. Ceci nous permet alors de définir la K-théorie topologique ($\beta\in\pi_2\bu $ désigne un générateur de Bott). 

\begin{df} \emph{(voir déf.\ref{deftop})} ---
$\ktop(T)=\kst(T)[\beta^{-1}]$. 
\end{df} 

Cette définition fait écho au travail de Thomason dans \cite[Thm.4.11]{thomet} dans lequel la K-théorie topologique à coefficients finis des points complexes d'une variété est donnée par la K-théorie algébrique à coefficients finis après inversion de l'élément de Bott. Elle réfère aussi au travail de Friedlander--Walker \cite{fwcomp} dans lequel la K-théorie topologique des points complexes d'une variété projective lisse est donné par la K-théorie semi-topologique au sens de Friedlander--Walker après inversion de l'élément de Bott. Notons qu'on peut considérer une variante de cette définition en inversant le générateur de Bott dans la K-théorie semi-topologique connective, il en résulte une théorie que l'on nomme K-théorie topologique pseudo-connective et qu'on note $\kctop(T)$. 

Une question naturelle est de demander que dans le cas $T=\lpe(X)$ d'une variété raisonnable, la K-théorie topologique $\ktop(\lpe(X))$ redonne la K-théorie topologique des points complexes. C'est ce que nous prouvons en guise de premier exemple en utilisant la dualité de Spanier--Whitehead dans la catégorie homotopique stable des schémas de Morel--Voevodsky, comme prouvé par Riou dans le cas d'une variété lisse. 

\begin{prop} \emph{(voir prop.\ref{compvarlisse})} ---
Pour $X$ un $\C$-schéma lisse de type fini, on a une équivalence $\ktop(\lpe(X))\simeq \ktopu(sp(X))$, où le dernier objet désigne la K-théorie topologique usuelle des points complexes de $X$. 
\end{prop}

Le résultat annoncé sur le caractère de Chern topologique est le suivant. Ici $\hcn(T)$ désigne l'homologie cyclique négative de l'espace non commutatif $T$ et $\hp(T)$ son homologie cyclique périodique. 

\begin{theo} \emph{(voir thm.\ref{carac})} ---
Le caractère de Chern algébrique $\kn(T)\lmo \hcn(T)$ se factorise par $\ktop(T)$ et on a alors un carré commutatif 
$$\xymatrix{\kn(T)\ar[r]^-{\ch} \ar[d]&\hcn(T) \ar[d] \\ \ktop(T)\ar[r]^-{\chtop}& \hp(T). }$$
\end{theo}

Il existe un intermédiaire comme candidat d'un <<espace>> associé à une dg-catégorie $T$ donnée. Le mot espace est à prendre en un sens plus large, puisqu'il s'agit d'un champ supérieur dérivé (au sens de \cite{hag2}) noté $\M_T$ défini par Toën--Vaquié dans \cite{modob}. Lorsque $T$ satisfait des conditions de finitude raisonnable, le champ $\M_T$ peut être considéré comme un espace géométrique lui-même raisonnable (localement géométrique) qui classifie les objets de $T$. Dans \cite{panditth}, ce champ est étudié dans le contexte des $\infty$-catégories stables, ses propriétés de descente et de géométricité. Le champ $\M_T$ admet une structure de monoïde commutatif à homotopie cohérente près donnée par la somme directe des dg-modules. Sa réalisation topologique est alors un monoïde commutatif à homotopie cohérente près dans les espaces, qui est de plus un groupe au sens homotopique. On peut donc par l'équivalence de Segal--Bousfield--Friedlander définir un spectre connectif à partir de celui-ci. 
Notre spectre de K-théorie semi-topologique connective $\kcst(T)$ ne semble pas directement relié au champ $\M_T$ lui-même, mais à une variante que je note $\M^T$. Le champ $\M_T$ classifie les dg-modules pseudo-parfaits et le champs $\M^T$ classifie les dg-modules parfaits. Lorsque $T$ est propre et lisse, ces deux champs sont équivalents. La proposition suivante fait le lien avec les définitions de $\ktop$ données dans \cite[§2.2.6]{kkp} \cite[Déf.8.2]{kal}. 

\begin{prop} \emph{(voir prop.\ref{kstmt})} --- Il existe une équivalence $\kcst(T)\simeq \re{\M^T}$. 
\end{prop}

\subsection{Applications, relations à d'autres travaux}

La première motivation de la définition de $\ktop(T)$ et de $\chtop$ est de donner un sens à la conjecture dite <<conjecture du réseau>>. 

\begin{conj} \emph{(du réseau)} --- Pour tout espace non commutatif $T$ propre et lisse, le morphisme naturel $\chtop\sm_\s H\C : \ktop(T)\sm_{\s} H\C\lmo \hp(T)$ est une équivalence. 
\end{conj}

Sa validité implique que l'image de $\ktop(T)$ donne la partie rationnelle d'une structure de Hodge non-commutative sur $\hp(T)$ pour $T$ propre et lisse. Avec notre résultat de comparaison pour les variétés algébriques lisses (prop. \ref{compvarlisse}) nous vérifions que la conjecture du réseau est vrai pour une telle variété. Dans notre dernier paragraphe, nous montrons que (la variante pseudo-connective de) la conjecture du réseau est vraie pour une algèbre associative de dimension finie sur $\C$. Ce résultat s'obtient en montrant la propriété d'invariance par extension infinitésimale pour le champ des modules projectifs de rang fini sur l'algèbre considérée. De plus ce résultat implique la formule suivante pour les groupes d'homologie périodique d'une algèbre de dimension finie notée $B$. Notons $\vect^B$ le champ des modules projectifs de rang fini sur $B$. On note $\re{\vect^B}^{ST}$ la stabilisation de l'espace $\re{\vect^B}$ par rapport au $B$-module $B$. On a alors la 

\begin{prop} \emph{(voir prop.\ref{formhp1})} --- Le caractère de Chern $\kctop(B)\lmo \hp(B)$ induit un isomorphisme de $\C$-espaces vectoriels pour tout $i=0,1$, 
$$ colim_{k\geq 0} \pi_{i+2k} \re{\vect^B}^{ST}\te_\Z \C\simeq\hp_i(B)$$
où la colimite est induite par l'action de l'élément de Bott $\beta$ sur les groupes d'homotopie, $\pi_i\re{\vect^B}^{ST} \lmos{\times \beta} \pi_{i+2} \re{\vect^B}^{ST}$.
\end{prop} 

Nous pensons que la conjecture est de plus vraie dans les cas suivants : un DM-champ propre et lisse sur $\C$, une dg-catégorie de factorisation matricielle $MF(X,W)$, avec $W:X\lmo \A^1$ une fonction ayant une singularité isolée, avec des hypothèses adéquates sur $W$. On notera que ce dernier exemple nécessite de travailler avec des dg-catégories $2$-périodiques, qui ont leur propre théorie homotopique (voir par exemple \cite{tobimf}), et de donner un sens au foncteur de réalisation dans ce contexte $2$-périodique. 

La définition de $\ktop(T)$ suggère d'autres conjectures sur la K-théorie des espaces non commutatifs propres et lisses. En \ref{conj} nous énonçons plusieurs conjectures sur la K-théorie à coefficients finis, et l'annulation des K-groupes semi-topologiques négatifs. 

Une autre piste d'investigation qui nous paraît intéressante est le comportement du morphisme $\chtop$ vis à vis de la connexion de Gauss--Manin supporté par l'homologie périodique d'une famille d'espaces non commutatifs paramétrée par une variété affine lisse (voir \cite{caract}, \cite{tsygm}). Nous pensons que le morphisme $\chtop$ est plat par rapport à cette connexion. On pourra consulter pour cela \cite[§2.2.5]{kkp} pour la relation avec la variation de structure de Hodge associée. 

Les applications suivantes demeurent plus vagues pour le moment mais il nous semble quand même intéressant de les citer. 

Une application de la définition de $\ktop$ et de son caractère de Chern est donnée par une définition nouvelle de la cohomologie de Deligne pour les espaces non commutatifs. Dans le cas classique, la cohomologie de Deligne paire d'une variété projective lisse est donnée par l'extension des classes entières de degré $(p,p)$ par sa $p$ième jacobienne intermédiaire. Ceci suggère la définition suivante de la cohomologie de Deligne d'un espace non commutatif $T$,
$$\mrm{H_{\mcal{D}}}(T):=\ktop(T)\times_{\hp(T)} \hcn(T),$$
qui est rendue possible par le théorème d'existence du morphisme $\chtop$.  On observe donc l'existence d'un caractère de Chern $\kn(T)\lmo \mrm{H_{\mcal{D}}}(T)$. Ce morphisme de Chern peut être pensé comme un mélange entre des classes caractéristiques usuelles et classes caractéristiques secondaires, et peut constituer un cadre pour l'étude de formules de type Grothendieck--Riemann--Roch pour un espace non commutatif associé à une déformation par quantification d'une structure symplectique et des formules d'indices \emph{à la} Bressler–Gorokhovsky–Nest–Tsygan. On cite aussi \cite{tabmarjac} concernant la théorie des jacobiennes intermédiaires pour les espaces non commutatifs. 

On peut citer comme autre application éventuelle l'utilisation de la K-théorie topologique pour répondre à une question de Freed posée dans \cite{freed}. Freed s'intéresse à une théorie de Chern--Simons qui associe une catégorie linéaire à une $1$-variété fermée, dont la réduction est sensée être donnée par un raffinement de l'homologie de Hochschild défini sur $\Z$. La K-théorie topologique telle que définie dans cette thèse doit répondre à ce problème. 

\paragraph*{Organisation du texte.}

Cette thèse est divisée en cinq sections. La section 0 est cette même introduction. La section 1 est constituée de préliminaires concernant la théorie homotopique des $\del$ et $\gam$-espaces (modèles de monoïdes à homotopie cohérente près), de la K-théorie algébrique des espaces non commutatifs, et du caractère de Chern algébrique. Cette dernière sous-partie sur le caractère algébrique utilise les résultats de Cisinski--Tabuada sur le motivateur $Mot_{\mrm{loc}}$ pour exhiber une structure de $\ka$-module sur la K-théorie algébrique de telle sorte que le morphisme de Chern $\ch$ est un morphisme de $\ka$-modules. La section 2 concerne la réalisation topologique des préfaisceaux simpliciaux et ses nombreuses propriétés. 
La section 3 donne les définitions et les premières propriétés de la K-théorie topologique ainsi que son lien avec le champ $\M^T$, et termine par des conjectures. La section 4 traite le premier exemple d'une variété lisse, que l'on espère généraliser rapidement à toutes les variétés en suivant une suggestion de DC Cisinski. Enfin elle contient aussi l'exemple des algèbres de dimension finie sur $\C$, ainsi que des conséquences dans le cas où l'algèbre est lisse.

\newpage

\subsection{Danksagung}

Mes premiers remerciements vont à mon directeur, Bertrand Toën. D'abord pour m'avoir fait confiance, lorsque j'étais en Master à Toulouse. Pour m'avoir soutenu pendant ces quatre années avec multe patience. Ce fut un réel plaisir de travailler en tant que thésard à ses côtés, dû à son attention et sa riche et vivante vision mathématique. Merci pour m'avoir simplement écouté quand rien n'était plus possible. 

Denis-Charles Cisinski et Boris Tsygan m'ont fait l'honneur de rapporter cette thèse, je leurs en suis très reconnaissant. Leurs rapports me sont très précieux. Merci à Denis-Charles pour la relecture attentive, les nombreuses remarques, corrections et améliorations suggérées. Merci à Boris pour ce moment de partage à Luminy pendant l'été 2012. 

Je remercie chaleureusement Bernhard Keller pour m'avoir soutenu en me connaissant seulement de loin, de m'avoir invité à Paris pour présenter mes travaux et de faire parti de ce jury. 

Merci à Michel Vaquié pour l'intérêt qu'il porte à mon travail, de m'avoir invité à Toulouse et de faire parti de ce jury. Je garde un souvenir très agréable des bains dans la calanque du Sugiton, à Luminy, lorque l'on s'échappait de la conférence à l'heure du déjeuner, en compagnie de Valentina. 

J'ai eu le plaisir de passer une semaine dans un des groupes de travail les plus formidables, le Talbot Workshop, en compagnie de Carlos Simpson. Je le remercie chaleureusement pour sa sympathie, son optimisme. C'est avec regret que j'ai du enlever son nom dans la liste du jury, auquel Carlos devait participer initialement. 

Des voyages, conférences, groupes de travail ont ponctué ces années de thèse d'une manière très motivante. Je voudrais exprimer mes remerciements à tous les mathématiciens et mathématiciennes rencontrés alors. Je remercie Benoit Fresse pour les rendez vous à Lille. Je remercie Gabriele Vezzosi pour son intérêt, les discussions enrichissantes ainsi que de m'avoir permis de présenter ces travaux à Paris. De même je remercie Georges Maltsiniotis pour son soutien. Je voudrais remercier Dimitri Kaledin pour l'intérêt qu'il porte à mon travail et les discussions qu'on a partagé. 

Le Talbot Workshop 2011 était une rencontre formidable, je voudrais remercier Hiro Lee Tanaka et Sheel Ganatra pour l'organisation et leur attention. Je suis très reconnaissant à toute l'équipe de UPenn ; Pranav, Aaron, Dragos, Tobias, David pour leur gentillesse et pour m'intégrer si naturellement parmi eux. 

Je remercie Tony Pantev et Ludmil Katzarkov pour leur attention à ce travail, ainsi que Ludmil pour me permettre de continuer mon travail de recherche à Vienne, dans son équipe. J'en suis à la fois honoré et très heureux. La conférence en juin 2013 à l'ESI était très enrichissante et un réel plaisir. 

Sans avoir rencontré Joseph Tapia je n'écrirais pas de remerciements pour cette thèse. Ces moments passés à suivre les cours de Joseph lorsque j'étais étudiant à Toulouse furent pour moi un réel plaisir. Il a éveillé en moi un certain goût pour les mathématiques et les constructions abstraites, tout en gardant un sens de l'auto-dérision. Je lui suis très reconnaissant de son soutien et de ces heures passées devant le tableau noir.   

Je salue les autres enseignants-chercheurs rencontrés à Toulouse qui m'ont encouragé, ainsi que les collègues et amis toulousains qui furent mes compagnons d'armes lors de nos études à Paul Sab. En particulier merci à Jean, Bertrand, Maxime, Victor NG. Je salue la mémoire d'Anthony Ceresuela pour être en partie responsable de mon goût pour les mathématiques. 

Je remercie la génération de thésard présente à mon arrivée à Montpellier, qui m'a acceuilli à bras ouverts. Ainsi que la génération suivante qui entretient une atmosphère vivante et aéroplane. En particulier je voudrais remercier Jonathan, Mickaël, Christophe, Tuong (Tutu) Huy, Florence pour leur attention pendant ces années. Merci aux chercheurs de Montpellier ; Alain Bruguières, Etienne Mann, Thierry Mignon, pour leur soutien et la participation aux groupes de travail. Merci à Eric pour ces moments sportifs !

Je veux exprimer ma plus grande gratitude à celui qui est mon camarade le plus proche, mon ami, mon collègue, Marco Robalo. Son soutien sans faille est sans égal, je lui en suis infiniment reconnaissant. Je remercie aussi les autres $\R$camarades ; Benjamin, James, pour leur soutien et les discussions enrichissantes. 

Je remercie mes parents et Rémi pour leur soutien total, leur compréhension durant toutes ces années. Je salue toute ma famille pour leur soutien et leur intérêt quand à comprendre en quoi consiste mon travail. Je remercie spécialement Josépha pour son soutien. Enfin un Merci Solaire à Nicolas pour ces moments de simplicité et de joie lors de nos pérégrinations.

\newpage

\newpage

\subsection{Notations, conventions}

\begin{itemize}
\item On choisit de régler les questions ensemblistes en choisissant des univers au sens de Grothendieck \cite{sga4-1}. En admettant l'axiome des univers, on choisit donc trois univers $\U\in\V\in \W$ tels que $\N\in \U$. 

\item En général nous considérons des objets (ensembles, ensembles simpliciaux, espaces topologiques, ...) $\V$-petits. On note $Set$ la catégorie des ensembles $\V$-petits, $SSet$ la catégorie des ensembles simpliciaux $\V$-petits, $Top$ la catégorie des espaces topologiques $\V$-petits, $Sp$ la catégorie des spectres symétriques $\V$-petits.
\item Si $C$ est une catégorie, on note $C^{op}$ la catégorie opposée à $C$. Si $D$ est une autre catégorie, on note $C^D$ la catégorie des foncteurs et des transformations naturelles de $D\lmo C$. 
\item Notre référence pour la notion de catégorie de modèles est le livre de Hovey (\cite{hoveymod}). Si $M$ est une catégorie de modèles, on note $Ho(M)$ la catégorie homotopique de $M$, c'est à dire la localisée de $M$ (au sens $1$-catégorique) par rapport aux équivalences faibles. On appellera simplement \emph{équivalences} d'une catégorie de modèles ses équivalences faibles. Si $M$ est une catégorie de modèles engendrée par cofibrations, et $C$ une catégorie, on utilisera à plusieurs endroits dans ce texte la structure de modèles projective sur la catégorie des foncteurs $M^C$. Son existence découle du théorème général d'existence de \cite[11.6]{hirs}. 
\item La catégorie $SSet$ est munie de sa structure de modèles standard (cf. par exemple \cite[Chap. 3]{hoveymod}). La catégorie $Sp$ des spectres symétriques est munie de la structure de modèles stable standard (cf. \cite{hss}), qui est Quillen équivalente à la structure usuelle des spectres usuels. La catégorie $Top$ est munie de sa structure de modèles standard (cf. par exemple \cite[2.4]{hoveymod}). 
\item A part à la remarque \ref{remcompw}, tous les spectres considérés sont des spectres symétriques. On fera parfois l'abus d'omettre le qualificatif 'symétrique' pour alléger la lecture, bien que nous aurons à faire à des spectres symétriques. 
\item Quand nous aurons à considérer des spectres symétriques en anneaux, par convention il s'agira toujours de spectre en anneau associatif et unitaire au sens strict. Si $A$ est un spectre symétrique en anneaux, l'expression $A$-modules signifie par convention $A$-modules à droite. La notation $A-Mod_\s$ désigne la catégorie des $A$-modules à droite. 
\item Si $C$ est une catégorie  et $x,y\in C$ deux objets de $C$, on note $\Hom_C(x,y)$ l'ensemble de morphismes de $x$ vers $y$ dans $C$. Si $C$ est enrichie dans une catégorie $V$ autre que celle des ensembles, on notera $\homi_C(x,y)$ l'objet des morphismes de $x$ vers $y$, en précisant si besoin la catégorie $V$ dans laquelle $C$ est enrichie. Si $C$ est de plus munie d'une structure de modèles compatible avec l'enrichissement, on notera $\rhomi_C$ son hom interne dérivé. Dans le cas particulier de $V=SSet$, on notera les hom internes $\map_C$ et $\R\map_C$. 
\item Dans une catégorie possédant un objet final on notera celui-ci par $\ast$.
\item On note $\del$ la catégorie simpliciale standard des ordinaux finis $[0], [1], [2], \hdots$ avec pour morphismes les applications croissantes. Si rien n'est précisé la notation $S^1$ réfère au modèle standard $S^1=\del^1/\partial \del^1$ du cercle simplicial, pointé par son zéro simplexe. 
\item Par convention, sauf mention du contraire, le mot schéma signifie schéma de type fini sur la base. Si $k$ un anneau commutatif noethérien, on note $\affk$ la catégorie des schémas affines de type fini sur $k$, on note $\schk$ la catégorie des schémas de type fini sur $k$. 

\item Si $C$ est une catégorie $\V$-petite et $V$ une catégorie localement $\V$-petite, on note $Pr(C, V)$ la catégorie des préfaisceaux sur $C$ à valeurs dans $V$. Dans le cas particulier de $V=SSet$, on note $Pr(C, V)=:SPr(C)$. Dans le cas particulier $V=Sp$, on note $Pr(C, Sp)=:Sp(C)$. 

\item \emph{Dans les cas particuliers où $V=SSet, Sp$ munies de leur structure de modèles standard, la catégorie $Pr(\affk, V)$ est munie par défaut de la structure de modèles projective pour laquelle les équivalences et les fibrations sont définies niveau par niveau (cf. \cite[11.6]{hirs}). \textbf{Sauf si mention du contraire, par les notations $SPr(\affk)$ et $\spafk$ on entend donc les structures de modèles projectives}.}
\item Si $M$ est une catégorie de modèles, et $F:I\lmo M$ un diagramme dans $M$ (avec $I$ une catégorie), on note $hocolim_I F$ la colimite homotopique de $F$, c'est-à-dire le foncteur dérivée gauche du foncteur $colim_I : M^I\lmo M$ pour la structure projective. On note $holim_I F$ la limite homotopique de $F$, c'est le foncteur dérivée droit du foncteur $lim_I : M^I\lmo M$ pour la structure injective. On notera $A\ph_C B$ les produits fibrés homotopiques (ou pullback homotopique) et $A\sh_C B$ les sommes amalgamées homotopiques (ou pushout homotopique). 
\item En présence d'un foncteur de Quillen à gauche $f : M\lmo N$ entre catégorie de modèles, on fera plusieurs fois l'abus de dire que $\lef f : Ho(M)\lmo Ho(N)$ commute aux colimites homotopiques pour signifier que pour toute catégorie $I$, et tout $I$-diagramme $F:I\lmo M$, le morphisme $hocolim_I \lef f(F)\lmo \lef f(hocolim_I F)$ est un isomorphisme dans $Ho(N)$. On pratiquera l'abus analogue pour les limites homotopiques. 
\item Dans ce texte, à part si il est précisé le contraire, les jeux de foncteurs adjoints écris horizontalement seront par convention écris de telle sorte que tout foncteur est adjoint à gauche du foncteur au dessous de lui.

\end{itemize}

 \newpage
 
\section{Préliminaires sur les monoïdes et la K-théorie}

\subsection{Monoïdes associatifs et commutatifs à homotopie près}\label{monoides}

On utilise dans ce travail des modèles particuliers pour la notion de monoïdes associatifs et commutatifs connus sous le nom de $\del$-espaces et de $\gam$-espaces respectivement. Dans cette section on fixe les notations les concernant et on rappelle les résultats essentiels, notamment l'équivalence entre la théorie homotopique des $\gam$-espaces très spéciaux (ie "group-like") et la théorie homotopique des spectres connectifs.

$\gam$ désigne le squelette de la catégorie des ensembles finis pointés (avec morphismes pointés) qui consiste en les objets $\n=\{0,\hdots, n\}$ avec $0$ pour point base. 

\begin{df}\label{defdel} Soit $M$ une catégorie de modèles. 
\begin{itemize}
\item Un $\del$-objet (resp. un $\gam$-objet) dans $M$ est un foncteur $\del^{op}\lmo M$ qui envoie $[0]$ sur $\ast$ (resp. un foncteur $\gam\lmo M$ qui envoie $\und{0}$ to $\ast$). Un morphisme de $\del$-objet (resp. de $\gam$-objet) dans $M$ est simplement un morphisme de foncteurs. On note $\del-M$ (resp. $\gam-M$) la catégorie des $\del$-objets (resp. des $\gam$-objets) dans $M$. Pour $E\in \del-M$ (resp. $F\in\gam-M$), on fixe les notations $E([n])=E_n$ et $F(\n)=F_n$. 

\item On dit qu'un $\del$-objet $E$ dans $M$ est \emph{spécial} si tous les morphismes de Segal de $E$ sont des équivalences, ie que pour tout $[n]\in\del$ le morphisme
$$ p_0^\ast \times \hdots \times p_{n-1}^\ast : E_n\lmo E_1^{\ph n} = E_1\ph \hdots \ph  E_1$$ 
est une équivalence dans $M$, avec $p_i : [1]\lmo [n]$, $p_i(0)=i$ et $p_i(1)=i+1$.

\item On dit qu'un $\gam$-objet $F$ dans $M$ est \emph{spécial} si pour tout $\n\in \gam$ le morphisme
$$ q^1_\ast \times \hdots \times q^n_\ast : F_n\lmo F_1^{\ph n}$$ 
est une équivalence dans $M$, avec $q^i : \n\lmo \und{1}$, $q^i(j)=1$ si $j=i$ et $q^i(j)=0$ si $j\neq i$.

\item Si $E$ est un $\del$-objet spécial dans $M$, on dit que $E$ est \emph{très spécial} si le morphisme
$$p_0^\ast \times d_1^\ast : E_2\lmo E_1\ph E_1$$ 
est une équivalence dans $M$, avec $d_1 : [1] \lmo [2]$ la face qui évite $1$ dans $[2]$.

\item Si $F$ est un $\gam$-objet spécial dans $M$, on dit que $F$ est \emph{très spécial} si le morphisme
$$q^1_\ast \times \mu_\ast : F_2\lmo F_1\ph F_1$$
est une équivalence dans $M$, où $\mu : \und{2}\lmo \und{1}$ est le morphisme défini par $\mu(1)=1$ et $\mu(2)=1$. 
\end{itemize}
\end{df}

\begin{rema} 

\begin{itemize}
\item Pour $M=SSet$, les $\del$-objets et les $\gam$-objets sont appelés $\del$-espaces et $\gam$-espaces (e.g. dans \cite{bf}). 

\item Les $\del$-objets sont des modèles particuliers pour la notion de monoïdes associatifs unitaires "à homotopie cohérente près" ou $A_\infty$-monoïdes, cette idée étant apparue pour la première fois dans l'article de Segal \cite{seg}. Un $\del$-objet spécial $E$ dans la catégorie $Set$ des ensembles $\U$-petits (avec la structure de modèles triviale) est la donnée d'ensembles $(E_n)_{n\geq 1}$ tel que les morphismes de Segal $p_0^* \times \hdots \times p_{n-1}^* : E_n\lmo E_1^n$ sont des isomorphismes pour tout $n\geq 0$. En particulier pour $n=2$ on obtient une loi de composition 
$$\xymatrix{E_2 \ar[d]_-{d_1^*} \ar[rr]_-{p_0^*\times p_1^*}  && E_1\times E_1 \ar@/_1pc/[ll]_-\sim \ar@/^1pc/[lld] \\ E_1 }$$
sur l'ensemble $E_1$. Les isomorphismes de Segal supérieurs ($n=3$ suffit) se traduisent par l'associavité de cette loi de composition et le morphisme de dégénérescence $s_0:E_0=\ast\lmo E_1$ donne l'unité de cette loi. La catégorie des $\del$-ensembles spéciaux est alors équivalente à la catégorie des monoïdes associatifs unitaires dans les ensembles via le foncteur évaluation en $1$. De même la catégorie des $\del$-objets très spéciaux est équivalente à la catégorie des groupes dans les ensembles. La définition de $\del$-objets consiste alors à remplacer les ensembles par des types d'homotopie voire des objets dans une catégorie de modèles fixée et la notion d'isomorphismes par la notion d'équivalences faibles dans cette catégorie de modèles. Si $E$ est un $\del$-objet spécial dans une catégorie de modèles, l'équivalence de Segal $p_0^*\times p_1^*$ n'a pas d'inverse dans $M$ mais seulement dans $Ho(M)$. De manière heuristique, la condition d'être spécial nous donne une "loi de composition à homotopie cohérente près"
$$\xymatrix{E_2 \ar[d]_-{d_1^*} \ar[rr]_-{p_0^*\times p_1^*}  && E_1\times E_1 \ar@{.>}@/_1pc/[ll] \ar@{.>}@/^1pc/[lld] \\ E_1 }$$
sur l'objet $E_1$. Les $\gam$-objets modélisent quand à eux les monoïdes commutatifs associatifs unitaires "à homotopie cohérente près" ou $E_\infty$-monoïdes. La commutativité est alors codée par le morphisme pointé $\und{2}\lmo \und{2}$ qui échange $1$ et $2$. 

\item Dans le cas d'une catégorie de modèles de la forme "préfaisceaux simpliciaux sur une catégorie" avec la structure globale, on peut remplacer les produits homotopiques  par des produits dans la définition des $\del$ et $\gam$-objets. Dans ce cas il existe une caractérisation simple des objets très spéciaux comme le montre le lemme suivant. 
\end{itemize}
\end{rema}

\begin{lem}\label{ek}
Soit $C$ une catégorie et $M=SPr(C)$ la catégorie de modèles globale des préfaisceaux simpliciaux sur $C$. Soit $E$ un $\del$-objet spécial dans $M$. Alors $E$ est très spécial si et seulement si le monoïde $\pi_0 E_1$ est un groupe. De même pour les $\gam$-objets. 
\end{lem} 

\begin{proof}
Supposons $E$ très spécial, alors le morphisme 
$$ p_0^\ast \times d_1^\ast : \pi_0 E_2\lmo \pi_0 E_1\times \pi_0 E_1$$ 
est un isomorphisme. Mais si on identifie $\pi_0E_2$ avec $\pi_0E_1\times \pi_0E_1$ via la condition d'être spécial, alors ce morphisme envoie $(a,b)$ sur $(a, ab)$. On en déduit que $\pi_0E_1$ est un groupe. 

Réciproquement, si $\pi_0E_1$ est un groupe, on doit montrer que le morphisme
$$ p_0^\ast \times d_1^\ast :  E_2\lmo E_1\times E_1$$ 
est une équivalence dans $M$. Premièrement ce morphisme induit un isomorphisme sur les $\pi_0$ car $\pi_0E_1$ est un groupe. Montrons ensuite qu'il induit un isomorphisme sur les $\pi_i(-, \ast)$ pour tout $i\geq 1$ et tout point base $\ast$. On remarque que pour tout $i\geq 1$, le préfaisceau de groupes $\pi_i(E_1, \ast)$ est muni d'une loi de monoïdes induite par la composée 
$$\xymatrix{\pi_i(E_1, \ast) \times \pi_i(E_1, \ast) \ar[r]^-\sim &\pi_i(E_1\times E_1 , \ast\times \ast) \ar[r]^-\sim &\pi_i(E_2, \ast) \ar[r]^-{\pi_i(d_1^\ast, \ast)} &\pi_i(E_1, \ast)}.$$
Le préfaisceau de groupes $\pi_i(E_1, \ast)$ est donc un monoïde dans la catégorie monoïdale (pour le produit cartésien) des préfaisceaux de groupes sur $C$, donc les deux lois coïncident, le monoïde $\pi_i(E_1, \ast)$ est un groupe et le morphisme $\pi_i(p_0^\ast \times d_1^\ast, \ast)$ est un isomorphisme pour tout $i\geq 1$. 
\end{proof}

\subsubsection{Structures de modèles et complétion en groupe}\label{strucdel}

On rappelle qu'on a au moins trois structures de modèles intéressantes sur $\del-M$ pour toute catégorie de modèles $M$ qui est combinatoire et propre à gauche (hypothèses nécessaires pour effectuer des localisations de Bousfield) : 

\begin{itemize}
\item La structure \emph{projective}, appelée stricte dans \cite{bf}, pour laquelle les équivalences et les fibrations sont définies niveau par niveau. On note simplement $\del-M$ cette catégorie de modèles. 

\item La structure \emph{spéciale}. C'est par définition la localisation de Bousfield gauche de la catégorie de modèles $\del-M$ par rapport à l'ensemble des morphismes
$$(\amalg_{i=0}^{n-1} h_{p_i}:h_{[1]}\amalg \cdots \amalg h_{[1]}\lmo h_{[n]})_{n\geq 1}\quad  \square \quad (\textrm{Cofibrations génératrices de $M$}),$$
où $\square$ désigne le produit boîte (voir \cite[Thm.3.3.2]{hoveymod}). On note $\del-M^{sp}$ cette catégorie de modèles. Par définition, les objets fibrants de $\del-M^{sp}$ sont les $\del$-objets spéciaux dans $M$ qui sont fibrant niveau par niveau. 

\item La structure \emph{très spéciale}. C'est par définition la localisation de Bousfield à gauche de la catégorie de modèles $\del-M^{sp}$ par rapport à l'ensemble de morphismes 
$$(h_{p_0} \amalg h_{d_1}  : h_{[1]} \amalg h_{[1]} \lmo h_{[2]}) \quad \square \quad (\textrm{Cofibrations génératrices de $M$}).$$
On note $\del-M^{tsp}$ cette catégorie de modèles. Par définition, les objets fibrants de $\del-M^{tsp}$ sont les $\del$-objets très spéciaux dans $M$ qui sont fibrant niveau par niveau. 
\end{itemize}

On a les foncteurs dérivés des foncteurs identités
$$\xymatrix{Ho(\del-M) \ar@<2pt>[r]^{\lef id} &  Ho(\del-M^{sp}) \ar@<2pt>[r]^-{\lef id'} \ar@<2pt>[l]^{\R id} & Ho(\del-M^{vsp})\ar@<2pt>[l]^{\R id'} }.$$
On note $mon:=\R id \lef id$ le foncteur monoïde homotopique libre et $(-)^+:=\R id'\lef id'$ la complétion en groupe homotopique.

\begin{rema}
La catégorie des $\gam$-objets dans une catégorie de modèles combinatoire propre à gauche admet des structures de modèles similaires à celles des $\del$-objets. On a donc des foncteurs monoïdes commutatifs homotopique libre et complétion en groupes commutatifs homotopique, 
$$\xymatrix{Ho(\gam-M)\ar[r]^-{com} & Ho(\gam-M^{sp}) \ar[r]^-{(-)^+} & Ho(\gam-M^{tsp})}.$$
Le foncteur de complétion en groupe commutatif a, par abus, la même notation que la complétion en groupe non nécessairement commutatif car on peut montrer qu'ils sont équivalents. Cependant nous n'aurons pas besoin de ce fait ici. 
\end{rema}

\begin{rema}\label{remseg}
On pose $M=SPr(C)$. Si $C=\ast$, on retrouve l'exemple familier des $\del$-espaces. Le théorème de Segal \cite[Prop 1.5]{seg} implique que la complétion en groupe $(-)^+$ a comme modèle la composée des foncteurs suivants,
$$\xymatrix{Ho(\del-M^{sp}) \ar[r]^-{\re{-} } & Ho(M_\ast) \ar[r]^-{\Omega_\bul} & Ho(\del-M^{vsp})  },$$
où $\re{-}$ est la réalisation des objets bisimpliciaux et pour un préfaisceau simplicial pointé fibrant niveau par niveau $(E, \ast)$, le niveau $n$ du $\del$-objet $\Omega_\bul E$, noté $\Omega_n E$ est défini par le carré cartésien
$$\xymatrix{ \Omega_n E \ar[r] \ar[d] & \homi_{SPr(C)} (\del^n, E) \ar[d] \\ \ast \ar[r] & E^{n+1}}$$
où le morphisme de droite est la restriction aux $n+1$ sommets de $\del^n$ et le morphisme en bas est donné par le point base de $E$. L'objet $\Omega_n E$ est donc le préfaisceau simplicial des morphismes $\del^n\lmo E$ qui envoie tous les sommets de $\del^n$ sur le point base de $E$. Le théorème de Segal montre en fait plus : le foncteur composé $(-)^+\circ mon$ est isomorphe au foncteur composé $\Omega_\bul \circ \re{-}$. 
\end{rema}

\begin{ex}\label{exespk}
Si $C$ est une catégorie de Waldhausen, on a un $\del$-espace, 
$$\K_\bul (C):=NwS_\bul C,$$
où $Nw$ désigne le nerf de la sous-catégorie des équivalences et $S_\bul$ est la $S$-construction de Waldhausen. Le niveau $1$ de ce $\del$-espace est équivalent à l'ensemble simplicial $NwC$ et ce $\del$-espace n'est en général pas spécial. On peut voir la K-théorie algébrique définie par Waldhausen (dans l'article fondateur \cite{wald}) comme un moyen canonique de transformer ce $\del$-espace en un $\del$-espace très spécial. L'espace de K-théorie algébrique de $C$ est par définition l'ensemble simplicial pointé
$$K(C):=(mon\K_\bul(C))^+_1\simeq \Omega\re{NwS_\bul C},$$
où $\Omega$ est l'ensemble simplicial des lacets pointés en l'objet nul, la dernière équivalence vient du théorème de Segal par la remarque (\ref{remseg}). Cette construction définit un foncteur, 
$$K : WCat\lmo SSet_{\ast},$$
où $WCat$ est la catégorie des grosses catégories de Waldhausen, $SSet_{\ast}$ celle des gros ensembles simpliciaux pointés. 

De plus on a la description standard du $K_0$ suivante. On vérifie que le foncteur $\pi_0$ commute au foncteur monoïde libre et complétion en groupe. C'est-à-dire que les carrés de catégories
$$\xymatrix{Ho(\del-SSet)\ar[r]^-{mon} \ar[d]^-{\pi_0} & Ho(\del-SSet) \ar[d]^-{\pi_0} \\ \del-Set \ar[r]^-{mon} & \del-Set }
\xymatrix{ Ho(\del-SSet^{sp})\ar[r]^-{(-)^+} \ar[d]^-{\pi_0} & Ho(\del-SSet^{sp}) \ar[d]^-{\pi_0} \\ \del-Set^{sp} \ar[r]^-{(-)^+} & \del-Set^{sp}  }$$
sont tous les deux commutatifs. Ceci se vérifie directement en utilisant Yoneda et le fait que $mon$, $(-)^+$, sont des adjoints à gauche. 

On en déduit un isomorphisme canonique d'ensembles
$$\pi_0 (mon \, \K_\bul (C))^+_1\simeq (mon \, \pi_0 \K_\bul (C))^+_1.$$
Le monoïde libre $mon \, \pi_0 \K_\bul (C)$ du $\del$-ensemble $\pi_0 \K_\bul (C)$ est le monoïde dans lequel on identifie $a\in \pi_0 NwC$ avec le produit de $a'$ et $a''$ chaque fois qu'on a une suite de cofibrations $a'\moi a \twoheadrightarrow a''$ dans $C$. Cette loi de composition est donc commutative et s'identifie avec la somme dans $C$. La complétion en groupe  $(mon \, \pi_0 \K_\bul (C))^+_1$ est donc le groupe abélien libre sur les classes d'équivalences d'objets de $C$ modulo la relation d'équivalence qui identifie $a$ avec la somme $a'\oplus a''$ chaque fois qu'on a une suite de cofibrations $a'\moi a \twoheadrightarrow a''$. On a donc montré que le groupe abélien $\pi_0K(C)$ est bien le groupe de Grothendieck $K_0(C)$. 
\end{ex}

On a un plongement des $\gam$-objets dans les $\del$-objets donné par le dual du foncteur 
$$\alpha : \del^{op} \lmo \gam,$$
définie sur les objets par $\alpha ([n])=\und{n}$. Et pour tout morphisme $f:[n]\lmo [m]$ dans $\del$, on définit $\alpha(f)=g : \und{m}\lmo \und{n}$ par 
$$g(i)=\left\lbrace
\begin{array}{ccc}
0 & \text{if  }& 0\leq i\leq f(0) \\ 
j & \text{if  }& f(j-1)< i\leq f(j) \\ 
0  & \text{if  }& f(n)<i \\ 
\end{array} 
\right.$$
On peut vérifier que $\alpha (p_i)=q^{i+1}$ pour $i=0,\hdots, n-1$, et que $\alpha (d_1)=\mu$. On en déduit que le foncteur pleinement fidèle
$$\alpha^\ast : \gam-M \lmo \del-M, $$
envoie les $\gam$-objets spéciaux (resp. très spéciaux) sur des $\del$-objets spéciaux (resp. très spéciaux) et on a donc un diagramme

$$\xymatrix{ Ho(\gam-M)   \ar[d]_-{\alpha^\ast} \ar[r]^-{com} &  Ho(\gam-M^{sp})  \ar[d]_-{\alpha^\ast} \ar[r]^-{(-)^+} & Ho(\gam-M^{vsp} )  \ar[d]_{\alpha^\ast} \\
 Ho(\del-M)  \ar[r]^-{mon} &  Ho(\del-M^{sp})  \ar[r]^-{(-)^+} &  Ho(\del-M^{vsp} )  }$$
Le carré de gauche n'est pas commutatif, mais on peut montrer que le carré de droite est commutatif à isomorphisme près comme nous l'avons dit plus haut, mais nous n'utiliserons pas ce fait.

\subsubsection{$\gam$-espaces très spéciaux et spectres connectifs}\label{groupsp}

On rappelle l'équivalence entre la théorie homotopique des $\gam$-espaces très spéciaux et celle des spectres connectifs. Cette équivalence apparaît d'abord dans le papier fondamental de Segal \cite{seg} et a été ensuite formalisé dans le langage des catégories de modèles de Quillen dans \cite{bf}. Le Théorème $5.8$ de \cite{bf} se généralise immédiatement aux préfaisceaux de $\gam$-objets et aux préfaisceaux en spectres sur une catégorie $C$. De plus, en suivant Schwede \cite[Exemple 2.39]{schw-sym}, on peut remplacer les spectres ordinaires par les spectres symétriques qui présente l'avantage de l'existence d'une bonne structure monoïdale. On a une paire de foncteurs adjoints 
$$\xymatrix{ \gam-SPr(C) \ar@<2pt>[r]^-{\mcal{B}} & \ar@<2pt>[l]^-{\mcal{A}} Sp(C) },$$
où $Sp(C)$ est la catégorie des préfaisceaux de spectres \emph{connectifs} sur $C$. La paire de foncteurs $(\mcal{B}, \mcal{A})$ peut être (grossièrement) définie de la manière suivante. Tout préfaisceau de $\gam$-espaces $E$ peut être étendu en un foncteur
$$E:Sp(C)\lmo Sp(C)$$ 
à travers une succession d'extension de Kan à gauche le long (de la version préfaisceautique) des inclusions $\gam\hookrightarrow Set \hookrightarrow SSet\hookrightarrow Sp$. Le foncteur $\mcal{B}$ peut alors être définie par 
$$\mcal{B}(E)=E(\s),$$
la valeur du foncteur $E$ en le spectre symétrique en sphère, qui est un spectre symétrique connectif pour chaque $\gam$-objet $E$. Le foncteur $\mcal{B}$ ainsi défini est identique au foncteur $B$ de Segal qui va des $\gam$-espaces spéciaux dans les spectres connectifs. De plus il préserve les équivalences entre tous $\gam$-objets, pas juste les cofibrants. De manière plus explicite, en suivant \cite[4.]{perm}, le foncteur $\mcal{B}$ peut se décrire de la manière suivante (dans le cas des $\gam$-espaces). Il existe un modèle du cercle simplicial $S^1$ vérifiant $S^1_n=\n$ (l'objet de $\gam$ associé à l'entier $n$), avec les faces et dégénérescences définies de manière adéquate. Il suffit en fait de prendre $S^1\simeq\del^1/\partial \del^1$. Le cercle $S^1$ est considéré comme pointé par son unique $0$-simplexe. Les termes du spectre en sphère $\s$ sont définis par $\s_n=S^n=(S^1)^{\sm n}$ sur lequel le groupe symétrique $\Sigma_n$ agit par permutation des facteurs. Pour tout $n\geq 0$ et tout $k\geq 0$, on a $S^n_k=S^1_k\sm\hdots\sm S^1_k=\und{k}\sm \hdots \sm \und{k}$, $n$ fois, où $\sm$ désigne le smash entre ensembles pointés. Le $n$-ième terme du spectre $\mcal{B}(E)$ est alors donné par 
$$\mcal{B}(E)_n=E(S^n)=d(k\mapsto E(S^n_k))$$ 
où $d:SSet^{\del^{op}} \lmo SSet$ est la réalisation des ensembles bisimpliciaux (par exemple la diagonale). Pour tout $n\geq 0$ et tout $k\geq 0$, le choix d'un $k$-simplexe $x\in S^1_k$ définit un morphisme 
$$S^{n-1}_k\lmo S^n_k$$
par $(i_1, \hdots, i_{k-1})\mapsto (x,i_1, \hdots, i_{k-1})$. En appliquant $E$ on a un morphisme $S^1_k\sm E(S^{n-1}_k) \lmo E(S^n_k)$, et en appliquant la réalisation on obtient de cette manière une collection de morphismes
$$S^1\sm E(S^{n-1})\lmo E(S^n)$$
qui sont les morphismes de transitions de $\mcal{B}(E)$.

Le foncteur $\mcal{A}$ est défini sur un objet $F\in Sp(C)$ comme étant le $\gam$-objet 
$$\n\mapsto\mcal{A}(F)_n=\map (\s^n, F),$$
où $\map$ désigne le préfaisceau simplicial des morphismes entre deux préfaisceaux en spectres symétriques. On remarque que le niveau $1$ de ce $\gam$-objet est donné par 
$$\mcal{A}(F)_1=\map(\s,F)\simeq \map (S^0,F_0) \simeq F_0$$
le niveau $0$ du préfaisceau en spectre $F$. 

\begin{theo}\label{bf}\emph{(\cite[Thm.5.8]{bf})} La paire de foncteurs adjoints $(\mcal{B},\mcal{A})$ est une adjonction de Quillen pour la structure de modèles très spéciale sur $\gam-SPr(C)$. De plus, cette adjonction induit une équivalence de catégories
$$\xymatrix{ Ho(\gam-SPr(C)^{tsp} ) \ar@<2pt>[r]^-{\lef \mcal{B}}  & Ho( \spcon(C)) \ar@<2pt>[l]^-{\R \mcal{A} } }.$$
entre la catégorie homotopique des $\gam$-objets très spéciaux et celle des préfaisceaux en spectres symétriques \textbf{connectifs}. La notation $Ho( \spcon(C))$ désignant la sous-catégorie pleine de $Ho(Sp(C))$ formée des préfaisceaux en spectres connectifs. 
\end{theo}

\subsection{Délaçage connectif canonique de l'espace de K-théorie}\label{delass}

On note $WCat$ la catégorie des catégories de Waldhausen $\U$-petites et des foncteurs exacts. Si $M$ est une catégorie et $k\geq 1$ un entier, on note $\del^k-M$ la catégorie des objets $k$-multisimpliciaux dans $M$, c'est-à-dire des foncteurs $(\del^{op})^k\lmo M$. Pour tout $k\geq 2$, on note 
$$B^{(k)} : \del-M\lmo \del^k-M$$
le foncteur défini par $(B^{(k)}E)_{n_1, \hdots, n_k}=E_{n_1\hdots n_k}$ où $n_1\hdots n_k$ est le produit des nombres $n_1,\hdots,n_k$. On pose $B^{(1)}=id$. 

On utilise les notations et les sorites de l'exemple \ref{exespk}. La $S$-construction de Waldhausen donne donc pour toute catégorie de Waldhausen $C$ un $\del$-espace
$$\K(C)=NwS_\bul C.$$
L'espace de K-théorie de $C$ est par définition l'ensemble simplicial pointé $K(C):=(mon \K_\bul (C))^+_1$. On a donc un foncteur "espace de K-théorie" 
$$K: WCat\lmo SSet_\ast.$$

\paragraph{Délassage donné par la somme.}\label{gamstru}

Toute catégorie de Waldhausen admet des sommes finies. En suivant Segal \cite{seg} (et Elmendorf-Mandell \cite{perm}) il est possible de construire à  partir d'une catégorie de Waldhausen $C$ un $\gam$-objet spécial dans $WCat$ qui encode cette loi de monoïde donné par la somme de $C$. Cette construction définit même un foncteur 
$$B_W : WCat \lmo \gam-WCat.$$
Si $C\in WCat$, on pose $B_WC_0=\ast$ et si $n\geq 1$, la catégorie $B_WC_n$ est la catégorie dont
\begin{itemize}
\item les objets sont les familles $\{a_S, \rho_{S,T} \}$ où $S,T$ parcourt les sous-ensembles de $n^+$ qui ne contiennent pas $0$ et $S\cap T=\varnothing$. Les $a_S$ sont des objets de $C$  et les $\rho_{S,T}$ des isomorphismes dans $C$ 
$$\rho_{S,T} : a_S\oplus a_T \simeq a_{S\cup T}$$
tel que $a_S=0$ et $\rho_{S,T}=id_{a_T}$ dès que $S=\varnothing$ et tel que les diagrammes suivants commutent pour tout triplet $S,T,U$ de sous-ensembles de $n^+$ deux à  deux disjoints : 
$$\xymatrix{ a_S\oplus a_T \ar[r]^-{\rho_{S,T} }  \ar[d]^-{\tau} & a_{S\cup T} \ar[d]^-{id} \\ a_T\oplus a_S \ar[r]^-{\rho_{T,S} } & a_{T\cup S} }$$

$$\xymatrix{ a_S\oplus a_T\oplus a_U \ar[r]^-{\rho_{S,T} \oplus id } \ar[d]^-{id\oplus \rho_{T,U} } & a_{S\cup T} \oplus a_U \ar[d]^-{\rho_{S\cup T, U} } \\ a_S\oplus a_{T\cup U} \ar[r]^-{\rho_{S,T\cup U} } & a_{S\cup T\cup U} } $$

\item un morphisme $f:\{a_S, \rho_{S,T} \} \lmo \{a'_S, \rho'_{S,T} \} $ est la donnée de morphismes $f_S:a_S\lmo a'_S$ tel que $f_\varnothing=id_0$ et qui commutent avec les isomorphismes $\rho_{S,T} $ et $\rho'_{S,T}$. 
\end{itemize}

On a $B_W C_0=\ast$ et $B_W C_1=C$. Pour tout $n^+\in \gam$ la catégorie $B_WC_n$ est naturellement munie d'une structure de catégorie de Waldhausen héritée de celle de $C$ en disant qu'un morphisme $f=\{f_S\}$ est une cofibration (resp. une équivalence) si chaque $f_S$ est une cofibration (resp. une équivalence). Etant donné que le foncteur d'espace de K-théorie commute aux produits finis, en prenant l'espace de K-théorie niveau par niveau on obtient un $\gam$-espace spécial :
$$K^\gam(C):=K(B_WC),$$
dont le niveau $1$ est l'espace de K-théorie $K(C)$. En fait ce $\gam$-espace est très spécial. En effet 
$$\pi_0(K^\gam(C)_1)\simeq \pi_0 K(C)=\pi_0 \Omega\re{NwS_\bul C}\simeq \pi_0 (mon \, NwS_\bul C)^+\simeq (mon \, \pi_0 NwS_\bul C)^+\simeq K_0(C).$$

\begin{df}\label{spk}
On définit le \textit{spectre de K-théorie connective de $C$} par :
$$\kc (C) :=\mcal{B} K^\gam(C),$$
où $\mcal{B}$ est le foncteur $\mcal{B} : Ho(\gam-SSet^{tsp}) \lmo Ho(\spcon)$ défini en \ref{groupsp}. 
\end{df} 

On résume la construction par le diagramme

$$\xymatrix{WCat\ar[r]^-{B_W} & \gam-WCat\ar[r]^-{K}  & \gam-SSet^{tsp} \ar[r]^-{\mcal{B}} & \spcon }.$$
On obtient donc un foncteur
$$\kc : WCat\lmo \spcon.$$

\begin{rema}\label{remcompw}
Le spectre de K-théorie connective de la définition \ref{spk} peut se comparer au spectre de K-théorie connective tel que défini par Waldhausen dans \cite{wald}. Cette comparaison n'est pas fondamentale dans le sens où toutes les preuves des théorèmes fondamentaux de Waldhausen concernant la K-théorie connective reste valable pour le spectre de la définition \ref{spk}. Cependant ces deux spectres sont isomorphes dans la catégorie homotopique stable des spectres (on ne compare pas les structures de spectres symétriques), et nous esquissons dans cette remarque une preuve. 

Le spectre de K-théorie de Waldausen est le spectre ayant pour $k$-ième termes $\re{NwS_\bul^{(k)} C}$, où $S_\bul^{(k)}$ désigne la $S$-construction itérée $k$ fois et $\re{-}$ désigne la réalisation des ensembles multisimpliciaux (par exemple la diagonale). Pour tout $n_1, \hdots, n_{k-1}\geq 0$, on a un foncteur exact 
\begin{equation}\label{fonctexs}
S_{n_{k-1}} S_{n_{k-2}} \hdots S_{n_1} C\lmo S_\bul S_{n_{k-1}} \hdots S_{n_1} C
\end{equation}
obtenu en considérant un objet de la première catégorie comme une suite de cofibrations de longueur $1$ dans la direction simpliciale (on considère la première catégorie comme constante dans la direction simpliciale). Ce foncteur exact induit un morphisme sur les réalisations 
$$S^1\sm \re{NwS_\bul^{(k)} C}\lmo \re{NwS_\bul^{(k+1)} C}$$
qui sont les morphismes de transitions du spectre de Waldhausen. D'après \cite[rem. suivant le cor.1.3.5]{wald}, les morphismes
$$\re{NwS_\bul^{(k)} C}\lmo \Omega\re{NwS_\bul^{(k+1)} C}$$
sont des équivalences pour tout $k\geq 1$. Son spectre est donc équivalent au spectre $\kc'(C)$ défini par 
$$\kc'(C)_n=\Omega\re{NwS_\bul^{(k+1)}} C\simeq \re{K(S_\bul^{(k)} C)}$$
pour tout $k\geq 0$ et avec les morphismes de transition induits par ceux du spectre précédent, c'est-à-dire induits par les foncteurs (\ref{fonctexs}). Il suffit donc de comparer les spectres $\kc(C)$ et $\kc'(C)$. On esquisse une construction d'un morphisme de spectre $\kc(C)\lmo \kc'(C)$. Au niveau $n$ il est donné par un morphisme 
$$\mcal{B}(K^\gam(C))_k\lmo \re{K(S^{sc,(k)}_\bul C)} \lmo \re{K(S^{(k)}_\bul C)}$$
où l'espace du milieu est un espace intermédiaire construit à partir de la $S$-construction scindée. 

On fixe quelques notations. Si $M$ est une catégorie et $k\geq 1$ un entier, on note $\del^k-M$ la catégorie des objets $k$-multisimpliciaux dans $M$, c'est-à-dire des foncteurs $(\del^{op})^n\lmo M$. Pour tout $k\geq 2$, on note 
$$B^{(k)} : \del-M\lmo \del^k-M$$
le foncteur défini par $(B^{(k)}E)_{k_1, \hdots, n_k}=E_{n_1\hdots n_k}$ où $n_1\hdots n_k$ est le produit des nombres $n_1,\hdots,n_k$. On pose $B^{(1)}=id$. Le même foncteur est défini pour les $\gam$-objets et on rappelle que pour tout entier positif $n_1$ et $n_2$, il existe un isomorphisme dans la catégorie $\gam$, 
\begin{equation}\label{prod}
\und{n_1n_2}\simeq \und{n_1}\sm \und{n_2}
\end{equation}
où $\sm$ est le smash des ensembles pointés. 

Il existe une version "scindée" de la $S$-construction où les cofibrations générales sont remplacées par les cofibrations scindées, c'est à  dire les suite de cofibrations de la forme 
$$a_1\moi a_1\oplus a_2\moi \cdots \moi a_1\oplus \cdots \oplus a_n.$$
On note $S^{sc}_\bul  : WCat \lmo \del-WCat$ le foncteur de $S$-construction scindée. Pour tout $[n]\in \del$, la catégorie $S^{sc}_n C$ a pour objets les suites de cofibrations scindée de longueur $n$ plus les choix de quotients habituels et pour morphismes les transformations naturelles de telles suites. Pour tout $[n]\in \del$, la catégorie $S^{sc}_n C$ a une structure de catégorie de Waldhausen en disant qu'un morphisme est une cofibration s'il est constitué de cofibrations dans $C$. La catégorie simpliciale $S^{sc}_\bul C$ a la propriété que pour tout entier positif $n_1$ et $n_2$ on a un morphisme 
$$S^{sc}_{n_2} S^{sc}_{n_1}C\lmo S^{sc}_{n_1 n_2}C$$
qui est une équivalence de catégories. En effet, la donnée d'une suite de cofibrations scindée est équivalente à  la donnée des quotients successifs. Ce morphisme envoie donc une suite de cofibrations scindée de longueur $n_2$ de suites de cofibrations scindées de longueur $n_1$ sur le $n_1n_2$-uplet d'objets qui caractérise ces suites. On en déduit que pour tout $n\geq 1$ on a une équivalence niveau par niveau dans $\del^n-WCat$,
\begin{align}
\xymatrix{ B^{(n)} S^{sc}_\bul C \ar[r]^\sim & S^{sc, (n)}_\bul C. }
\end{align}
On a un foncteur exact $\phi : B_W C \lmo S^{sc}_\bul C$ donné par
\begin{align*}
\phi :(B_WC)_n & \lmo S^{sc}_n C \\
(a_1, \hdots, a_n) &\longmapsto (a_1\moi a_1\oplus a_2\moi \cdots \moi a_1\oplus \cdots \oplus a_n).
\end{align*}
On vérifie que ce morphisme est bien simplicial en $n$. On le vérifie pour les dimensions $1$ et $2$ (la vérification en dimension quelconque est semblable) pour lesquelles on écrit les faces et les dégénérescences.  
\begin{center}
\begin{tabular}{c|c}
$B_W C$ & $S_\bul^{sc} C$\\ \hline
$d_0(a,b)=b$ & $d_0(x\moi y)=y/x$ \\
$d_1(a,b)=a\oplus b$ & $d_1(x\moi y)=y$\\
$d_2(a,b)=a$ & $d_2(x\moi y)=x$\\
$s_0(a)=(0,a)$ & $s_0(a)=(0\moi a)$ \\
$s_1(a)=(a,0)$ & $s_1(a)=id_a$
\end{tabular}
\end{center}
On vérifie alors que
\begin{align*}
 & \phi(d_0(a,b)) = \phi (b)=b=d_0(a\moi a\oplus b)=d_0(\phi (a,b)), \\
 & \phi(d_1(a,b)) = \phi(a\oplus b)=a\oplus b=d_1(a\moi a\oplus b)=d_1(\phi(a,b)),\\
 & \phi(d_2(a,b))= \phi(a)=a=d_2(a\moi a\oplus b)=d_2(\phi(a,b)) ,\\
 & \phi(s_0(a))=\phi (0,a)=(0\moi a)=s_0(a)=s_0(\phi(a)),\\
 & \phi(s_1(a))=\phi(a,0)=id_a=s_1(a)=s_1(\phi(a)).
\end{align*}

Pour tout $k\geq 1$, en appliquant le foncteur $B^{(k)}$ on obtient des foncteurs exacts, 
$$\xymatrix{ \phi_k : B^{(k)} B_W C\ar[r] &  B^{(k)} S^{sc}_\bul C \ar[r]^\sim& S^{sc,(k)}_\bul C }$$
dans $\del^k-WCat$. Le foncteur $\phi$ est une équivalence, et on en déduit que $\phi_k$ est une équivalence pour tout $k\geq 1$. En appliquant la K-théorie niveau par niveau et la réalisation on obtient une équivalence 
$$\xymatrix{\re{K(B^{(k)} B_W C)}\ar[r]^\sim & \re{K(S^{sc,(k)}_\bul C)} }.$$
Vu la définition du foncteur $\mcal{B}$ en \ref{groupsp} et les isomorphismes (\ref{prod}), on remarque que l'espace $\re{K(B^{(k)} B_W C)}$ est identique au niveau $k$ du spectre $\mcal{B}(K^\gam(C))$, c'est à dire du spectre $\kc(C)$. On a donc des morphismes $\kc(C)_k\lmo \re{K(S^{sc,(k)}_\bul C)}$. D'autre part, on a pour tout $k\geq 1$, un foncteur d'inclusion
$$S^{sc,(k)}_\bul C\hookrightarrow S^{(k)}_\bul C$$
des cofibrations scindées dans les cofibrations non nécessairement scindées. Un reformulation du théorème d'additivité de Waldhausen (\cite[Prop.1.3.2]{wald}) nous dit que le morphisme induit 
$$\xymatrix{K(S^{sc, (k)}_\bul C)\ar[r]^-{\sim} & K(S^{(k)}_\bul C) }$$
est une équivalence niveau par niveau dans $\del^k-SSet$, et donc induit une équivalence sur les réalisations 
$$\xymatrix{\re{K(S^{sc, (k)}_\bul C)}\ar[r]^-{\sim} & \re{K(S^{(k)}_\bul C)} }$$
On a donc construit (par composition) une collection de morphismes 
$$\xymatrix{\kc(C)_k\ar[r]^\sim & \kc'(C)_k}$$
qui sont des équivalences dans $SSet$. Cette collection de morphismes forme un morphisme de spectre "à homotopie près" dans le sens où les diagrammes obtenus avec les morphismes de transitions,
$$\xymatrix{ S^1\sm \mcal{B}(K^\gam(C))_k \ar[r] \ar[d] &  \mcal{B}(K^\gam(C))_{k+1} \ar[d] \\ S^1\sm \re{K(S^{(k)}_\bul C)} \ar[r] & \re{K(S^{(k+1)}_\bul C)}  }$$
sont commutatifs à homotopie près. Cela vient du fait que le passage de $B^{(k)} S^{sc}_\bul C$ à  $S^{sc,(k)}_\bul C$ constitue un morphisme de spectres, c'est-à-dire que pour tout $n_1,\hdots, n_k$, les diagrammes
$$\xymatrix{ S^{sc}_{n_{k-1}} \hdots S^{sc}_{n_1} C \ar[r] \ar[d] &  S^{sc}_{n_1\hdots n_{k-1}} \ar[d] \\S^{sc}_{n_k} S^{sc}_{n_{k-1}} \hdots S^{sc}_{n_1} C  \ar[r] & S^{sc}_{n_1\hdots n_k } }$$
sont commutatifs à isomorphismes près, où le morphisme vertical gauche est obtenu comme les morphismes (\ref{fonctexs}) et le morphisme vertical droit est induit par un morphisme $\und{n_1}\sm\hdots\sm \und{n_{k-1}}\lmo \und{n_1} \sm \hdots \sm\und{n_k}$ dans $\gam$ dépendant du choix d'un élément dans $\und{n_k}$. En effet le morphisme vertical gauche est le morphisme qui définit les morphismes de transitions du spectre $\kc'(C)$ et le morphisme vertical droit est le morphisme qui définit les morphismes de transitions du spectres $\kc(C)$ (vu la description du foncteur $\mcal{B}$ en \ref{groupsp}). On en déduit qu'il existe un isomorphisme 
$$\kc(C)\simeq \kc'(C)$$
dans la catégorie homotopique stable. 

\end{rema}

\paragraph{Cas des cofibrations scindées.}

Soit $C$ une catégorie de Waldhausen quelconque. On rappelle qu'on a un $\gam$-objet spécial $B_W C$ dans $WCat$ qui encode la structure de monoïde donné par la somme. En prenant le nerf des équivalences niveau par niveau, on a un $\gam$-espace spécial $NwB_W C$. On rappelle qu'on a aussi défini un $\gam$-espace de K-théorie noté $K^\gam(C)=K(B_W C)$, qui est un $\gam$-espace très spécial. On a en toute généralité un morphisme de $\gam$-espaces,
\begin{equation}\label{morphgam}
NwB_W C\lmo K^\gam (C).
\end{equation}
On définit ce morphisme de la manière suivante. Soit $n\geq 0$ un entier et $D=(B_W C)_n$ la catégorie de Waldhausen qui est le niveau $n$ de $B_W C$. On a $K^\gam (C)_n=K((B_W C)_n)=K(D)$. Il faut donc définir un morphisme $NwD\lmo K(D)=\Omega\re{NwS_\bul D}$, ou encore un morphisme $S^1\sm NwD\lmo \re{NwS_\bul D}$. Mais ce morphisme est simplement l'inclusion du $1$-squelette dans la réalisation 
$$\re{NwS_\bul D}\simeq coeq(\xymatrix{\coprod_{n\in \del} NwS_n D\times \del^n & \coprod_{(n_1\mo n_2)\in \del} NwS_{n_2} D\times \del^{n_1}  \dar[l]  )}.$$
On vérifie directement que la collection de morphismes $Nw(B_W C)_n\lmo K^\gam(C)_n$ définit bien un morphisme de $\gam$-espaces. Comme le $\gam$-espace $K^\gam(C)$ est très spécial on a donc un morphisme de $\gam$-espaces très spéciaux, 
\begin{equation}\label{kgam}
(NwB_W C)^+\lmo K^\gam(C).
\end{equation}

\begin{lem}\label{cofsc}
Soit $C$ une catégorie de Waldhausen telle que toute cofibration de $C$ est scindée (c'est-à-dire qu'elle admet une section). Alors le morphisme \emph{(\ref{kgam})} est une équivalence niveau par niveau. 
\end{lem} 

\begin{proof} 
Comme la source et le but sont des $\gam$-espaces spéciaux, il suffit de vérifier que le morphisme induit une équivalence sur le niveau $1$. On note $f:(NwB_W C)^+_1\lmo K^\gam (C)_1$ ce morphisme. Le morphisme $\phi$ défini plus haut nous donne un morphisme de $\del$-espaces, 
$$NwB_W C\lmo NwS^{sc}_\bul C$$ 
qui est une équivalence niveau par niveau. Mais vu l'hypothèse sur les cofibrations de $C$, le morphisme canonique de $\del$-espaces $NwS_\bul^{sc}C\lmo NwS_\bul C$ est une équivalence. Le morphisme composé $NwB_W C\lmo NwS_\bul C$ est donc une équivalence. On en déduit une équivalence sur les complétions en groupes $(NwB_W C)^+\lmo (NwS_\bul C)^+$. 
Mais le $\del$-espace $(NwS_\bul C)^+$ a le même niveau $1$ que le $\gam$-espace $K^{\gam}(C)$, c'est-à-dire l'ensemble simplicial $K(C)$ (cf. exemple \ref{exespk}), et le morphisme $f$ est équivalent au morphisme $(NwB_W C)^+_1\lmo (NwS_\bul C)^+_1 $. On en déduit que $f$ est une équivalence et donc que le morphisme du lemme est une équivalence niveau par niveau.  
\end{proof}

\subsection{K-théorie algébrique des dg-catégories}

\subsubsection{Rappels sur la théorie homotopique des dg-catégories}

Tout le long de ce travail on fixe un anneau de base, associatif, commutatif, unitaire $k$. Vu que l'on ne va considérer que des schémas de type fini sur $k$, on suppose de plus que $k$ est un anneau noethérien. On note $C(k)$ la catégorie des complexes non-bornés de $k$-modules $\U$-petits. On prend la convention d'écrire nos complexes cohomologiquement, c'est-à-dire avec une différentielle qui augmente le degré de $1$. Par définition une $k$-dg-catégorie $\U$-petite est une catégorie $\U$-petite enrichie sur la catégorie monoïdale $C(k)$. Celles-ci forment une catégorie $\V$-petite notée $\dgcat$ dans laquelle les morphismes sont par définition les foncteurs $C(k)$-enrichis (ou dg-foncteurs). Dans toute la suite, l'expression "soit $T$ une dg-catégorie" signifie soit $T$ un objet de $\dgcat$. Nos références pour la théorie des dg-catégories sont \cite{kedg, ldgcat, dgmor, modob}. Si $T$ est une dg-catégorie, et $x,y$ deux objets de $T$, on note $T(x,y)$ le complexe des morphismes entre $x$ et $y$ dans $T$. On note $[T]$ sa catégorie homotopique. Il s'agit de la catégorie qui a les mêmes objets que $T$ mais les morphismes sont donnés par le $H^0$ des complexes de morphismes de $T$, avec la composition naturelle et les identités naturelles. On note $T^{op}$ la dg-catégorie opposée de $T$. On note $T^{op}-Mod$ la catégorie des $T^{op}$-dg-modules, c'est-à-dire des dg-foncteurs $T^{op}\lmo C(k)$, avec pour morphismes les transformations naturelles de degré $0$.
On munit la catégories $C(k)$ de la structure de modèles projective (voir \cite{hoveymod}), dans laquelle les équivalences et les fibrations sont définies point par point. On munit $T^{op}-Mod$ de la structure de modèles projective en suivant \cite{modob}. C'est alors une catégorie de modèles stable au sens de \cite[chapter 7]{hoveymod}. On note $\m{T}$ la dg-catégorie ($\V$-petite) des $T^{op}$-dg-modules \emph{cofibrants}. On a alors une équivalence de catégories triangulées $[\m{T}]\simeq D(T)$. Si $x$ est un objet de $T$, on note $\uh_x$ le $T^{op}$-dg-module représenté par $x$, c'est-à-dire $y\longmapsto T(y,x)$. Ceci définit le plongement de Yoneda 
$$\uh : T\lmo T^{op}-Mod.$$
C'est un dg-foncteur quasi-pleinement fidèle. L'évaluation en $x$ donne un foncteur
$$\xymatrix{ T^{op}-Mod\ar[r]^-{x^*} & C(k)}$$
qui commute aux limites, et a pour adjoint à gauche le foncteur noté $x_!$. Le foncteur $x^*$ est de Quillen à droite pour les structures de modèles projective de chaque côtés. On en déduit que $x_!$ est de Quillen à gauche. Par Yoneda on vérifie que $x_!(\unit)=\uh_x$, où $\unit$ est le complexe avec $k$ en degré $0$ et $0$ partout ailleurs. Comme $\unit$ est un objet cofibrant de $C(k)$, on en déduit que $\uh_x$ est un $k$-dg-module cofibrant pour tout objet $x$ de $T$. Le plongement de Yoneda se factorise donc en un dg-foncteur quasi-pleinement fidèle\footnote{Comme la dg-catégorie $\m{T}$ est seulement $\V$-petite, ce plongement de Yoneda n'est pas un morphisme dans $\dgcat$ mais dans la catégorie des dg-catégories $\V$-petites.},
$$\uh : T\lmo \m{T}.$$
On note $\mpa{T}$ la sous-dg-catégorie pleine de $\m{T}$ des dg-modules cofibrants parfaits, c'est-à-dire des dg-modules homotopiquement de présentation finie dans la catégorie de modèles $T^{op}-Mod$ au sens de \cite{modob}. C'est une $k$-dg-catégorie essentiellement $\U$-petite. On rappelle que $\mpa{T}$ est la plus petite sous-dg-catégorie de $\m{T}$ qui contient l'image essentielle du plongement de Yoneda et qui est stable par rétracts, shifts et pushouts homotopiques. Le plongement de Yoneda se factorise par $\mpa{T}$. Si on note $D_{pe}(T)$ la sous-catégorie pleine de $D(T)$ formée des $T^{op}$-dg-modules parfaits, on a une équivalence de catégories triangulées $[\mpa{T}]\simeq D_{pe}(T)$. 

On utilisera les deux structures de modèles suivantes (mises au jour par Tabuada \cite{tabth}) : 

\begin{itemize}
\item La structure de modèles standard (ou de Dwyer--Kan) pour laquelle les équivalences sont les quasi-équivalences. On note $Ho(\dgcat)$ sa catégorie homotopique. 

\item La structure de modèles Morita, notée $\dgmor$, pour laquelle les équivalences sont les équivalences de Morita dérivées. On rappelle qu'un dg-foncteur $f:T\lmo T'$ est une équivalence de Morita dérivée si le foncteur triangulé associé $\lef f_! : D_{pe}(T)\lmo D_{pe}(T')$ est une équivalence de catégories. Sa catégorie homotopique est donc notée $Ho(\dgmor)$. 
\end{itemize}

On rappelle la notion de suite exacte de dg-catégories. Une suite de morphismes 
\begin{equation}\label{suite}
\xymatrix{T'\ar[r]^-i & T\ar[r]^-p & T''} 
\end{equation}
dans $\dgcat$ est dite \emph{exacte}  si la suite de catégories triangulées 
$$\xymatrix{D_{pe}(T')\ar[r]^-{i_!} & D_{pe}(T)\ar[r]^-{p_!} &D_{pe}(T'')}$$ 
est exacte, c'est-à-dire que $i_!$ est pleinement fidèle et $p_!$ induit une équivalence à facteurs direct près entre $D_{pe}(T'')$ et le quotient de Verdier $D_{pe}(T)/D_{pe}(T')$. Une suite comme en (\ref{suite}) est dite \emph{strictement exacte scindée} si elle exacte et si de plus $i_!$ a un adjoint à droite noté $i^*$, $p_!$ a un adjoint à droite noté $p^*$ tels que $i^*i_!\simeq id_{\parf(T')}$ et $p_!p^*\simeq id_{\parf(T'')}$ via les morphismes d'adjonction. Une \emph{suite exacte scindée} dans $\dgcat$ est une suite dans $\dgcat$ qui est isomorphe dans $Ho(\dgmor)$ à une suite strictement exacte scindée.

\subsubsection{K-théorie connective}

On rappelle la définition de la K-théorie algébrique connective des dg-catégories telle qu'elle est étudiée dans \cite{tabhkt} (voir aussi \cite{kedg}). 

Soit $T\in \dgcat$ une dg-catégorie. On considère $\parf(T)$, la catégorie dont les objets sont les mêmes que la dg-catégorie $\mpa{T}$, mais on ne retient que les morphismes de degré $0$. On munit $\parf(T)$ d'une structure de catégorie de Waldhausen induite par la structure de modèles sur $T^{op}-Mod$. En d'autres termes, un morphisme est une équivalence (resp. une cofibrations) dans $\parf(T)$ si c'est une équivalence (resp. une cofibration) dans $T^{op}-Mod$. Les axiomes qui définissent une structure de catégorie de Waldhausen sont satisfaits car le pushout homotopique de deux dg-modules parfaits est encore un dg-module parfaits et l'axiome (weq2, gluing lemma) est satisfait par le lemme du cube \cite[lemma 5.2.6]{hoveymod} dans la catégorie de modèles $T^{op}-Mod$. De plus la catégorie de Waldhausen $\parf(T)$ satisfait l'axiome de saturation, l'axiome d'extension, a un objet cylindre qui satisfait l'axiome du cylindre. 

Soit $f:T\lmo T'$ un morphisme dans $\dgcat$. Alors $f$ induit une paire de Quillen, 
$$\xymatrix{ T^{op}-Mod \ar@<3pt>[r]^-{f_!}  &T'^{op}-Mod \ar@<3pt>[l]^-{f^*} }$$
où l'adjoint à droite est définie par $f^*(E)=E\circ f$ pour tout $E\in T'^{op}-Mod$. Le foncteur $f^*$ est clairement de Quillen à droite pour les structures projectives, donc $f_!$ est de Quillen à gauche. Il s'en suit que $f_!$ préserve les dg-modules parfaits et donc induit un foncteur exact
$$f_!:\parf(T)\lmo \parf(T').$$
Cette construction définit un foncteur faible $\dgcat\lmo WCat$ auquel on applique la procédure de strictification canonique pour obtenir un foncteur 
$$\parf : \dgcat \lmo WCat.$$

\begin{df} 
\emph{L'espace de K-théorie d'une dg-catégorie $T$} est par définition l'ensemble simplicial pointé $K(T):=K(\parf(T))$. Ceci définit un foncteur 
$$K : \dgcat\lmo SSet_*.$$
\emph{Le spectre de K-théorie connective d'une dg-catégorie $T$ }est par définition le spectre symétrique connectif $\kc(T):=\kc(\parf(T))$. Ceci définit un foncteur 
$$\kc:\dgcat\lmo Sp.$$
\end{df} 

\begin{rema}\label{algebra}
Soit $A$ une $k$-algèbre associative. On note $\proj(A)$ la catégorie de Waldhausen des $A$-modules à droites projectifs de type fini, dont les cofibrations sont les monomorphismes admissibles (c'est-à-dire les monomorphismes scindés à conoyau projectif de type fini) et les équivalences sont les isomorphismes. Si $A$ est commutative, c'est la catégorie des fibrés vectoriels sur $\spec(A)$. Alors le foncteur d'inclusion $\proj(A)\hookrightarrow\parf(A)$ induit une équivalence $K(\proj(A))\simeq K(\parf(A))$ dans $SSet$, et donc une équivalence $\kc(\proj(A))\simeq \kc(\parf(A))$ dans $Sp$. Cette équivalence peut s'obtenir en considérant la catégorie intermédiaire des complexes strictement parfaits sur $A$ notée $\parf^{\mrm{strict}} (A)\hookrightarrow \parf(A)$. Par le théorème de Gillet--Waldhausen \cite[Thm.1.11.7]{tt}, l'inclusion $\proj(A)\hookrightarrow \parf^{\mrm{strict}} (A)$ induit une équivalence en K-théorie. Pour l'autre inclusion, on peut considérer les dg-catégories associées $\mpa{A}^{\mrm{strict}}\hookrightarrow \mpa{A}$ des $A$-dg-modules strictement parfaits à l'intérieur des $A$-dg-modules parfaits. Ce foncteur est une équivalence de Morita dérivée et donc induit une équivalence en K-théorie par ce qui suit. 
\end{rema} 

Les propriétés fondamentales de la K-théorie connective sont rappelées dans la proposition suivante. Dans l'esprit de la théorie motivique de Tabuada, ces propriétés permettent de donner la propriété universelle de la K-théorie connective des dg-catégories : elle est coreprésentée par l'unité de la structure monoïdale dans la catégorie motivique additive de Tabuada. Ces résultats sont certainement bien connus, mais nous incluons quelques preuves par manque de références précises. 

\begin{prop}
\begin{enumerate}
\item Le foncteur $\kc : \dgcat\lmo Sp$ commute aux colimites homotopiques filtrantes de dg-catégories. 
\item Le foncteur $\kc$ envoie les équivalences Morita dérivées sur des équivalences de spectres. 
\item Pour toute suite exacte \textit{\textbf{scindée}} $\xymatrix{T'\ar[r]^-i & T\ar[r]^-p & T''}$, les foncteurs $i^*$ et $p^*$ préservent les parfaits et le morphisme induit 
$$(i_!,p^*):\kc(T')\oplus \kc(T'')\lmo \kc(T)$$ 
est un isomorphisme dans $Ho(Sp)$.
\end{enumerate}
\end{prop} 

\begin{proof} 
2. C'est une conséquence directe du résultat de Thomason \cite[Thm. 1.9.8]{tt} une fois qu'on a remarqué que $Ho(\parf(T))\simeq D_{pe}(T)$. 

1. Le spectre de K-théorie connective est définie à partir de l'espace de K-théorie en appliquant des foncteurs qui commutent aux colimites, il suffit donc de le prouver pour l'espace de K-théorie. Soit $T=hocolim_{i\in I} T_i$ une colimite homotopique filtrante dans $\dgcat$. D'après \cite[Prop. 2.2]{modob}, cette colimite homotopique est équivalente à une brave colimite (non-homotopique). On a donc un isomorphisme $T\simeq colim_{i\in I} T_i$ dans $Ho(\dgcat)$. On a une suite de morphismes 
$$\xymatrix{colim_{i\in I} K(T_i)=colim_{i\in I} K(\parf(T_i)) \ar[r] & K(colim_{i\in I} \parf(T_i) ) \ar[r] &  K(\parf(T))=K(T) }$$
On sait d'après Waldhausen que la K-théorie commute aux colimites filtrantes de catégories de Waldhausen, donc le premier morphisme est un isomorphisme dans $Ho(SSet)$. D'après \cite[Lem. 2.10]{modob}, le morphisme naturel 
$$colim_{i\in I} \mpa{(T_i)}\lmo \mpa{T}$$ 
est une équivalence dans $\dgcat$. On en déduit que le foncteur exact 
$$colim_{i\in I} \parf (T_i)\lmo \parf(T)$$
induit une équivalence sur les catégories dérivées, et induit donc une équivalence en K-théorie par le résultat de Thomason. 
\\

3. Se déduit du théorème d'additivité de Waldhausen de la manière suivante. Vu le point 2. on peut supposer que la suite est strictement scindée. Soit $\xymatrix{T'\ar[r]^-i & T\ar[r]^-p  & T''}$ une suite strictement exacte scindée dans $\dgcat$. L'inverse à $i_!\oplus p^* : \parf(T')\times \parf(T'') \lmo\parf(T)$ en K-théorie est donné par le foncteur exact $(i^*,p_!):\parf(T)\lmo \parf(T')\times \parf(T'')$.  En effet pour tout $(E,F)\in \parf(T')\times \parf(T'')$ on a des isomorphismes
$$(i^*,p_!)\circ (i_!\oplus p^*) (E,F)=(i^*,p_!) (i_!(E)\oplus p^*(F))=(i^*i_!(E)\oplus i^*p^*(F), p_!i_!(E)\oplus p_!p^*(F))$$
Mais $i^*p^*\simeq 0$, $p_!i_!\simeq 0$, $i^*i_!\simeq id_{\parf(T')}$ et $p_!p^*\simeq id_{\parf(T'')}$. Donc $(i_!\oplus p^*) \circ (i^*,p_!)\simeq id$. Dans l'autre sens on a 
$$ (i_!\oplus p^*) \circ (i^*,p_!) (E) = (i_!\oplus p^*) (i^*(E), p_!(E))=i_!i^*(E)\oplus p^*p_!(E).$$
D'après la suite exacte scindée $\xymatrix{D_{pe}(T')\ar[r]^-{i_!} & D_{pe}(T)\ar[r]^-{p_!} &D_{pe}(T'')}$, on a une suite de cofibrations de foncteurs exacts
$$i_!i^*\lmo id_{\parf(T)}\lmo p^*p_!.$$
Or le théorème d'additivité de Waldhausen nous dit que l'égalité $i_!i^*+p^*p_!=id$ a lieu en K-théorie, ce qui implique que $(i^*,p_!)$ est bien l'inverse de $i_!\oplus p^*$. \end{proof}

\subsubsection{K-théorie non-connective}

On rappelle comment définir la K-théorie non-connective des dg-catégories en suivant Tabuada--Cinsinki \cite{nck}. Cette définition coïncide avec la construction de Schlichting par \cite[Prop 6.6]{nck}.

L'ingrédient principal est le foncteur de "complétion en somme dénombrable" ou "enveloppe flasque" : 
$$\mcal{F} : \dgcat\lmo \dgcat,$$
(voir \cite[Section 6]{nck}). Cet objet vient avec un morphisme $T\lmo \mcal{F}(T)$.  La propriété essentielle de la dg-catégorie $\mcal{F}(T)$ est qu'elle admet des sommes dénombrables et donc elle vérifie $\kc(\mcal{F}(T))=0$. 

On définit alors le foncteur de suspension
$$\mcal{S} : \dgcat\lmo \dgcat,$$
par $\mcal{S}(T):=\mcal{F}(T)/T$, la dg-catégorie quotient $Ho(\dgcat{k})$. 
La suite de spectres symétriques $(\kc(\mcal{S}^n(T))_{n\geq 0}$ forme un spectre dans $Sp$ (see \cite[Prop 7.2]{nck}) on définit le spectre de K-théorie non-connective comme le niveau $0$ du $\Omega$-spectre associé. 

\begin{df}\emph{(\cite[Prop 7.5]{nck})}
Le spectre de K-théorie algébrique non-connective d'une dg-catégorie $T$ est par définition le spectre symétrique
$$\kn(T):=hocolim_{n\geq 0} \kc(\mcal{S}^n (T))[-n].$$
Ceci définit un foncteur
$$\kn : \dgcat\lmo Sp.$$
\end{df}

Par définition, pour tout $T\in \dgcat$ on a un morphisme naturel $\kc(T)\lmo \kn(T)$. Les propriétés essentielles de la K-théorie algébrique non-connective sont listées dans la proposition suivante. 

\begin{prop}\label{propkn}
\begin{description}
\item[a.] Pour tout $T\in \dgcat$, le morphisme naturel $\kc(T)\lmo \kn(T)$ induit un isomorphisme sur les $\pi_i$ pour tout $i\geq 0$. Par conséquent $\kc(T)$ est le revêtement connectif du spectre $\kn(T)$. 
\item[b.] Le foncteur $\kn$ commute aux colimites homotopiques filtrantes dans $\dgcat$. 
\item[c.] Le foncteur $\kn$ envoie les équivalences Morita dérivées sur des équivalences dans $Sp$. 
\item[d.] Soit $\xymatrix{T'\ar[r]^-i & T\ar[r]^-p & T''}$  une suite exacte de dg-catégories. Alors la suite de morphismes de spectres,
$$\xymatrix{ \kn(T') \ar[r]^{i_!} & \kn(T) \ar[r]^{p_!} & \kn(T'')}$$
est une suite de fibration. 
\end{description}
\end{prop}

\begin{proof} 
\begin{description}
\item[a.] Cela vient de la définition et du fait que $\kn(T)$ est un $\Omega$-spectre. 
\item[b.] Cela vient du fait correspond pour le foncteur $\kc$, de ce que le foncteur revêtement connectif commute aux colimites filtrantes, et de ce qu'un morphisme de spectre $E\lmo E'$ est une équivalence si et seulement si les morphismes $\widetilde{E}[n]\lmo \widetilde{E'}[n]$ sont des équivalences pour tout $n\geq 0$. 
\item[c.] \cite[12.3, Prop 3]{schl}.
\item[d.] \cite[12.1, Thm 9]{schl}. 
\end{description}
\end{proof}

\begin{nota}\label{notaksch}
En suivant Schlichting \cite[6.5]{schl}, on définit la K-théorie d'un schéma comme la K-théorie de sa catégorie des complexes parfaits. Si $X\in \schk$ est un schéma, on considère la dg-catégorie $\lpe(X)$ des complexes parfaits de $\ocal_X$-modules quasi-cohérents sur $X$. On suppose qu'on a un modèle fonctoriel
$$\lpe : \schk \lmo \dgcat.$$
On réfère à \cite[§8.3]{dgmor} pour une définition convenable de $\lpe(X)$. La K-théorie de $X\in\schk$ est par définition le spectre symétrique $\kn(X):=\kn(\lpe(X))$. Ceci définit un foncteur toujours noté
$$\kn : \schk\lmo \dgcat.$$
On rappelle que Schlichting a montré (voir \cite[Thm.5]{schl}) que pour tout schéma $X$ de type fini sur $k$ (qui est donc quasi-compact quasi-séparé) la K-théorie au sens de la définition précédente coïncide avec la K-théorie définie par Thomason--Trobaugh (\cite{tt}). 
\end{nota}

\subsection{Caractère de Chern algébrique}

\subsubsection{Homologies cycliques}

C'est dans le but de généraliser la cohomologie de de Rham aux algèbres associatives que Tsygan (voir \cite{tsyhom}) a définie l'homologie cyclique. Le parallèle est donné par le théorème dit HKR pour Hochschild--Konstant--Rosenberg, qui énonce l'existence d'un isomorphisme entre l'homologie cyclique périodique d'une algèbre commutative et la cohomologie de de Rham des formes différentielles algébriques de cette algèbre commutative. 

On fixe ici les notations concernant l'homologie cyclique. Les différentes versions de l'homologie cyclique peuvent toutes être définies à partir d'un même objet appelés le complexe mixte associé à une dg-catégorie comme définie par Keller dans \cite{kelcyclic}. Le complexe mixte est un foncteur
$$Mix=Mix(-\mid k): \dgcat\lmo \Lambda-Mod,$$
à valeurs dans la catégorie des $\Lambda$-dg-modules avec $\Lambda$ la $k$-dg-algèbre engendrée par un élément $B$ en degré $-1$ soumis à la relation $B^2=0$ et $d(B)=0$. Un $\Lambda$-dg-module est communément appelé un complexe mixte. Si $T\in \dgcat$ est une dg-catégorie (localement cofibrante), on définit un complexe précyclique de $k$-modules au sens de \cite[2.1]{kelinv}, noté $C(T)$ dont le $n$-ième niveau ($n\geq 0$) est donné par 
$$C(T)_n=\bigoplus_{x_0,x_1,\hdots,x_n\in T} T(x_n,x_0) \te_k T(x_{n-1},x_n)\te_k T(x_{n-2},x_{n-1})\te_k \hdots \te_k T(x_0,x_1)$$
où la somme est prise sur les suites d'objets de $T$. La différentielle $d:C(T)_n\lmo C(T)_{n-1}$ est défini par 
$$d(f_n\te f_{n-1}\te \hdots \te f_0) = f_{n-1}\te \hdots \te f_1\te f_0 f_n + \sum _{i=1}^n (-1)^n f_n\te \hdots\te f_{i-1}f_i \te \hdots \te f_0,$$
où les $f_i$ sont des éléments homogènes. L'opérateur cyclique $t:C(T)_n\lmo C(T)_n$ est défini par 
$$t(f_n\te \hdots \te f_0) = (-1)^{n+1} f_0\te f_n\te \hdots \te f_1.$$
Le complexe mixte associé au complexe précyclique $C(T)$ (au sens de \cite[2.1]{kelinv}) est le complexe mixte $Mix(T)$. Celui-ci est un invariant relatif et dépend de manière cruciale de l'anneau de base $k$. Ce n'est que par abus de notation qu'on omet de préciser l'anneau de base dans la notation. Par défaut tout les spectres d'homologie cyclique dans la suite de ce travail sont calculés relativement à l'anneau de base $k$. Keller a montré que le foncteur $Mix$ est un invariant localisant au sens de Tabuada \cite{tabth}. 

\begin{theo}\label{cyclic} \emph{(Keller \cite{kelcyclic})} --- Le foncteur $Mix$ commute aux colimites homotopiques filtrantes dans $Ho(\dgcat)$, envoie les équivalences Morita de dg-catégories sur des équivalences et envoie les suites exactes de dg-catégories sur des triangles distingués dans la catégorie dérivée des complexes mixtes. 
\end{theo} 

On note
$$H : C(\Z)\lmo Sp $$
le foncteur standard qui va des complexes (non bornés, noté cohomologiquement) de $\Z$-modules dans les spectres symétriques. Bien que l'on peut se référer à \cite[après le cor.B.1.8]{shischcat} pour sa construction, on rappelle une définition de $H$. Ce foncteur nous est utile pour passer des dg-modules aux spectres symétriques, notamment dans le cas de l'homogie cyclique. Si $E\in C(\Z)$ est un complexe et $n$ un entier, la notation $E[-n]$ désigne le complexe $E[-n]^k=E^{k-n}$, avec pour différentielle $(-1)^n$ fois la différentielle de $E$. Pour tout ensemble simplicial $K$ on note $\Z[K]$ le $\Z$-module simplicial libre engendré par $K$, et $\tilde{\Z}[K]$ le $\Z$-module simplicial réduit libre engendré par $K$. Le décalage $E[-n]$ est alors canoniquement isomorphe au produit tensoriel $E\te_\Z \tilde{\Z}[S^n]$, avec $S^n=(S^1)^{\sm n}$. Ce complexe décalé supporte donc une action naturelle du groupe symétrique $\Sigma_n$ par permutation des facteurs. 
On note 
$$\xymatrix{ N : Mod(\Z)^{\del^{op}} \lmo C(\Z)} $$
le foncteur de normalisation standard. Pour tout $M\in Mod(\Z)^{\del^{op}}$ le complexe $NM$ est concentré en degrés négatifs ou nul. On définit un foncteur 
$$W:C(\Z)\lmo Mod(\Z)^{\del^{op}}$$
en posant pour tout $E\in C(\Z)$, $W(E)_n=\homi(N\Z[\del^n], E)$, où $\homi$ est le $\Z$-module des morphismes dans $C(\Z)$. La restriction de $W$ à la sous-catégorie de $C(\Z)$ formée des complexes concentrés en degrés négatifs est l'inverse de $N$ dans l'équivalence de Dold--Kan. On note encore $\Z$ le complexe avec $\Z$ en degré $0$ et $0$ partout ailleurs. On a $W(\Z[-1])= \tilde{\Z}[S^1]$. Le foncteur $H$ associe à un complexe $E$ le spectre symétrique (de $\Z$-modules simpliciaux) $H(E)$ dont le $n$-ième terme est  
$$H(E)_n=W(E[-n])\simeq W(E\te_\Z \tilde{\Z}[S^n])$$
et les morphismes de transitions $S^1\sm H(E)_n \lmo H(E)_{n+1}$ définis par le composé
\begin{align*}
 \Z[S^1]\sm W(E[-n])  & \lmo \tilde{\Z}[S^1]\sm W(E[-n]) \\
 & \lmo W(\Z[-1])\sm W(E[-n])\\
 &\lmos{AW} W(\Z[-1]\te_\Z E[-n])\simeq W(E[-n-1]),
\end{align*}
où $AW$ désigne le morphisme d'Alexander--Whitney (voir \cite{richter}). On peut vérifier que les morphismes de transition sont bien équivariant par rapport à l'action symétrique, et donc que $H(E)$ est un spectre symétrique de $\Z$-modules simpliciaux. 

On considère les catégories $C(\Z)$ et $Sp$ comme des catégories monoïdales avec leur structures monoïdales usuelles données par le dg-produit tensoriel et le smash produit respectivement. Le foncteur $H$ ainsi défini est un foncteur lax monoïdal et induit donc un foncteur entre les monoïdes et les modules sur ces monoïdes. On a donc un anneau en spectre $Hk$, une $Hk$-algèbre en spectre $H\Lambda$ et le foncteur $H$ induit un foncteur,
$$H: \Lambda-Mod\lmo \hlamb$$ 
à valeurs dans la catégorie des $H\Lambda$-modules en spectres. On munit $\Lambda-Mod$ et $\hlamb$ de la structure de modèles donnée par \cite[sect 4]{ss-alg} où un morphisme est une équivalence si c'est une équivalence sur les dg-modules et spectres respectivement. Par définition, le foncteur $H$ préserve les équivalences. 

On note alors $\hh:=H\circ Mix$ le foncteur composé
$$\xymatrix{ \dgcat \ar[r]^-{Mix} \ar@/^2pc/[rr]^-{\hh} & \Lambda-Mod \ar[r]^-{H} & \hlamb }$$
et on l'appelle l'homologie de Hochschild sur $k$. C'est encore un invariant localisant (c'est-à-dire qu'il vérifie les conditions du Thm. \ref{cyclic}). C'est aussi un foncteur  monoïdal faible car $H$ et $Mix$ sont monoïdaux faibles. Dans la définition suivante on considère $Hk$ comme un $H\Lambda$-module par l'augmentation naturelle $H\Lambda\lmo Hk$. 

\begin{df} Soit $T\in \dgcat$. 
\begin{itemize}
\item Le \emph{spectre symétrique d'homologie cyclique de $T$}  ou l'\emph{homologie cyclique de $T$} est le spectre 
$$\hc (T) := \hh(T)\sml_{H\Lambda} Hk.$$
\item Le \emph{spectre symétrique d'homologie cyclique négative de $T$}  ou l'\emph{homologie cyclique négative de $T$} est le spectre 
$$\hcn (T):=\rhomi_{H\Lambda} (Hk,\hh(T)).$$
\item Le $Hk$-module $\hcn(T)$ est un module sur la $Hk$-algèbre $\rhomi_{H\Lambda} (Hk,Hk)\simeq Hk[u]$ avec $u$ un générateur de degré $-2$. On définit \emph{le spectre symétrique d'homologie périodique de $T$} ou \emph{l'homologie périodique de $T$} comme étant le spectre 
$$\hp(T):=\hcn(T)\wh_{Hk[u]} Hk[u,u^{-1}].$$
\end{itemize}
On obtient des foncteurs
$$\hcn : \dgcat\lmo Hk[u]-Mod_\s,$$
$$\hp : \dgcat\lmo Hk[u,u^{-1} ]-Mod_\s.$$
On a par définition un morphisme naturel de $Hk[u]$-modules $\hcn\lmo \hp$. 
\end{df}

\begin{rema}
La définition du foncteur d'homologie périodique $\hp$ donnée ici n'est pas la définition standard (comme dans \cite[p.210]{kasselcyclic}) mais \cite[Thm.6.1.24]{rosenbergk} énonce le fait que la définition classique coïncide avec la définition donnée ici pour n'importe quelle algèbre associative unitaire. 
\end{rema}

\begin{nota}\label{notapref}
Si $\mbf{E}:\dgcat\lmo V$ est un foncteur, on définit un foncteur,
$$\und{\mbf{E}} : \dgcat\lmo Pr(\affk, V)$$
par $\und{\mbf{E}}(T)(\spec(A))=\mbf{E}(T,A):=E(T\tel_k A):=E(QT\te_k A)$ où $Q:\dgcat\lmo \dgcat$ est un foncteur qui vérifie les propriétés suivantes : 
\begin{itemize}
\item Il existe une transformation naturelle en $T$, $Q(T)\lmo T$ qui est une équivalence. 
\item Pour tout $T\in \dgcat$, la dg-catégorie $Q(T)$ est plate\footnote{On dit qu'une dg-catégorie $T$ est plate si pour toute paire d'objets $x,y\in T$, le complexe $T(x,y)$ est un complexe constitué de $k$-modules plats.} sur $k$. 
\end{itemize}
Vu que $\dgcat$ est une catégorie de modèles pour laquelle les équivalences sont les quasi-équivalences, un tel foncteur $Q$ est par exemple donné par un foncteur de remplacement cofibrant dans $\dgcat$. Dans cette thèse, nous sommes principalement concerné par $k=\C$, on pourra dans ce cas là choisir simplement pour $Q$ le foncteur identité. 

On utilisera cette notation pour les invariants classiques de dg-catégories $\mbf{E}=\kn, \kc, \hh, \hc, \hcn, \hp$. 
\end{nota}

\subsubsection{Le caractère de Chern $\ka$-linéaire via la théorie de Cisinski--Tabuada}\label{caracalg}

Dans cette partie nous utilisons la catégorie motivique de Tabuada (voir \cite{tabhkt}, \cite{nck}) et les théorèmes de Tabuada et de Cisinski--Tabuada de coreprésentabilité de la K-théorie dans le but de construire un caractère de Chern $\ka$-linéaire, où $\ka$ est le préfaisceau spectral de K-théorie $\ka(\spec(A))=\kn(A)$. Le théorème de coreprésentabilité nous donne une structure naturelle d'anneau en spectre sur $\ka$ et de $\ka$-module sur $\ukn(T)$ pour toute dg-catégorie $T$, ce qui nous permettra de considérer la K-théorie semi-topologique comme un module et d'inverser l'élément de Bott. 

Le motivateur localisant universel $\mcal{M}_{loc}(k)$ de Tabuada (\cite[Déf.11.2]{tabhkt}) peut être décrit (par \cite[Prop.12.4]{tabhkt}) comme le dérivateur associé à la catégorie de modèles projective des préfaisceaux en spectres sur les dg-catégories de type fini, localisée suivant les équivalences Morita et les suites exactes courtes. Si $\mrm{Hodgcat}$ désigne le dérivateur associé à la catégorie de modèles $\dgmor$, il existe un morphisme de dérivateurs
$$\mcal{U}_l : \mrm{Hodgcat} \lmo \mcal{M}_{loc}(k) $$
appelé l'invariant localisant universel, qui est universel pour la propriété de commuter aux colimites homotopiques filtrantes de dg-catégories, d'envoyer les équivalences Morita sur des équivalences et les suites exactes courtes sur des triangles distingués (voir \cite[Thm.11.5]{tabhkt}). Dans \cite{tabcisym}, en modifiant sensiblement la construction de $\mcal{M}_{loc}(k)$, les deux auteurs construisent une structure monoïdale symétrique sur $\mcal{M}_{loc}(k)$ à partir du produit de convolution de Day et du produit tensoriel dérivé des dg-catégories, en utilisant le fait que celui-ci laisse invariants la classe des dg-catégories de type fini avec laquelle on construit le motivateur universel. Cette classe de dg-catégories de type fini doit en outre vérifier un certain nombre d'autres propriétés de platitude (voir \cite{tabcisym} 7.1, propriétés a) à f)). Dans cette partie on s'intéresse seulement à la catégorie homotopique $\mcal{M}_{loc}(k)(\ast)=Ho(\mloc)$. Nous énonçons alors une version simplement catégorique de la propriété universelle de Tabuada, qui sera suffisante pour notre propos. Il existe une variante de la construction de $\mcal{M}_{loc}(k)$ qui consiste à prendre les préfaisceaux sur les dg-catégories de type fini non pas à valeurs dans les catégorie $Sp$ des spectres symétriques, mais dans la catégorie $\spafk$ des préfaisceaux en spectres symétriques.

On note $\dgcattf\subseteq \dgcat$ une sous-catégorie de dg-catégories $\U$-petites qui vérifie les propriétés a) à f) de la section 7.1 de \cite{tabcisym}. Grossièrement, cela signifie que toute dg-catégorie de $\dgcattf$ est de type fini au sens de \cite[Déf.2.4]{modob}, localement plate sur $k$, et que la sous-catégorie $\dgcattf$ est stable par produit tensoriel et par le foncteur $Q$ de remplacement cofibrant fixé à la notation \ref{notapref}. On considère la catégorie $Pr(\dgcattf, \spafk)$ des préfaisceaux sur $\dgcattf$ à valeurs dans $\spafk$. On munit la catégorie $\spafk$ de la structure de modèles \emph{injective}, on réfère pour cela au théorème général d'existence \cite[Prop.A.2.8.2]{htt}. On munit alors la catégorie de préfaisceaux $Pr(\dgcattf, \spafk)$ de la structure de modèles projective. On note 
$$\uh : \dgcattf\lmo Pr(\dgcattf, \spafk)$$
le plongement de Yoneda donné par $\uh_T(T')(\spec(A))=\Hom_{\dgcat}(T',T\te_k A)$. Soit $\mcal{R}$ l'ensemble des morphismes de $Pr(\dgcattf, \spafk)$ qui sont de la forme 
\begin{enumerate}
\item $\uh_T\lmo \uh_{T'}$, avec $T\lmo T'$ une équivalence Morita. 
\item $\cone(\uh_i)\lmo \uh_{T''}$, pour $\xymatrix{T'\ar[r]^-i & T\ar[r]^-p & T''}$ une suite exacte de dg-catégories. 
\end{enumerate}
Alors $\mloc$ est définie comme la localisation de Bousfield à gauche $\spafk$-enrichi (voir \cite{barwenr} pour l'existence de localisation enrichie) : 
$$\mloc:=L_{\mcal{R}} Pr(\dgcattf, \spafk).$$
C'est donc une $\spafk$-catégorie de modèles au sens de \cite[Déf.4.2.18]{hoveymod}, et on note $\rhomi_{\mloc}$ son hom dérivé enrichi dans $\spafk$. Par \cite[7.5]{tabcisym}, c'est aussi une catégorie de modèles monoïdale au sens de \cite[Déf.4.2.6]{hoveymod}. On note par un smash $\sm$ le produit de sa structure monoïdale. L'unité est donné par l'objet $\uk$ définie par $\uk(T)=\Hom_{\dgcat}(T, k)$ et constant en tant que préfaisceau sur $\affk$. La proposition suivante découle des propriétés générales des localisations de Bousfield à gauche (voir \cite[Prop.3.3.18]{hirs}). 

\begin{prop}\label{propuni}
Le plongement de Yoneda $\uh:\dgcattf\lmo \mloc$ a les propriétés suivantes. 
\begin{enumerate}
\item Pour tout $T\in \dgcat$, $\uh_T$ est un objet cofibrant de $\mloc$. 
\item Le foncteur $\uh$ envoie les équivalences Morita dérivées sur des équivalences dans $\mloc$. 
\item Le foncteur $\uh$ envoie les suites exactes courtes de dg-catégories sur des suites de fibrations dans $\mloc$. 
\end{enumerate}
De plus, $\uh$ est universel pour ces 3 propriétés. Pour toute $\spafk$-catégorie de modèles $V$, le foncteur
$$\uh^* : \Hom_!(\mloc, V) \lmo \Hom_*(\dgcattf, V)$$
est une équivalence de catégories, où $\Hom_!$ est la catégories des foncteurs de Quillen à gauche $\spafk$-enrichis et $\Hom_*$ désigne la catégorie des foncteurs qui satisfont les propriétés 1 à 3 ci-dessus. 
\end{prop} 

\begin{rema}
Soit $E : \dgcat\lmo Sp$ un foncteur qui commute aux colimites homotopiques filtrantes, qui envoie les équivalences Morita sur des équivalences, et qui envoie les suites exactes sur des suites de fibrations ; alors $E$ se restreint en un foncteur $E : \dgcattf\lmo Sp$ qui appartient $\Hom_*(\dgcattf, \spafk)$. L'extension de $\und{E}$ à $\mloc$ est noté 
$$\und{E}_! : \mloc\lmo Sp.$$
Toute dg-catégorie $T$ s'écrit comme une colimite homotopique filtrante de dg-catégories appartenant à $\dgcattf$. Il existe donc un diagramme filtrant $\{T_\alpha\}$ de dg-catégories et une équivalence $T\simeq hocolim_\alpha T_\alpha$ dans $\dgcat$. D'autre part l'objet $\uh_T\in \mloc$ défini par 
$$\uh_T(T')(A)=\Hom_{\dgcat}(T',T\tel_k A)$$ 
représente l'objet $T$ dans $\mloc$. Comme le foncteur de Yoneda standard est équivalent au foncteur de Yoneda homotopique (voir \cite[Lem.4.2.2]{hag1}), on en déduit une équivalence $\uh_T\simeq colim_\alpha \uh_{T_\alpha}$ dans $\mloc$. On a donc des équivalences
$$\und{E}_!(\uh_T)\simeq hocolim_\alpha \und{E}_!(\uh_{T_\alpha})\simeq hocolim_\alpha \und{E}(T_\alpha)\simeq \und{E}(T).$$
L'extension $\und{E}_!$ coïncide donc à équivalence près avec $\und{E}$. Ceci se reformule de la manière suivante. On note $\Hom(\dgcattf, Sp)$ la catégorie de modèles projective des foncteurs, et $\Hom_! (\dgcat, Sp)$ la catégorie de modèles projective des foncteurs qui commutent aux colimites homotopiques filtrantes. Alors le foncteur d'inclusion $\dgcattf\lmo \dgcat$ induit une équivalence de catégories
\begin{equation}\label{inddg}
Ho(\Hom_! (\dgcat, Sp))\lmos{\sim} Ho(\Hom(\dgcattf, Sp)).
\end{equation}
Tout foncteur $E\in \Hom(\dgcattf, Sp)$ admet une unique extension à isomorphisme près dans $Ho(\Hom_! (\dgcat, Sp))$. 
Ainsi on a l'extension du foncteur homologie de Hochschild $\uhh$ toujours noté par abus $\uhh$ au lieu de $\uhh_!$.

\end{rema}

\begin{theo}\label{repk}\emph{(Cisinski--Tabuada \cite[Thm. 7.16]{nck})}
Pour tout $T\in \dgcattf$, il existe un isomorphisme canonique fonctoriel en $T$, 
$$\rhomi_{\mloc}(\und{k}, \uh_T)\simeq \ukn(T),$$
\end{theo}

Par la formule (\ref{inddg}), l'isomorphisme du théorème \ref{repk} se relève en un unique isomorphisme (à isomorphisme près) dans $Ho(\Hom_! (\dgcat, Sp))$. 

On considère l'objet $\uh_k$ de $\mloc$. Il est donné par $\uh_k(T)(\spec(A))=\Hom_{\dgcat} (T,A)$ pour tout $T\in \dgcat$ et tout $\spec(A)\in \affk$. L'objet $\uh_k$ a une structure de monoïde dans $\mloc$ car on a, par définition de la structure monoïdale de $\mloc$, un isomorphisme $\uh_k\sm \uh_k \simeq \uh_{k\tel k}$ dans $\mloc$. Pour tout $T\in \dgcattf$, on a aussi un $\uh_k$-module $\uh_T$ avec l'action de $\uh_k$ définie par l'isomorphisme naturel $\uh_T\sm \uh_k \simeq \uh_{T\te_k k} \simeq \uh_T$. Résumons la situation en listant les objets, nous avons : 
\begin{itemize}
\item Une catégorie de modèles monoïdale $\mcal{V}=\spafk$. 
\item Deux catégories de modèles monoïdales $\spafk$-enrichis $\mcal{M}=\mloc$ et $\mcal{N}=Pr(\affk, \hlamb)$. Les unités sont notés $\unit_{\mcal{M}}$ et $\unit_{\mcal{N}}$. Ce sont des objets cofibrants de $\mcal{M}$ et $\mcal{N}$ respectivement. On a par conséquent deux foncteurs monoïdaux $\mcal{V}\lmo \mcal{M}$ et $\mcal{V}\lmo \mcal{N}$ donnés par le produit avec l'unité. Leur adjoints à droite $\homi(\unit, -)$ (Hom $\mcal{V}$-enrichi) sont lax monoïdaux et envoie donc monoïdes sur monoïdes. 
\item Un foncteur de Quillen à gauche $\mcal{V}$-enrichi lax monoïdal $F : \mcal{M}\lmo \mcal{N}$ donné par $F=\und{\hh}$. (Dans notre cas $F(\unit_{\mcal{M}})$ est isomorphe à $\unit_{\mcal{N}}$).
\item Un monoïde (cofibrant) $a=\uh_k$ dans $\mcal{M}$ et un $a$-module (cofibrant) $m$ donné par $m=\uh_T$. 
\item Deux morphismes dans $\mcal{V}$ donnés par la fonctorialité de $\lef F$ : 
$$t:\rhomi_{\mcal{M}} (\unit_{\mcal{M}}, a)\lmo \rhomi_{\mcal{N}}(F(\unit_{\mcal{M}}), F(a)),$$
$$u : \rhomi_{\mcal{M}} (\unit_{\mcal{M}}, m)\lmo \rhomi_{\mcal{N}}(F(\unit_{\mcal{M}}), F(m)),$$
où $\rhomi_{\mcal{M}}$ et $\rhomi_{\mcal{N}}$ sont les homs $Ho(\mcal{V})$-enrichis. 
\end{itemize}

Les résultats généraux d'existence de structure de modèles sur les objets en monoïdes et en modules de \cite{ss-alg} requiert de vérifier l'axiome dit du monoïde qui n'est pas vérifié par la catégorie de modèles $\mloc$. Cependant d'après les résultats de Hovey \cite[Thm.3.3 et 2.1]{hoveymon}, les objets en monoïdes et en modules dans $\mloc$ jouissent de propriétés homotopiques suffisantes pour affirmer ce qui suit. Grossièrement, les monoïdes ne forment pas une catégorie de modèles au sens habituel, mais on peut quand même former la catégorie homotopique des monoïdes. Ceci peut se résumer en disant qu'étant donné un monoïde $M$, il existe un remplacement cofibrant fonctoriel $QM\lmos{\sim} M$ et si $M$ est cofibrant, il existe un remplacement fibrant fonctoriel $M\lmos{\sim} RM$. La catégorie homotopique des monoïdes est alors le quotient de la catégorie des monoïdes cofibrants-fibrants par la relation d'homotopie usuelle. 

\begin{itemize}
\item Les objets $\rhomi_{\mcal{M}} (\unit_{\mcal{M}}, a)$ et $\rhomi_{\mcal{N}}(F(\unit_{\mcal{M}}), F(a))$ sont des monoïdes.  
\item L'objet $\rhomi_{\mcal{M}} (\unit_{\mcal{M}}, m)$ est un $\rhomi_{\mcal{M}} (\unit_{\mcal{M}}, a)$-module. 
\item L'objet $\rhomi_{\mcal{N}}(F(\unit_{\mcal{M}}), F(m))$ est un $\rhomi_{\mcal{N}}(F(\unit_{\mcal{M}}), F(a))$-module et donc un $\rhomi_{\mcal{M}} (\unit_{\mcal{M}}, a)$-module dans $\mcal{V}$ via $t$. 
\item Le morphisme $u$ est un morphisme de $\rhomi_{\mcal{M}} (\unit_{\mcal{M}}, a)$-modules. 
\end{itemize}
En l'appliquant dans notre contexte on obtient, si $T\in \dgcattf$ est une dg-catégorie : 

\begin{itemize}
\item Un préfaisceau d'anneaux en spectres $\ka:=\ukn(\ast)\simeq \rhomi_{\mloc} (\uk, \uh_k)$. 
\item Un préfaisceau en $\ka$-modules $\ukn(T)\simeq \rhomi_{\mloc} (\uk, \uh_T)$.
\item Un préfaisceau en $\ka$-modules $\uhcn(T)=\rhomi (\uk, \uhh(T))\simeq \rhomi (\uhh(\uk), \uhh(\uh_T))$, où le $\rhomi$ est celui de la catégorie de modèles $Pr(\affk, \hlamb)$. 
\item Un morphisme de $\ka$-modules $\ukn(T)\lmo \uhcn(T)$. 
\end{itemize}
On note $\ka-Mod_\s$ la catégorie des objets en $\ka$-modules dans $\spafk$. 

\begin{df}\label{defchernlin}
Soit $T\in \dgcattf$. Le caractère de Chern algébrique associé à $T$ est le morphisme de $Ho(\ka-Mod_\s)$,
$$\ch_T : \ukn(T)\lmo \uhcn(T)$$ 
défini précédemment. En prenant des remplacements fibrants convenables, on montre que cela définit un morphisme dans $Ho(\ka-Mod_\s^{\dgcattf})$, 
$$\ch : \ukn\lmo \uhcn.$$
Par la formule (\ref{inddg}), le morphisme $\ch$ se relève en un unique morphisme (à isomorphisme près) dans $Ho(\ka-Mod_\s^{\dgcat})$. 
\end{df} 

\begin{rema}
Il existe une version additive noté $\mota(k)$ de la catégorie de modèles $\mloc$ où l'on remplace les suites exactes courtes de dg-catégories par les suites exactes courtes \emph{scindées}. La K-théorie connective est alors coreprésentable en tant que foncteur $\mota(k)\lmo \spafk$, et la même construction que précédemment nous donne un caractère de Chern connectif,
$$\chc : \ukc \lmo \uhcn$$
dans $Ho(\tilde{\ka}-Mod_\s^{\dgcat})$, où $\tilde{\ka}$ est le revêtement connectif de $\ka$. Par construction le morphisme composé 
$$\xymatrix{\ukc\ar[r] &  \ukn \ar[r]^-{\ch} & \uhcn}$$ 
est isomorphe à $\chc$ dans $Ho(Sp^{\dgcat})$. 
\end{rema}

\begin{rema}
Par définition le caractère de Chern de la définition \ref{defchernlin} coïncide à isomorphisme près dans $Ho(\spafk^{\dgcat})$ avec le caractère de Chern défini par Tabuada dans \cite[page 4]{tabpro}, et on sait donc qu'il coïncide avec le caractère de Chern classique par \cite[Thm.2.8]{tabpro}. 
\end{rema}

\newpage

\section{Réalisation topologique des préfaisceaux sur les nombres complexes}

La réalisation topologique d'un préfaisceau simplicial, définie disons sur les $\C$-schémas de type fini, est un espace topologique canoniquement associé à ce préfaisceau. C'est en quelque sorte l'analogue de l'espace étalé associé à un faisceau sur un espace topologique. Considérons le foncteur qui à un $\C$-schéma de type fini associe l'espace de ses points complexes munie de la topologie transcendante de $\C^N$. Alors ce foncteur s'étend de manière canonique à tous les préfaisceaux simpliciaux. Ce foncteur est appellé réalisation topologique, mais mérite aussi le nom de réalisation de Betti, étant donné que la cohomologie de Betti d'un schéma sur $\C$ est la cohomologie (singulière) de sa réalisation de Betti. 

Plusieurs définitions du foncteur de réalisation topologique ont déjà été mises au jour. La première référence est peut-être l'article de Simpson \cite{simpsontop}, où il montre que la réalisation topologique envoie les équivalences locales (appelées équivalences d'Illusie) sur des équivalences faibles d'espaces. Dans le papier fondateur de Morel--Voevodsky \cite{mv}, les auteurs définissent une réalisation topologique sur $\C$ au niveau de la catégorie $\ao$-homotopique des faisceaux simpliciaux. Dugger--Isaksen ont montré dans \cite{hyptop} qu'on peut définir un foncteur de réalisation topologique qui est un foncteur de Quillen à gauche pour la structure étale-locale sur les préfaisceaux. Je suivrai leur approche et utiliserai leurs résultats. On commence d'abord par définir le foncteur de réalisation au niveau de la structure globale des préfaisceaux, on rappelle ensuite qu'il s'étend en un foncteur de Quillen sur la structure $\ao$-étale. On énonce ensuite plusieurs propriétés de la réalisation qui seront utiles par la suite.

\subsection{Généralités sur la réalisation topologique}\label{genrel}

On rappelle que $\affc$ désgine la catégorie des $\C$-schémas affines de type fini. On note 
$$sp : \affc \lmo Top$$
le foncteur qui a un $\C$-schéma affine $X$ associe l'espace topologique sous-jacent à son espace analytique complexe associé (voir \cite{sga1} pour la définition du foncteur d'analytification). Donc si $X=\spec(A)$ est un $\C$-schéma affine de type fini avec une présentation $A=\C[X_1,\hdots, X_n] /\mathfrak{a}$, $\mathfrak{a}$ un idéal de l'algèbre de polynômes $\C[X_1,\hdots, X_n]$, alors cette présentation définit une immersion fermée $i:X\hookrightarrow \A^n$ et $sp(X)=i(X)(\C)$ avec la topologie induite par la topologie transcendante de $\A^n(\C)=\C^n$. Ce foncteur de réalisation est en fait défini pour tous les schémas de type fini non nécessairement affines, mais comme nous allons étendre aux préfaisceaux simpliciaux, cela n'a pas d'importance. Cela est plus commode pour nous de considérer les schémas affines, pour pouvoir au besoin tensoriser ces objets avec une dg-catégorie (voir \ref{notapref}). En composant le foncteur $sp$ avec le foncteur complexe singulier $S : Top\lmo SSet$, on obtient une réalisation topologique notée $ssp$, à valeurs dans les ensembles simpliciaux,
$$\xymatrix{ \affc \ar[r]^-{sp} \ar[dr]_-{ssp} & Top \ar[d]^-{S} \\ & SSet }$$

\begin{df} 
On note encore $ssp$ (resp. $sp$) et on appelle \emph{réalisation topologique} l'extension de Kan à gauche enrichie dans $SSet$ du foncteur $ssp : \affc\lmo SSet$ (resp. du foncteur $sp$) le long du plongement de Yoneda $h:\affc\hookrightarrow SPr(\affc)$. On obtient alors des foncteurs 
$$\xymatrix{ \affc \ar@{^{(}->}[r]^-h & SPr(\affc)\ar[r]^-{sp} \ar[dr]_-{ssp} & Top \ar[d]^-{S} \\ & & SSet  }$$
\end{df} 

On rappelle rapidement la notion d'extension de Kan enrichie. Soient $\mcal{V}$ une catégorie monoïdale, $F:C\lmo D$ un foncteur, avec $C$ quelconque et $D$ une catégorie $\mcal{V}$-enrichie, cocomplète au sens enrichi, $G:C\lmo C'$ un foncteur avec $C'$ une catégorie $\mcal{V}$-enrichie cocomplète au sens enrichi. L'extension de Kan à gauche $\mcal{V}$-enrichie de $F$ le long de $G$ est l'unique foncteur $\mcal{V}$-enrichi $\tilde{F} : C'\lmo D$  (défini à isomorphisme près) qui commute aux colimites $\mcal{V}$-enrichies et tel que le triangle 
$$\xymatrix{ C\ar[r]^-G \ar[dr]_-F & C' \ar[d]^-{\tilde{F}} \\ & D }$$
commute à isomorphisme près. Dans le cas qui nous intéresse, nous n'avons pas besoin de considérer des colimites enrichies générales, les colimites du type (\ref{colimenr}) de la remarque \ref{remcolim} plus bas suffisent. 

Les propriétés générales du foncteur $ssp$ sont dans la proposition suivante. On en déduit immédiatement les propriétés analogues pour le foncteur $sp$.

\begin{prop}\label{proprel}
\begin{enumerate}
\item Le foncteur $ssp : SPr(\affc) \lmo SSet$ commute aux colimites et possède comme adjoint à droite le foncteur qui a un ensemble simplicial $K$ associe le préfaisceau $R(K) : X\mapsto \map (ssp(X), K)$. La paire de foncteurs adjoints $(ssp, R)$ est une paire de Quillen pour la structure globale sur $SPr(\affc)$. 
\item Le foncteur $ssp$ commute aux produits homotopiques finis de préfaisceaux simpliciaux. En particulier c'est un foncteur monoïdal pour la structure monoïdale cartésienne sur $\spr$ et la structure monoïdale cartésienne sur $SSet$. 
\item Pour tout $K\in SSet$, et tout $E\in SPr(\affc)$, on a un isomorphisme canonique $ssp(K\times E)\simeq K\times ssp(E)$. Pour tout $K\in SSet_*$ et tout $E\in SPr(\affc)_*$ on a un isomorphisme canonique $ssp(K\sm E)\simeq K\sm ssp(E)$.
\item Soit $E\in \spr$, on note $\pi_0^{pr} E$ le préfaisceau d'ensembles obtenue en appliquant le $\pi_0$ niveau par niveau. Alors il existe un isomorphisme d'ensembles $\pi_0 ssp(E)\simeq \pi_0 ssp(\pi_0^{pr} E)$. 
\end{enumerate}
\end{prop}

\begin{proof} 
\begin{enumerate}
\item Le fait que $R$ est l'adjoint à droite vient de la définition de $ssp$. De par sa formule, le foncteur $R$ envoie fibrations triviales et fibrations sur fibrations triviales et fibrations respectivement. La paire $(ssp, R)$ est donc de Quillen. On peut aussi citer le résultat général \cite[Prop 1.1]{duguni}, duquel on déduit directement que $(ssp, R)$ est de Quillen. 
\item Comme tout préfaisceau simplicial est colimite de préfaisceaux de la forme $K\times h_X$ avec $h_X$ le préfaisceau représenté par le schéma $X$ et $K$ un ensemble simplicial (voir la remarque \ref{remcolim}), et que dans un topos les colimites commutent aux produits finis, il suffit (en utilisant le point 3) de le vérifier pour les préfaisceaux représentables. Cela découle alors du fait que le foncteur d'analytification commute aux produits finis (\cite[exposé XII]{sga1}) et que l'espace sous-jacent à un produit de deux espaces analytiques est canoniquement isomorphe au produit des espaces sous-jacents.
\item La preuve du cas pointé est semblable au cas non-pointé. On a des isomorphismes
\begin{align*}
\Hom_{SSet}(\re{K\times E}, K') & \simeq \Hom_{SPr(\C)}(K\times E, R(K')) \\
& \simeq  \Hom_{\spr}(E, R(K')^K) \\
& \simeq \Hom_{\spr}(E, R(K'^K)) \\
& \simeq \Hom_{SSet}(\re{E}, K'^K)\\
& \simeq \Hom_{SSet}(K\times \re{E}, K').
\end{align*}
Et donc, par le lemme de Yoneda, on en déduit $\re{K\times E}\simeq K\times \re{E}$. 
\item Soit $A\in Set$ un ensemble, on a des isomorphismes
\begin{align*}
\Hom_{Set}(\pi_0 ssp(\pi_0^{pr} E), A) & \simeq \Hom_{SSet}(ssp(\pi_0^{pr} E), A) \\
& \simeq  \Hom_{\spr}(\pi_0^{pr} E, R(A)) \\
& \simeq \Hom_{\spr}(E, R(A)) \\
& \simeq \Hom_{SSet}(ssp(E), A)\\
& \simeq \Hom_{Set}(\pi_0 ssp(E), A).
\end{align*}
On en déduit par Yoneda l'isomorphisme cherché.  
\end{enumerate}
\end{proof}

\begin{rema}\label{remcolim}
Tout préfaisceau en ensembles $F\in Pr(\affc)$ s'écrit comme colimites de préfaisceaux représentables. Plus précisément, le morphisme
\begin{equation}\label{yo}
\xymatrix{colim_{X\in \affc/F} h_X \ar[r]^-\sim & F}
\end{equation}
est un isomorphisme de faisceaux. Donc vu que le foncteur $ssp$ commute aux colimites, on en déduit la formule
\begin{equation}\label{formssp}
\xymatrix{ssp(F)\simeq colim_{X\in \affc/F} ssp(X)}.
\end{equation}
Nous utiliserons en \ref{pi0rel} la version homotopique de ce fait. Maintenant supposons que $E\in \spr$ est un préfaisceau simplicial. La formule (\ref{yo}) n'est plus valable. Ceci peut s'expliquer par le fait que l'objet $E$ est de nature homotopique et donc l'objet $\affc/E$ doit vraiment être considérer comme une $\infty$-catégorie pour que la formule prenne un sens et non pas comme une $1$-catégorie. Cet objectif est réalisé en choisissant une théorie des $\infty$-catégories, et en montrant que la formule (\ref{yo}) a un sens pourvu que la colimite soit comprise en un sens supérieur convenable. 

Une alternative est d'écrire $E$ comme une colimite un peu plus compliquée (voir \cite[page 8]{duguni}). En effet $E$ peut s'écrire comme le coégalisateur standard
\begin{equation}\label{colimenr}
\xymatrix{ E\simeq colim ( \coprod_{X\in \affc} E(X)\times h_X  & \coprod_{Y\lmo Z} E(Z)\times h_Y ) \ar@<3pt>[l] \ar@<-3pt>[l] }.
\end{equation}
Et donc
$$\xymatrix{ ssp(E)\simeq colim ( \coprod_{X\in \affc} E(X)\times ssp(X)  & \coprod_{Y\lmo Z} E(Z)\times ssp(Y) ) \ar@<3pt>[l] \ar@<-3pt>[l] }.$$
\end{rema}

\subsection{Réalisation topologique de la théorie $\ao$-homotopique des schémas}

On rappelle rapidement les résultats de Jardine (voir \cite{jardsimp}, \cite{dhi}) sur les structures de modèles locales sur les préfaisceaux simpliciaux sur un site de Grothendieck. Si $C$ est un tel site, dont nous notons $\tau$ sa topologie, il existe une structure de modèle dite structure de modèles $\tau$-locale sur $SPr(C)$ qui contrairement à la structure de modèles projective prend en compte la topologie $\tau$. Elle peut se définir comme la localisation de Bousfield à gauche de la structure de modèles projective $SPr(C)$ par rapport à l'ensemble des flèches de la forme 
$$hocolim_{\del^{op}} h_{U_\bul} \lmo h_X$$
dans $SPr(C)$, pour tout $X\in C$ et tout $\tau$-hyperrecouvrement $U_\bul\lmo X$ dans $C$. Cela revient à forcer les préfaisceaux simpliciaux fibrant sur $C$ à vérifier une condition de type descente homotopique par rapport à la topologie $\tau$. On note $SPr(C)^\tau$ la structure de modèles $\tau$-locale. On l'utilisera dans les cas particuliers où $C=\affc$ avec la topologie étale ou autre. Dans ce cas on l'appellera la structure étale locale. 

On munit maintenant $\affc$ d'une topologie de Grothendieck. Disons la topologie étale pour fixer les idées. Tout ce qu'on va énoncer dans cette partie reste vraie pour la topologie Nisnevich. En suivant \cite{mv}, on peut construire à partir de $\affc$ la théorie $\ao$-homotopique étale des schémas. L'idée conductrice (voir \cite{duguni}) est de mimer la situation topologique : la catégorie homotopique des espaces peut s'obtenir en localisant celle des variétés topologiques en forçant la droite réelle, par exemple, à être contractile. De la même manière, on force les préfaisceaux de la catégorie de modèles $SPr(\affc)$ à vérifier la descente étale et on force la droite affine $\ao$ à être contractile. Soit $\mcal{R}$ l'ensemble des morphismes de $SPr(\affc)$ de la forme : 
\begin{enumerate} 
\item $hocolim_{\del^{op}}h_{U_\bul} \lmo h_X$, pour $U_\bul\lmo X$ un hyperrecouvrement étale d'un schéma $X\in \affc$. 
\item une projection $h_X\times h_{\ao}\lmo h_X$, pour $X\in \affc$. 
\end{enumerate}

Un modèle pour la catégorie $\ao$-homotopique étale des schémas est donné par la localisation de Bousfield à gauche 
$$L_{\mcal{R}} SPr(\affc)=:\spretao.$$
On appellera cette structure de modèles la \emph{structure $\ao$-étale} sur les préfaisceaux. Les équivalences dans $\spretao$ sont appelées les $\ao$-équivalences. On note 
$$sph : \spretao\lmo Top$$
le foncteur de réalisation correspondant, et $ssph$ son analogue simplicial. On note $Rh$ l'adjoint à droite de $ssph$. Une propriété centrale de la réalisation topologique est le résultat suivant. 

\begin{theo}[Dugger--Isaksen \cite{hyptop} Thm. 5.2 ]\label{di} La paire $(sph, Rh)$ est une paire de Quillen
$$\xymatrix{ \spretao \ar@<3pt>[r]^-{sph}  & Top \ar@<3pt>[l]^-{Rh} },$$
pour la structure de modèles $\ao$-étale. Plus précisément, le foncteur $sp$ envoie les relations de type (1) et de type (2) sur des équivalences d'espaces. 
\end{theo}

\begin{rema}\label{remderrel}
\begin{itemize}
\item Il s'en suit que le foncteur de réalisation à valeurs dans les ensembles simpliciaux $ssph : \spretao\lmo SSet$ est aussi un foncteur de Quillen à gauche et conserve les relations de type (1) et (2) qui définissent $\spretao$. 
\item Une autre conséquence de \ref{di} est que les foncteurs dérivés 
$$\lef ssp:Ho(SPr(\affc)) \lmo Ho(SSet)$$
$$\textrm{et  }\lef ssph : Ho(\spretao)\lmo Ho(SSet)$$
vérifient que $\lef ssp(E)\simeq \lef ssph(E)$ pour tout préfaisceau simplicial $E$. En d'autres termes, le défaut pour le foncteur $ssp : \spretao\lmo SSet$ de préserver les $\ao$-équivalences vient de la condition d'être cofibrant seulement dans la catégorie de modèles globale $SPr(\affc)$. Donc si $Q$ désigne un foncteur de remplacement cofibrant dans $SPr(\affc)$, alors $\lef ssph (E)=ssp(QE)$. 
\item La preuve du Théorème \ref{di} repose sur le fait non-trivial qu'étant donné un hyperrecouvrement topologique ouvert $U_\bul \lmo X$ d'un espace, le morphisme $hocolim_{\del^{op}} U_\bul \lmo X$ est une équivalence d'espaces. La preuve de ce fait consiste à se ramener à un hyperrecouvrement borné et à effectuer une récurrence sur la dimension de l'hyperrecouvrement. Nous mimons cette preuve en \ref{sectrestliss} pour fournir une preuve d'un fait analogue pour les hyperrecouvrements propres d'espaces suffisamment gentils. 
\end{itemize}
\end{rema}

\begin{df} On appelle \emph{foncteur de réalisation topologique dérivée} et on note avec la même notation (vu la remarque \ref{remderrel}) $\re{-}$ les deux foncteurs $\lef ssp$ et $\lef ssph$. On a alors un triangle commutatif,
$$\xymatrix{ Ho(\spr) \ar[rd]_-{\re{-}} \ar[rr]^-{\R id} && Ho(\spretao) \ar[ld]^-{\re{-}} \\ & Ho(SSet) &  }$$
On notera $\ret{-}$ son analogue à valeurs dans $Ho(Top)$. 
\end{df} 

\begin{prop}\label{proprelder}
\begin{enumerate}
\item Le foncteur de réalisation topologique dérivée $\re{-}$ commute aux colimites homotopiques. Son adjoint à droite est noté $(-)_B$. Pour tout $K\in SSet$ et tout $X\in\affc$, on a $K_B(X)=\R\map(ssp(X), K)$. 
\item Le foncteur $\re{-}$ commute aux produits finis homotopiques. 
\item Pour tout $K\in SSet$, et tout $E\in \spretao$, on a un isomorphisme canonique $K\times \re{E}\simeq \re{K\times E}$ dans $Ho(SSet)$. Pour tout $K\in SSet_*$ et tout $E\in \spretao$ pointé, on a un isomorphisme canonique $K\sm \re{E}\simeq \re{K\sm E}$ dans $Ho(SSet_*)$. 
\end{enumerate}
\end{prop} 

\begin{proof} 
Les preuves sont essentiellement les mêmes que pour la réalisation non dérivée. Donc nous les omettons. 
\end{proof}

\subsection{Le $\pi_0$ de la réalisation}\label{pi0rel}

On donne ici une description relativement explicite de l'ensemble $\pi_0\re{E}$ pour un préfaisceau simplicial quelconque $E\in \spr$, qui sera utilisée plus bas. Cette description est basée sur la formule (\ref{yo}) donné par Yoneda.

\begin{lem}\label{pi0relens}
Soit $F\in Pr(\affc)$ un préfaisceau d'ensembles. Alors il existe un isomorphisme d'ensembles,
$$\pi_0 ssp(F)\simeq F(\C) / \sim $$
où $\sim$ désigne la relation d'équivalence définie par $[x]\sim [y]$ si il existe une courbe algébrique connexe $C$, un morphisme $f:C\lmo F$ dans $Ho(\spr)$, et deux points complexes $x',y'\in C(\C)$ tels que $f(x')=x$ et $f(y')=y$. 
\end{lem}

\begin{proof} Comme il existe un isomorphisme $\pi_0 ssp(F)\simeq \pi_0 sp(F)$, où on rappelle que $sp$ désigne la réalisation à valeurs dans $Top$, il suffit de le prouver pour $\pi_0 sp(F)$. D'après la formule (\ref{yo}) de la remarque \ref{remcolim}, il existe un morphisme naturel $F(\C)\lmo sp(F)$ dans $Top$ et on obtient un morphisme $F(\C)\lmo \pi_0sp(F)$ en prenant le $\pi_0$. On a alors un diagramme commutatif d'ensembles
$$\xymatrix{ F(\C) \ar[r] \ar[d]_-\wr & \pi_0 sp(F) \ar[d]^-\wr \\ colim_{X\in \affc/F} X(\C) \ar[r] &  colim_{X\in \affc/F} \pi_0 sp(X) }$$
où la flèche verticale gauche s'obtient en prenant les points complexes dans la formule (\ref{yo}), la flèche verticale droite s'obtient en appliquant le foncteur $\pi_0 sp$ à la formule (\ref{yo}). Ces deux flèches sont donc des isomorphismes. Pour tout $X\in \affc$, l'application $X(\C)\lmo \pi_0sp(X)$ est surjective, on en déduit donc que le morphisme horizontal du haut est surjectif. Il reste donc à analyser quand est-ce que deux points complexes $a,b\in F(\C)$ ont même image par $F(\C)\lmo  \pi_0 sp(F)$. En passant à la flèche du bas, ces deux points $a$ et $b$ correspondent à deux schémas affines $(X,x)$ et $(Y,y)$ au dessus de $F$, chacun muni d'un point complexe. Par hypothèse ces deux éléments ont même image dans $colim_{X\in \affc/F} \pi_0sp(X)$. Donc il existe un schéma affine $Z\lmo F$ au dessus de $F$, deux morphismes $p:X\lmo Z$ et $q:Y\lmo Z$, et un chemin continu qui joint $p(x)$ et $q(y)$ dans l'espace $sp(Z)$. C'est un fait bien connu en géométrie algébrique que dans cette situation, en utilisant le théorème de Bertini, on peut en déduire qu'il existe une courbe algébrique connexe $g:C\lmo Z$ au dessus de $Z$, deux points complexes $x',y'\in C(\C)$ tel que $g(x')=p(x)$ et $g(y')=q(y)$. Par composition on obtient un morphisme $f:C\lmo F$ tel que $f(x')=a$ et $f(y')=b$. Ceci achève la preuve.
\end{proof}

\begin{prop}\label{pi0cor}
Soit $E\in \spr$ un préfaisceau simplicial. Alors il existe un isomorphisme d'ensembles, 
$$\pi_0 \re{E} \simeq \pi_0(E(\C))/\sim$$
où $\sim$ désigne la relation d'équivalence définie par $[x]\sim [y]$ si il existe une courbe algébrique connexe $C$, un morphisme $f:C\lmo E$ dans $Ho(\spr)$, et deux points complexes $x',y'\in C(\C)$ tels que $f(x')=x$ et $f(y')=y$. 
\end{prop}

\begin{proof} 
On note $\pi_0^{pr}E$ le préfaisceau obtenue à partir de $E$ en appliquant le $\pi_0$ niveau par niveau. On applique le lemme \ref{pi0relens} au préfaisceau en ensembles $F=\pi_0^{pr}E$. On obtient un isomorphisme $\pi_0 ssp(\pi_0^{pr} E) \simeq \pi_0^{pr} E(\C) /\sim \simeq (\pi_0 E(\C))/\sim$. D'après \ref{proprel}, on a un isomorphisme $\pi_0 ssp(\pi_0^{pr} E)\simeq \pi_0 ssp(E)$. D'autre part, d'après la remarque \ref{remderrel}, on a un isomorphisme $\pi_0(ssp(E))\simeq \pi_0\re{E}$. Ce qui montre la formule. 
\end{proof}

\subsection{Réalisation de préfaisceaux structurés}\label{prestru}

\paragraph{Groupes stricts.}

Soit $SGp_\C$ (resp. $SGp$) la sous-catégorie de $\spr$ (resp. de $SSet$) formée des objets en groupes (au sens strict, par opposition aux $\del$-objets de la section \ref{monoides}). On note 
$$\mrm{B}:SGp\lmo SSet$$
le foncteur "espace classifiant". Pour tout $G\in SGp$, l'espace $\mrm{B}G$ est définie comme la colimite homotopique 
$$\mrm{B}G:=hocolim\, G^\bul$$
où $G^\bul$ est l'objet simplicial de $SSet$ définie par $[n]\mapsto G^n$, avec les faces et dégénérescences obtenues en utilisant le produit et le neutre de $G$. Il existe un remplacement cofibrant canonique de l'objet simplicial $G^\bul$, ce qui justifie que l'on considère $\mrm{B}$ comme un foncteur sur les catégories non localisées. Pour $G\in SGp_\C$, on pose $\mrm{B}G(\spec(A))=\mrm{B}(G(\spec(A))$. Parce que $\re{-}$ commute aux produits, la réalisation d'un préfaisceau de groupes simpliciaux est un groupe simplicial. On a alors un diagramme
$$\xymatrix{Ho(SGp_\C) \ar[r]^-{\mrm{B}} \ar[d]^-{\re{-}} & Ho(\spr) \ar[d]^-{\re{-}} \\ Ho(SGp) \ar[r]^-{\mrm{B}} & Ho(SSet)  }$$

\begin{prop}\label{gp}
Pour tout $G\in SGp_\C$, il existe un isomorphisme canonique $\mrm{B}\re{G}\simeq \re{\mrm{B}G}$ dans $Ho(SSet)$. 
\end{prop}

\begin{proof}
Cela découle directement du fait que $\re{-}$ commute aux produits finis homotopiques et aux colimites homotopiques. 
\end{proof}

\paragraph{Spectres.}\label{realsp}

On étend le foncteur de réalisation $ssp :\spr \lmo SSet$ aux préfaisceaux en spectres symétriques en prenant simplement l'extension de Kan à gauche enrichi dans $Sp$ du foncteur composé
$$\xymatrix{ \spr\ar[r]^-{ssp} & SSet\ar[r]^-{\sinf(-)_+} & Sp }$$
le long du foncteur de suspension infinie $\sinf(-)_+ : \spr\lmo \spaf$. On note alors $\sspsp$ le foncteur obtenue. On obtient donc un diagramme commutatif
$$\xymatrix{ \spr \ar[r]^-{ssp} \ar[d]_-{\sinf(-)_+} & SSet\ar[r]^-{\sinf(-)_+} & Sp  \\ \spaf \ar@/_1pc/[rru]_-{\sspsp} }.$$
Son adjoint à droite envoie un spectre $E$ sur le préfaisceau $X\mapsto \homi_{Sp}(\sinf X_+, E)$. On montre de même que pour la réalisation des préfaisceaux simpliciaux les propriétés de commutation suivantes. 
\begin{prop}\label{proprelsp}
\begin{itemize}
\item Pour tout $F\in Sp$ et tout $E\in \spaf$ on a un isomorphisme canonique de spectres $F\sm \sspsp(E)\simeq \sspsp(F\sm E)$.
\item Pour tous $E,E'\in \spaf$, on a un isomorphisme canonique de spectres $\sspsp(E\sm E') \simeq \sspsp(E)\sm \sspsp(E')$, c'est à dire que le foncteur $\sspsp$ est monoïdal. 
\end{itemize}
\end{prop}
La première est une conséquence de la deuxième, mais comme en \ref{proprel}, on l'utilise pour montrer la deuxième. 

On note $\spetao$ la structure de modèle $\ao$-étale sur la catégorie $\spaf$ des préfaisceaux en spectres symétriques sur $\affc$. Une conséquence de la proposition \ref{di} est que le foncteur associé 
$$ssph_\s:\spetao\lmo Sp$$ 
est de Quillen à gauche pour la structure de modèles $\ao$-étale. La remarque \ref{remderrel}  s'applique aussi pour les préfaisceaux en spectres : pour tout $E\in \spaf$ on a un isomorphisme canonique $\lef \sspsp (E)\simeq \lef ssph_\s (E)$. 
\begin{df} On appelle \emph{réalisation topologique spectrale} et on note $\resp{-}$ les foncteur dérivés $\lef \sspsp$ et $\lef ssph_\s$. On a donc un triangle commutatif,
$$\xymatrix{ Ho(\spaf) \ar[rd]_-{\resp{-}} \ar[rr]^-{\R id} && Ho(\spetao) \ar[ld]^-{\resp{-}} \\ & Ho(Sp) & }$$
On note $\hbs$ l'adjoint à droite de $\resp{-}$. Pour tout $E\in Ho(Sp)$, $\hbs(E)$ est le préfaisceau 
$$\hbs(E) : X\longmapsto \rhomi_{Ho(Sp)} (\resp{X}, E).$$
\end{df} 

La notation $\hbs$ est pour rappeler la cohomologie de Betti puisque si $E=H\C$ est le spectre d'Eilenberg--Mac-Lane associé à $\C$, alors le préfaisceau $\hbs(H\C)$ n'est rien d'autre que la cohomologie de Betti. Comme $\sspsp$ est un foncteur monoïdal on en déduit que le foncteur de réalisation topologique spectrale $\resp{-}$ est un foncteur monoïdal. 

\paragraph{$\del$ et $\gam$-espaces.}

On étend, d'autre part, le foncteur de réalisation aux $\del$-objets dans $\spr$. Bien entendu, la même remarque peut être formulée pour les $\gam$-objets. On applique les notations de \ref{strucdel} pour les structures de modèles sur les catégories de $\del$-objets. Soit $E\in \del-\spr$ un $\del$-préfaisceau. La réalisation topologique de $E$, notée $ssp_{\del}(E)$ est le $\del$-espace défini par $ssp_{\del}(E)_n=ssp(E_n)$. Ceci définit un foncteur 
$$ssp_{\del} : \del-\spr\lmo \del-SSet.$$
Ce foncteur a pour adjoint à droite le foncteur qui a un $\del$-espace $F$ associe le $\del$-préfaisceau définie par $R(F)_n=R(F_n)$ (où $R$ est l'adjoint à droite du foncteur $ssp:\spr\lmo SSet$). Le foncteur associé 
$$ssph_{\del}:\del-\spretao \lmo \del-SSet$$ 
est de Quillen à gauche pour la structure $\ao$-étale. 

\begin{df} 
On appelle \emph{réalisation des $\del$-préfaisceaux} et on note $\redel{-}$ les foncteurs dérivés $\lef ssp_{\del}$ et $\lef ssph_{\del}$. On a donc un triangle commutatif, 
$$\xymatrix{ Ho(\del-\spr) \ar[rd]_-{\redel{-}} \ar[rr]^-{\R id} && Ho(\del-\spretao) \ar[ld]^-{\redel{-}} \\ & Ho(\del-SSet) & }$$
On note $\regam{-}$ son $\gam$-analogue. 
\end{df} 
 
La réalisation $\redel{-}$ commute aux produits homotopiques finis, donc la réalisation d'un $\del$-objet spécial (resp. très spécial) est un $\del$-objet spécial (resp. très spécial). On a donc un diagramme 
$$\xymatrix{ Ho(\del-\spr) \ar[r]^-{\redel{-}}  \ar[d]^-{mon} & Ho(\del -SSet)\ar[d]^-{mon}\\ 
Ho(\del-\spr^{sp}) \ar[r]^-{\redel{-}}  \ar[d]^-{(-)^+} & Ho(\del -SSet^{sp})   \ar[d]^-{(-)^+} \\
Ho(\del-\spr^{tsp} )  \ar[r]^-{\redel{-}} & Ho(\del -SSet^{tsp}) }$$

\begin{prop}\label{proprelplus}
Pour tout $E\in Ho(\del-\spr)$ on a un isomorphisme canonique dans $Ho(\del-SSet)$, 
$$\redel{mon(E)} \simeq mon \redel{E}.$$
\end{prop} 

\begin{proof} 
Il suffit de montrer que les adjoints à droite commute. Mais cela vient du fait que si $F\in \del-SSet$ est spécial alors le $\del$-objet $\R R(F)=\R\map(\re{-}, F)$ est spécial. 
\end{proof}

\begin{prop}\label{proprelplus}
Pour tout $\del$-préfaisceau \emph{spécial} $E\in Ho(\del-\spr)$ on a un isomorphisme canonique dans $Ho(\del-SSet)$, 
$$\redel{E^+} \simeq \redel{E}^+.$$
\end{prop} 

\begin{proof} 
Il suffit de montrer que les adjoints à droites commutent. Mais cela vient du fait que si $F\in \del-SSet$ est très spécial alors le $\del$-préfaisceau $\R R(F)=\R\map(\re{-}, F)$ est très spécial. 
\end{proof}

Comparons maintenant la réalisation des $\gam$-préfaisceaux et des préfaisceaux en spectres (connectifs). On a un diagramme,

$$\xymatrix{Ho(\gam-\spr) \ar[r]^-{\regam{-}} \ar[d]^-{\mcal{B}} & Ho(\gam-SSet) \ar[d]^-{\mcal{B}}  \\ Ho(\spaf) \ar[r]^-{\resp{-}} & Ho(Sp) }$$

\begin{prop}\label{proprelb}
Pour tout $E\in Ho(\gam-\spr)$, on a un isomorphisme canonique dans $Ho(Sp)$, 
$$\mcal{B} \regam{E} \simeq \resp{\mcal{B}E}.$$
\end{prop} 

\begin{proof} 
Cela vient du fait que le foncteur $\mcal{B}$ peut de par sa définition être exprimé comme une colimite homotopique, et de ce que le foncteur $\resp{-}$ commute à de telles colimites. 
\end{proof}

On rappelle (cf. \ref{strucdel}) l'existence d'un foncteur 
$$\alpha^*: Ho(\gam-\spr) \lmo Ho(\del-\spr).$$
On remarque alors que par définition de $\alpha^*$, pour tout $E\in \gam-\spr$, on a un isomorphisme canonique 
$$\redel{\alpha^* E} \simeq \alpha^*\regam{E}.$$

\subsection{Restriction aux schémas lisses}\label{sectrestliss}

Dans cette partie on montre que l'on ne change pas la réalisation topologique d'un préfaisceau simplicial en le restreignant aux schémas lisses. C'est une conséquence d'un résultat de Suslin--Voevodsky concernant la topologie cdh (\cite{sv1, sv2}) : cette topologie donne lieu à des catégories de faisceaux équivalentes suivant qu'on démarre avec tous les schémas ou avec les schémas lisses. Ce fait sera utiliser dans la suite pour comparer la définition de la K-théorie topologique d'une variété lisse "au sens dg" avec la définition classique, ainsi que pour montrer que la K-théorie semi-topologique négative ("au sens dg") d'une algèbre commutative lisse est triviale. 

On rappelle qu'on note $\affc$ la catégorie des $\C$-schémas affines de type fini et $\afflissc$ la catégorie des $\C$-schémas affines lisses de type fini. On note $l:\afflissc\hookrightarrow \affc$ le foncteur d'inclusion. La restriction sur les préfaisceaux simpliciaux est notée $l^*:\spr\lmo \sprliss$. Le théorème qu'on établit ici est alors le suivant. 

\begin{theo}\label{restliss}
Soit $F\in SPr(\affc)$ un préfaisceau simplicial. Alors il existe un isomorphisme naturel $\re{l^* F} \simeq \re{F}$ dans $Ho(SSet)$. 
\end{theo}

La preuve occupe les deux sous-parties suivantes. Elle est basée sur la topologie propre sur les schémas. L'idée est d'utiliser l'existence d'une équivalence de Quillen entre la catégorie de modèles $\sprpro$ des préfaisceaux simpliciaux avec la structure propre locale et la catégorie de modèles $\sprproliss$ des préfaisceaux simpliciaux sur les lisses avec le même type de structure de modèles. Nous montrons aussi que le foncteur de réalisation topologique se comporte bien par rapport à la topologie propre, plus précisément qu'il reste de Quillen à gauche pour la structure propre locale. Pour cela nous sommes amener à montrer un résultat purement topologique : l'analogue du théorème de Dugger--Isaksen \cite[Thm. 1.2]{hyptop} pour les hyperrecouvrement propres d'espaces. 

Comme nous allons utiliser l'existence d'une résolution des singularités sur $\C$ pour obtenir des recouvrements propres par des schémas lisses, il est plus commode de travailler avec des préfaisceaux définis sur tous les schémas non-nécessairement affines. On note $\schc$ la catégorie des $\C$-schémas de type fini et $\lissc$ la catégorie des $\C$-schémas lisses de type fini. On note
$$\xymatrix{ \afflissc\ar@{^{(}->}[r]^-{l} \ar@{^{(}->}[d]_k & \affc \ar@{^{(}->}[d]^-j \\ \lissc\ar@{^{(}->}[r]^-i & \schc }$$
les inclusions. On a des foncteurs de restriction/extension entre catégories de préfaisceaux simpliciaux munies de la structure étale locale,
$$\xymatrix{ \sprset \ar@<-3pt>[r]_-{i^*} \ar@<3pt>[d]^-{j^*} & \sprslisset \ar@<-3pt>[l]_-{i_!}   \ar@<3pt>[d]^-{k^*} \\ \spret \ar@<-3pt>[r]_-{l^*} \ar@<3pt>[u]^-{j_!}  & \sprlisset \ar@<3pt>[u]^-{k_!} \ar@<-3pt>[l]_-{l_!} }$$
Toutes ces paires de foncteurs adjoints sont de Quillen par le résultat général sur les changements de sites \cite[Prop. 7.2]{dhi}. Le foncteur de restriction $j^*$ conserve les équivalences étale locales et c'est une équivalence de Quillen. En effet cela peut être déduit du fait analogue pour les faisceaux en utilisant la structure de modèle projective locale sur les faisceaux simpliciaux de \cite[Thm.2.1]{blande}. Le théorème \ref{restliss} peut donc s'énoncer de manière équivalente avec la catégorie $\sprs$. En effet pour tout $F\in \spr$, on a une équivalence $\re{\lef j_! F} \simeq \re{F}$ par définition de la réalisation et car $\lef j_!$ est un adjoint à gauche. D'autre part si $E\in \sprs$, le morphisme de counité $\lef j_! j^* E\lmo E$ est une équivalence étale locale. Le morphisme induit $\re{\lef j_! j^* E} \lmo \re{E}$ est donc une équivalence par le théorème \ref{di}. D'autre part on a $\re{\lef j_! j^* E}\simeq \re{j^*E}$. On a donc une équivalence $\re{j^*E}\simeq \re{E}$. Par conséquent on travaillera avec la catégorie de préfaisceaux $\sprs$. 

On munit donc $\schc$ de la topologie propre, c'est-à-dire qu'une famille est couvrante si chaque morphisme de la famille est un morphisme propre et que cette famille est conjointement surjective. On définit la topologie propre sur la catégorie $\lissc$ en disant qu'un crible $S\subseteq h_X$ d'un objet $X$ de $\lissc$ est couvrant s'il contient un crible engendré par un recouvrement propre de $X$ dans $\schc$. On rappelle que la catégorie $\lissc$ n'admet pas de produits fibrés en général. 

\begin{rema}\label{hiro}
Par le théorème d'Hironaka de résolution des singularités, tout crible couvrant propre $S$ d'un schéma $X$ admet un raffinement par un crible couvrant propre $T\subseteq S$ engendré par des morphismes de source des schémas lisses. En effet pour tout morphisme $Z\lmo Y$ dans $S$, par Hironaka, il existe un schéma lisse $Z'$ et un morphisme propre surjectif $Z'\lmo Z$.
\end{rema}

\begin{nota} 
On note $\sprpro$ la catégorie de modèles propre locale, c'est-à-dire la localisation de Bousfield à gauche de la structure projective $\sprs$ le long des morphismes $$hocolim_{\del^{op}} h_{Y_\bul} \lmo h_X$$ 
pour tout hyperrecouvrement propre $Y_\bul \lmo X$ dans $\schc$. Pour établir une preuve du théorème \ref{restliss}, nous aurons besoin du résultat suivant (l'analogue du théorème \ref{di} pour la topologie propre). 
\end{nota}

\begin{prop}\label{quipro}
Pour tout hyperrecouvrement propre $Y_\bul\lmo X$ d'un schéma $X\in \schc$, le morphisme induit,
$$hocolim_{\del^{op} }  \re{Y_\bul}\lmo \re{X}$$
est un isomorphisme dans $Ho(SSet)$. On en déduit que le foncteur
$$ssp : \sprpro \lmo SSet$$
est de Quillen à gauche. 
\end{prop} 

La preuve de la proposition \ref{quipro} est relativement non triviale et basée sur le fait purement topologique qu'un hyperrecouvrement propre induit une équivalence quand on passe à la colimite homotopique. Elle sera donné en \ref{hyppro}. On utilise la proposition \ref{quipro} dans la partie suivante pour montrer notre résultat de restriction. 
 
 \subsubsection{La propriété de restriction aux lisses}

Le foncteur $i:\lissc\hookrightarrow \schc$ donne lieu à un foncteur continu entre les sites propres associés et à un jeu d'adjonction entre catégories de faisceaux propres d'ensembles, 
$$\xymatrix{\shpro \ar[r]^-{i^\ast} & \shproliss \ar@<7pt>[l]^-{e} \ar@<-10pt>[l]_-{i_!}   }.$$
Le foncteur $i^*$ est la restriction aux lisses. Pour tout faisceau $F\in \shproliss$, le faisceau $eF$ est (a priori) le faisceautisé du préfaisceau $X\longmapsto \Hom(i_\ast h_X, F)$. 

Par \cite[Prop. 7.2]{dhi}, la paire de foncteurs adjoints 
$$\xymatrix{\sprpro \ar@<-3pt>[r]_-{i^\ast} & \sprproliss \ar@<-3pt>[l]_-{i_!}   },$$
est une paire de Quillen.

La preuve de \ref{restliss} est basé sur le résultat suivant qui est une variante de \cite[Thm. 6.2]{sv2}. 

\begin{prop}\label{eqpreliss}
 La paire de Quillen $(i_!, i^*)$ est une équivalence de Quillen et induit donc une équivalence de catégories, 
$$\xymatrix{Ho(\sprpro) \ar@<-3pt>[r]_-{i^*} & Ho(\sprproliss) \ar@<-3pt>[l]_-{\lef i_!}   }.$$
\end{prop}  

\begin{proof} 
On se ramène à montrer que les catégories de faisceaux d'ensembles sont équivalentes. En effet il existe une structure de modèles projective locale sur les faisceaux simpliciaux $SSh(\schc)^{\mrm{pro}}$ (\cite[Thm.2.1]{blande}) tel que la paire de foncteurs de Quillen formée par l'inclusion des faisceaux simpliciaux dans les préfaisceaux simpliciaux et du foncteur de faisceautisation est une équivalence de Quillen,
$$\xymatrix{ SSh(\schc)^{\mrm{pro}} \ar@{^{(}->}@<-3pt>[r] & \sprpro \ar@<-3pt>[l]_{\und{a}}  }$$
d'après \cite[Thm.2.2]{blande}. De plus la catégorie $SSh(\schc)^{\mrm{pro}}$ est la catégorie des objets simpliciaux dans le topos $\shpro$, et sa structure de modèles ne dépend que du topos $\shpro$. Il en est de même pour la catégorie de modèles $SSh(\lissc)^{\mrm{pro}}$. 

A partir de là, il suffit de montrer que la paire de foncteurs adjoints 
$$\xymatrix{\shpro \ar@<3pt>[r]^-{i^\ast} & \shproliss \ar@<3pt>[l]^-{e} }$$
est une équivalence de catégories. Mais ceci est une conséquence directe de \cite[Exposé III, Thm.4.1]{sga4-1}. 

\end{proof}

\begin{lem}\label{preserve}
Le foncteur de restriction $i^* : \sprpro\lmo \sprproliss$ préserve toutes les équivalences locales.
\end{lem} 
\begin{proof} 
Vu le résultat de Blander rappellé au cours de la preuve précédente, il suffit de le montrer pour les catégories de modèles de faisceaux simpliciaux associés. 
Ceci est alors une conséquence du fait montré au cours de la preuve précédente que la restriction $i^*$ induit une équivalence  de catégories $i^*:\shpro\lmos{\sim} \shproliss$. En effet les équivalences locales sont caractérisées comme étant les morphismes qui induisent des isomorphismes sur tous les faisceaux d'homotopie et en tout point base ; de plus, la construction des faisceaux d'homotopie dans une catégorie d'objets simpliciaux dans un topos se réalise de façon purement catégorique et ne dépend que des topos sous-jacents. 
\end{proof}

\begin{proof} de \ref{restliss}  --- Comme le triangle de foncteurs
$$\xymatrix{ \lissc \ar[rr]^-i \ar[dr]_-{ssp} & & \schc \ar[dl]^-{ssp} \\ & SSet }$$
est commutatif, on en déduit que le triangle d'extensions de Kan à gauche
$$\xymatrix{ Ho(SPr(\lissc)) \ar[rr]^-{\lef i_!} \ar[dr]_-{\re{-}} & & Ho(SPr(\schc)) \ar[dl]^-{\re{-}} \\ & Ho(SSet) }$$
est commutatif. Cela veut dire que pour tout $G\in SPr(\lissc)$, on a un isomorphisme $\re{\lef i_! G}\simeq \re{G}$ dans $Ho(SSet)$. Donc si $G=\R i^* F$ pour $F\in SPr(\schc)$, on a des isomorphismes
\begin{align*}
\re{i^* F} & \simeq  \re{\R i^* F} &&  \textrm{(par le Lemme \ref{preserve})} \\
& \simeq  \re{\lef i_! \R i^* F} &&  \textrm{(par ce qui précède)} \\
& \simeq \re{F}  && \textrm{(par la proposition \ref{eqpreliss} et la proposition \ref{quipro})} \\
\end{align*}
\end{proof}

\subsubsection{Hyperdescente propre des espaces topologiques}\label{hyppro}
 
 \begin{proof} de la proposition \ref{quipro} --- La première assertion implique la deuxième sur le fait que $ssp$ est de Quillen à gauche pour la structure propre locale par les résultats généraux sur les localisations de Bousfield à gauche. 
 
 Notons alors $A=\ret{X}$ et $B_\bul=\ret{Y_\bul}$. Le morphisme d'espace simpliciaux $B_\bul \lmo A$ est un hyperrecouvrement propre d'espaces topologiques, c'est-à-dire un hyperrecouvrement pour la topologie sur $Top$ avec pour familles couvrantes les familles conjointement surjectives d'applications continues propres. Par conséquent pour prouver la proposition \ref{quipro} il suffit de prouver que pour tout hyperrecouvrement propre $B_\bul \lmo A$ d'espaces topologiques suffisamment gentils, le morphisme induit $hocolim_{\del^{op}} B_\bul\lmo A$ est une équivalence faible d'espaces. Par 'suffisamment gentils' on entend séparés localement compacts qui ont le type d'homotopie de CW-complexes. Ces hypothèses sont satisfaites par les espaces topologiques qui sont la réalisation de schémas de type fini (voir \cite{hironakatri}). La preuve est alors complétée par la proposition \ref{hypdescpro}. 
\end{proof}

\begin{prop}\label{hypdescpro}
 Soit $B_\bul\lmo A$ un hyperrecouvrement propre d'espaces topologiques avec $(B_n)_{\geq 0}$ et $A$ étant des espaces séparés localement compacts qui ont le type d'homotopie de CW-complexes. Alors le morphisme induit $hocolim_{\del^{op}} B_\bul\lmo A$ est une équivalence dans $Top$. 
\end{prop} 

\begin{proof} 
La preuve est organisée en plusieurs étapes qui reproduisent les étapes de la preuve par Dugger--Isaksen de l'analogue 'ouvert' de cette assertion (c'est-à-dire là où l'on remplace la topologie propre par la topologie ouverte usuelle), voir \cite[Thm. 4.3]{hyptop}. On résume ces étapes avant de donner le corps de la preuve. 
\begin{enumerate}
\item On réduit l'assertion au cas des hyperrecouvrements bornés au sens de \cite[Déf. 4.10]{dhi}. 
\item On se ramène au cas des hyperrecouvrements propres qui sont des 'nerfs' de morphismes propres surjectifs en utilisant le même argument que dans \cite[Lem. 4.2]{hyptop}. 
\item On prouve l'assertion pour le nerf d'un morphisme propre surjectif en utilisant le théorème de changement de bases propre de Lurie \cite[Corollary 7.3.1.18]{htt}, ce qui ramène la preuve au cas du point $A=\ast$. C'est dans cette étape qu'on utilise l'hypothèse faîte sur la nature des espaces. 
\item Un lemme purement simplicial donne le résultat pour le point. 
\end{enumerate}

Nous introduisons d'abord un peu de terminologie et de notations concernant les hyperrecouvrements. Soit $\mcal{C}$ une catégorie complète et cocomplète. Pour tout $[n]\in \del$ et tout objet simplicial $C_\bul$ dans $\mcal{C}$ on note $sk_n C_\bul$ son $n$-squelette et $cosk_n C_\bul$ son $n$-cosquelette. On note $\mcal{C}^{\del^{op}}_{\leq n}$ la catégorie des objets simpliciaux $n$-tronqués dans $\mcal{C}$. On a un foncteur naturel $\mcal{C}^{\del^{op}}\lmo \mcal{C}^{\del^{op}}_{\leq n}$. Le $n$-squelette et le $n$-cosquelette sont les adjoints à gauche et droite respectivement à ce foncteur naturel : 

$$\xymatrix{ \mcal{C}^{\del^{op}}_{\leq n}  \ar@<7pt>[r]^-{sk_n} \ar@<-7pt>[r]_-{cosk_n} & \mcal{C}^{\del^{op}}  \ar[l]  }.$$

On fera l'abus de langage qui consiste, étant donné un objet simplicial $C_\bul$ dans $\mcal{C}$, à considérer les objets $sk_n C_\bul$ et $cosk_n C_\bul$ comme des objets de $\mcal{C}^{\del^{op}}$. Si $\mcal{C}=Top$, on a $(cosk_n C_\bul)_i\simeq \map(sk_n \del^i, C_\bul)$, où $\map$ est l'ensemble simplicial des morphismes dans $Top^{\del^{op}}$. Il existe une version augmentée de ces objets, si $C_\bul\lmo D$ est un objet simplicial augmenté vers un objet simplicial constant $D$, alors on note $sk_n^D C_\bul$ et $cosk_n^D C_\bul$ le foncteur squelette et cosquelette relativement à la catégorie $(\mcal{C}\downarrow D)^{\del^{op}}$ des objets simpliciaux sur $D$. On note
$$M_n C=lim_{(\del^{op} \downarrow n)\setminus id} C_\bul$$
le $n$-ième matching objet de $C_\bul$, où $(\del^{op} \downarrow n)\setminus id$ est la catégorie des morphismes $[n]$ dans $\del^{op}$ privée du morphisme identité de $[n]$. Il y a une version augmentée des matching objets. Si $C_\bul\lmo D$ est un objet simplicial augmenté de $\mcal{C}$ vers un objet simplicial constant $D$, alors on peut calculer la limite en regardant $C$ comme un foncteur de $(\del^{op} \downarrow n)\setminus id$ vers la catégorie $\mcal{C}\downarrow D$ des objets sur $D$ dans $\mcal{C}$. On note $M^D_n C_\bul$ cette version augmentée. On a des morphismes naturels $C_n\lmo M_n C_\bul$ et $C_n\lmo M^D_n C_\bul$. 

Supposons que $\mcal{C}$ est munie d'une topologie de Grothendieck de façon que l'on peut parler d'hyperrecouvrement dans $\mcal{C}$, pour nous ce sera $Top$ avec la topologie propre. On remarque que si $C_\bul\lmo D$ est un hyperrecouvrement de $D$ dans $\mcal{C}$, alors pour tout $n\geq 0$, le morphisme $cosk_n^D C_\bul \lmo D$ est un hyperrecouvrement. Un hyperrecouvrement $C_\bul\lmo D$ est dit \emph{borné} s'il existe un entier $N\geq 0$ tel que les morphismes $C_n\lmo M_n^D C_\bul$ sont des isomorphismes pour tout $n> N$. Le minimum des $N\geq 0$ avec cette propriété est appellé la \emph{dimension} de l'hyperrecouvrement. Un hyperrecouvrement est borné  de dimension $\leq N$ si et seulement si le morphisme unité $C_\bul \lmo cosk^D_N C_\bul$ est un isomorphisme. Si $f:C\lmo D$ est un morphisme dans $\mcal{C}$, on en déduit un morphisme d'objets simpliciaux constant $f:C\lmo D$ dans $\mcal{C}^{\del^{op}}$. L'objet simplicial augmenté $cosk^D_0 C_\bul \lmo D$ est appelé le nerf $f$. On a $(cosk^D_0 C_\bul)_i=C\times_D\cdots \times_D C$, $n+1$ fois. Les faces et les dégénérescences sont les projections et les diagonales respectivement. Si $f:C\lmo D$ est un recouvrement dans $\mcal{C}$ alors le nerf de $f$ est un hyperrecouvrement de $D$ de dimension $0$, et tous les hyperrecouvrements de dimension $0$ sont de cette forme. 

Première étape : on réduit au cas des hyperrecouvrements bornés. On observe que pour tout entier $k\geq 0$, l'hyperrecouvrement $cosk_{k+1}^A B_\bul\lmo A$ est borné et le morphisme unité $B_\bul\lmo cosk_{k+1}^A B_\bul$ induit un isomorphisme sur les $(k+1)$-squelettes. On a un triangle d'espaces,
$$\xymatrix{ hocolim_{\del^{op}} B_\bul \ar[r] \ar[rd] & hocolim_{\del^{op}} cosk^A_{k+1} B_\bul \ar[d] \\ & A }$$
Par le lemme \ref{sque} plus bas, le morphisme horizontal supérieur dans ce triangle induit un isomorphisme sur le groupe d'homotopie $\pi_k$ en tout point base. Supposons l'assertion prouvée pour les hyperrecouvrements bornés. Alors le morphisme vertical droite du triangle induit un isomorphisme sur le $\pi_k$ en tout point base, car l'hyperrecouvrement $cosk^A_{k+1} B_\bul\lmo A$ est borné. On en déduit que le morphisme $hocolim_{\del^{op}} B_\bul \lmo A$ induit un isomorphisme sur le $\pi_k$ en tout point base, et donc que c'est une équivalence car $k$ est arbitraire. 

\begin{lem}\label{sque}
Soit $C_\bul\lmo D_\bul$ un morphisme dans $Top^{\del^{op}}$ qui induit un isomorphisme sur les $(k+1)$-squelettes. Alors le morphisme induit
$$\pi_i (hocolim_{\del^{op}} C_\bul) \lmo \pi_i (hocolim_{\del^{op}} D_\bul)$$
est un isomorphisme pour tout $0\leq i\leq k$ et tout point base. 
\end{lem}

\begin{proof} de \ref{sque} --- On se ramène au cas d'un morphisme dans $SSet^{\del^{op}}$, c'est-à-dire les ensembles bisimpliciaux. En effet on note 
$$\xymatrix{ SSet \ar@<2pt>[r]^-{Re} & Top \ar@<2pt>[l]^-{S}  }$$
l'équivalence de Quillen standard entre les espaces topologiques et les ensembles simpliciaux, où $S$ est le foncteur complexe singulier. Cette paire de Quillen induit une adjonction sur les objets simpliciaux en prenant appliquant simplement $Re$ et $S$ niveau par niveau. Pour tout $C_\bul\in Top^{\del^{op}}$, le morphisme de counité $Re(S(C_\bul))\lmo C_\bul$ est une équivalence niveau par niveau. On en déduit que le morphisme induit 
$$hocolim_{\del^{op}} Re(S( C_\bul)) \lmo hocolim_{\del^{op}} C_\bul$$ 
est une équivalence. D'autre part on a une équivalence naturelle 
$$hocolim_{\del^{op}} Re( S (C_\bul)) \simeq Re( hocolim_{\del^{op}}S( C_\bul)).$$
Par conséquent on a une équivalence 
$$Re( hocolim_{\del^{op}}S C_\bul)\simeq hocolim_{\del^{op}} C_\bul.$$
Pour tout entier $i\geq 0$ et tout point base, on a alors un isomorphisme canonique de groupes 
$$\pi_i(hocolim_{\del^{op}}S C_\bul) \simeq \pi_i Re( hocolim_{\del^{op}}S C_\bul)\simeq \pi_i (hocolim_{\del^{op}} C_\bul).$$
Il suffit donc de prouver l'assertion pour un morphisme $C_\bul\lmo D_\bul$ dans $SSet^{\del^{op}}$ qui induit une équivalence sur les $(k+1)$-squelettes . Dans ce cas on a une équivalence naturelle $hocolim_{\del^{op}} C_\bul \simeq d(C_\bul)$ où $d:SSet^{\del^{op}}\lmo SSet$ est le foncteur diagonal. On a donc une équivalence $d(sk_{k+1} C_\bul) \lmos{\sim} d(sk_{k+1} D_\bul)$. D'autre part le foncteur $sk_{k+1}$ commutent au foncteur $d$, et on obtient donc une équivalence $sk_{k+1} d(C_\bul) \lmos{\sim} sk_{k+1} d(D_\bul)$. A partir de là on peut montrer que le morphisme $d(C_\bul)\lmo d(D_\bul)$ induit donc un isomorphisme sur les $\pi_i$ pour tout $0\leq i\leq k$ pour tout point base. Ceci achève la preuve de \ref{sque}. 
\end{proof}

Deuxième étape : on effectue une récurrence sur la dimension de l'hyperrecouvrement pour se ramener au cas des hyperrecouvrements de dimension $0$, c'est-à-dire des nerfs de morphismes propres surjectifs. Soit $n\geq 0$ un entier. Supposons que l'assertion de la proposition \ref{hypdescpro} soit vraie pour les hyperresouvrements de dimension $\leq n$ et soit $B_\bul\lmo A$ un hyperrecouvrement propre borné de dimension $n+1$. On considère le morphisme unité $B_\bul \lmo cosk_n^A B_\bul=:C_\bul$. Alors $C_\bul$ est un hyperrecouvrement borné de dimension $\leq n$. Considérons l'espace bisimplicial qui est le nerf du morphisme $B_\bul\lmo C_\bul$, 
$$E_{\bul\bul} :=( \xymatrix{B_\bul & B_\bul \times_{C_\bul} B_\bul \ar@<2pt>[l] \ar@<-2pt>[l]  & B_\bul \times_{C_\bul} B_\bul \times_{C_\bul} B_\bul \ar@<3pt>[l] \ar[l] \ar@<-3pt>[l] \cdots } ).$$
En regardant $C_\bul$ comme constant dans une des directions simpliciales, on a un morphisme $E_{\bul\bul}\lmo C_\bul$. La $k$-ième ligne de ce morphisme est le nerf du morphisme $B_k\lmo C_k$. Considérons la diagonale $D_\bul :=d(E_{\bul\bul})$. Alors par des résultats généraux de théorie de l'homotopie (voir e.g. \cite{hirs}), l'espace $hocolim_{\del^{op}} D_\bul$ est équivalent à l'espace obtenue en prenant la colimite homotopique de chaque ligne de $E_{\bul\bul}$, puis en prenant la colimite homotopique l'espace simplicial obtenu. Mais par hypothèse de récurrence, la $k$-ième ligne de $E_{\bul\bul}$ étant un hyperrecouvrement de dimension $0$ de $C_k$, sa colimite homotopique est équivalente à $C_k$. L'espace simplicial obtenu est alors $C_\bul$, qui est de dimension $\leq n$, et donc par hypothèse de récurrence, $hocolim_{\del^{op}} C_\bul \simeq A$. On a donc prouvé que $hocolim_{\del^{op}} D_\bul \simeq A$. 

Montrons maintenant que $B_\bul$ est un rétract de $D_\bul$ sur $A$, et donc que l'application continue $hocolim_{\del^{op}} B_\bul \lmo A$ est une équivalence, étant un rétract de $hocolim_{\del^{op}} D_\bul \lmo A$ qui est une équivalence. On a un morphisme naturel $D_\bul\lmo B_\bul$ donné par la dégénérescence horizontale $E_{0,k} \lmo E_{k,k}$. On cherche alors un morphisme $D_\bul \lmo B_\bul$. Pour ça il est suffisant de définir un morphisme $sk^A_{n+1}D_\bul \lmo sk^A_{n+1}B_\bul$, car alors le morphisme obtenu par adjonction $D_\bul \lmo cosk^A_{n+1} sk^A_{n+1} B_\bul \simeq B_\bul$ nous donne le morphisme cherché. On observe que par définition de $C_\bul$, le morphisme $B_k\lmo C_k$ est un isomorphisme pour $0\leq k\leq n$ et le morphisme $sk^A_n B_\bul \lmo sk^A_n D_\bul$ est un isomorphisme. Soit $[0]\lmo [n+1]$ n'importe quelle coface, elle nous donne alors un morphisme de face 
$E_{n+1, n+1} \lmo E_{0,n+1}$ qui définit un morphisme $sk^A_{n+1}D_\bul \lmo sk^A_{n+1}B_\bul$. On vérifie immédiatement que le morphisme composé $B_\bul \lmo D_\bul \lmo B_\bul$ est l'identité, ce qui prouve que $B_\bul$ est un rétract de $D_\bul$. Comme expliqué plus haut on a alors $hocolim_{\del^{op}} B_\bul \simeq A$. Il suffit donc de prouver l'assertion de la proposition \ref{hypdescpro} pour les hyperrecouvrements de dimension $0$. 

Troisième étape : dimension $0$. Soit $\pi : B_\bul\lmo A$ un nerf d'un morphisme propre surjectif $B_0\lmo A$, avec les $B_0$ et $A$ qui sont des espaces séparés localement compacts et du type d'homotopie d'un CW. On a donc $B_n\simeq B_0\times_A\cdots\times_A B_0$, $(n+1)$ fois. On va utiliser un théorème de changement de base propre pour les préfaisceaux simpliciaux. On va donc travailler avec des préfaisceaux simpliciaux sur l'espace simplicial $B_\bul$ et étudier leur comportement par rapport à $\pi$. Deligne a défini dans \cite{deligneIII} une notion de faisceaux sur un espace simplicial, en construisant un site à partir de celui-ci et en prenant les faisceaux sur ce site. Sa construction se généralise aux préfaisceaux simpliciaux. En effet soit $\tilde{B}_\bul$ la catégorie qui a pour objets les paires $([n], V)$ avec $[n]\in \del$ et $V\subseteq B_n$ une partie ouverte. Un morphisme $([n], V)\lmo ([m], V')$ est la donnée d'un morphisme $a:[n]\lmo [m]$ dans $\del$ et d'une application continue $V'\lmo V$ tel que le carré
$$\xymatrix{V' \ar[r] \ar[d] & V\ar[d] \\ B_m \ar[r]^-{B(a)} & B_n   }$$
commute. La composition et les identités sont définies de manière naturelle. La catégorie $\tilde{B}_\bul$ est naturellement munie de la topologie ouverte induite par la topologie de chaque $B_n$ et $\del$ est considéré comme discret. On peut alors former la catégorie $SPr(B_\bul)$ des préfaisceaux simpliciaux sur le site $\tilde{B}_\bul$. Un objet de $F$ de $SPr(B_\bul)$ est la donnée pour chaque entier $n\geq 0$, d'un préfaisceau simplicial $F_n$ sur $B_n$, et pour tout morphisme $a:[n]\lmo [m]$ dans $\del$, d'un morphisme de préfaisceau $u_a : F_n\lmo B(a)_\ast  F_m$, tel que $u_{id_{[n]}} =id_{F_n} $ et pour tous $a:[n]\lmo [m]$ et $b:[m]\lmo [k]$ dans $\del$, $u_{ba} =B(a)_\ast u_b u_a$. 

Le morphisme d'espaces simpliciaux $\pi : B_\bul\lmo A$ donne un foncteur entre sites $\pi : \tilde{B}_\bul\lmo \tilde{A}$. Considérons le diagramme de catégories,
$$\xymatrix{\tilde{B}_\bul \ar[rr]^-{\pi} \ar[rd]_-q && \tilde{A} \ar[ld]^-p \\ & \ast  &  }$$
où $\ast$ la catégorie ponctuelle. Alors en passant aux préfaisceaux simpliciaux on obtient des jeux d'adjonctions, 
$$\xymatrix{SPr(B_\bul) \ar@<2pt>[rr]^-{\pi_\ast}  \ar@<2pt>[rd]^-{q_\ast} && SPr(A) \ar@<2pt>[ll]^-{\pi^{-1}} \ar@<2pt>[ld]^-{p_\ast} \\ & SSet \ar@<2pt>[lu]^-{cst} \ar@<2pt>[ru]^-{cst} & }$$
où $cst(K)$ est le préfaisceau simplicial constant de valeur $K$ pour tout $K\in SSet$. Les foncteurs $p_\ast$ et $q_\ast$ sont connus sous le nom de section globale et sont adjoints à droite de $cst$. On munit $SPr(B_\bul)$ et $SPr(A)$ de la structure locale relativement à la topologie des espaces en jeu. On choisit ici la localisation de Bousfield à gauche de la structure de modèles injective sur les préfaisceaux (ce qui est réellement important ici est le fait que les équivalences sont les équivalences locales, mais nous arons besoin de calculer une limite homotopique, qui se définie à partir de la structure injective). Les foncteurs $\pi_\ast$, $p_\ast$ et $q_\ast$ sont alors de Quillen à droite. Pour tout $K\in SSet$ on a $\pi^{-1}\circ cst(K)\simeq cst(K)$ et des isomorphismes $\lef \pi^{-1}\simeq  \pi^{-1}$ et $\lef cst\simeq cst$. On a donc un isomorphisme canonique $\R p_\ast \R\pi_\ast \simeq \R q_\ast$. 

Le résultat de Toën \cite[Thm. 2.13]{galhom} implique que pour tout préfaisceau simplicial constant $K\in SPr(A)$ on a un isomorphisme canonique 
$$\R p_\ast (K)\simeq \R\map (S A, K)$$ 
dans $Ho(SSet)$. Cela utilise l'hypothèse selon laquelle $A$ a le type d'homotopie d'un CW-complexe. Nous calculons ensuite les sections globales dérivées $\R q_\ast (K)$ d'une préfaisceau simplicial constant $K$ sur $B_\bul$. Le site $\tilde{B}_\bul$ vient avec un foncteur $\alpha : \tilde{B}_\bul\lmo \del$ où $\del$ est comme un site discret. Le foncteur $\alpha$ est simplement la projection $\alpha ([n], V)=[n]$. On note $\beta :\del\lmo \ast$ le foncteur vers la catégorie ponctuelle. On a alors un diagramme 
$$\xymatrix{SPr(B_\bul) \ar[r]^-{\alpha_\ast} \ar[rd]_-{q_\ast} & SPr(\del) \ar[d]^-{\beta_\ast} \\ & SSet }$$
Ces foncteurs sont de Quillen à droite ($SPr(\del)$ est munie de la structure injective). Le foncteur dérivé $\R \beta_\ast$ est isomorphe à $holim_\del$. Pour tout préfaisceau simplicial constant $K\in SPr(B_\bul)$, on a $\R\alpha_\ast (K)=( \R q_{n\ast} K)_{n\geq 0}$ où $q_n:\tilde{B}_n\lmo \ast$. En utilisant \cite[Thm. 2.13]{galhom} (vu que les $B_n$ ont le type d'homotopie de CW-complexes), on obtient un isomorphisme $\R\alpha_\ast (K)\simeq (\R\map(SB_n, K))_{n\geq 0}$ dans $Ho(SPr(\del))$. Et donc un isomorphisme canonique dans $Ho(SSet)$
$$\R q_\ast (K)\simeq holim_{\del^{op}} \R\map(SB_\bul, K)\simeq \R\map(hocolim_{\del^{op}} SB_\bul, K).$$
Supposons le lemme suivant établi.

\begin{lem}\label{pbc}
Soit $K\in SPr(B_\bul)$ un préfaisceau simplicial constant, avec $K$ un ensemble simplicial tronqué (c'est-à-dire isomorphe à un de ses squelettes). Alors le morphisme unité 
$$K\lmo \R\pi_\ast \pi^{-1} (K)\simeq \R\pi_\ast(K)$$
est un isomorphisme dans $Ho(SPr(A))$ (où $K$ désigne le même préfaisceau simplicial constant sur $A$). 
\end{lem}

En appliquant l'isomorphisme $\R p_\ast \R\pi_\ast \simeq \R q_\ast$ à un préfaisceau simplicial constant tronqué $K$ sur $B_\bul$, on obtient un isomorphisme canonique pour tout ensemble simplicial $K$,
$$\R\map(SA, K)\simeq \R\map(hocolim_{\del^{op}} SB_\bul, K).$$
Ce qui implique que le morphisme $hocolim_{\del^{op}} SB_\bul\lmo SA$ est un isomorphisme dans $Ho(SSet)$. En appliquant le foncteur $Re$ (qui commute aux colimites homotopiques), on en déduit que le morphisme $hocolim_{\del^{op}} B_\bul\lmo A$ est un isomorphisme dans $Ho(Top)$, ce qui prouve notre assertion dans le cas des hyperrecouvrements de dimension $0$. 

Pour résumer, on est simplement ramener à prouver \ref{pbc}. 

\begin{proof} de \ref{pbc} --- Il suffit pour cela de montrer que le morphisme unité  $K\lmo \R\pi_\ast \pi^{-1} (K)\simeq \R\pi_\ast(K)$ induit un isomorphisme sur la fibre en un point quelconque $a\in A$. Pour tout $a\in A$ on a un carré cartésien d'espaces simpliciaux 
$$\xymatrix{B_\bul^a\ar[r]^-{\pi^a} \ar[d]_-\phi & \ast \ar[d]^-a \\ B_\bul \ar[r]^-{\pi}  & A }$$
On prétend que le théorème de changement de base propre est vrai pour ce carré cartésien, c'est le contenu du lemme suivant.
  
\begin{lem}\label{pbcreal}
Pour tout préfaisceau simplicial constant tronqué $K\in SPr(B_\bul)$, le morphisme canonique 
$$a^{-1} \R \pi_\ast (K)\lmo \R\pi^a_\ast \phi^{-1} (K)\simeq\R\pi^a_\ast (K) $$
est un isomorphisme dans $Ho(SSet)$. 
\end{lem}

Supposons pour le moment que le lemme \ref{pbcreal} est établi, alors $K\lmo \R\pi^a_\ast (K)$ est un isomorphisme dans $Ho(SSet)$, ce qui est exactement l'énoncé de \ref{hypdescpro} pour $A=\ast$ et $B_\bul$ un hyperrecouvrement de dimension $0$. En effet on a un isomorphisme $ \R\pi^a_\ast (K)\simeq holim_{\del^{op}} \R\map(SB_\bul^a, K)$. 

Pour prouver \ref{pbcreal}, on calcule la fibre $a^{-1} \R \pi_\ast (K)$. Pour toute partie ouverte $U\subseteq A$ on a un carré cartésien d'espaces simpliciaux 
$$\xymatrix{B_\bul^U\ar[r]^-{\pi^U} \ar[d]_-{\phi^U} & U \ar[d]^-i \\ B_\bul \ar[r]^-{\pi}  & A }$$
On a alors des isomorphismes $\R\Gamma (U, \R\pi_\ast^U (K))\simeq \R\Gamma(B_\bul^U, K)\simeq holim_{\del^{op}} \R\map(SB_\bul^U, K)$ dans $Ho(SSet)$. Le fibre 
 $a^{-1} \R \pi_\ast (K)$ est isomorphe à la colimite filtrante standard 
$$colim_{a\in U\subseteq A} \R\Gamma (U, \R\pi_\ast^U (K)) \simeq colim_{a\in U\subseteq A} holim_{\del^{op}} \R\map(SB_\bul^U, K).$$
Maintenant l'hypothèse selon laquelle $K$ est tronqué implique que cette limite homotopique est isomorphe à un limite homotopique finie. En effet si $K$ est $n$-tronqué, alors l'ensemble simplicial $\R\map(SB_\bul^U, K)$ est aussi $n$-tronqué et on a calcule cette limite en restreignant les indices à la sous-catégorie de $\del^{op}$ donné par les simplexes de $\del$ de dimension $\leq n+1$. La colimite filtrante et la limite homotopique finie commute et on a 
$$a^{-1} \R \pi_\ast (K) \simeq holim_{\del^{op}} colim_{a\in U\subseteq A}  \R\map(SB_\bul^U, K).$$
On applique alors le théorème de changement de bases propre de Lurie \cite[Cor. 7.3.1.18]{htt} au carré cartésien d'espaces séparés localement compacts, 
$$\xymatrix{ B_n^a\ar[r]^-{\pi_n^a} \ar[d]_-{\phi_n} & \ast \ar[d]^-a \\ B_n \ar[r]^-{\pi_n}  & A  }$$
Lurie démontre son résultat dans le contexte de la descente et non de l'hyperdescente ; cependant d'après \cite[Cor. 7.2.1.12]{htt} un préfaisceau simplicial constant tronqué a la condition d'hyperdescente si et seulement s'il a la condition de descente ordinaire. On obtient un isomorphisme $a^{-1} \R\pi_{n\ast } (K)\simeq \R\pi_{n\ast}^a(K)$. Par le même argument que précédemment $a^{-1} \R\pi_{n\ast } (K)\simeq colim_{a\in U\subseteq A}  \R\map(SB_n^U, K)$, ce qui prouve que $a^{-1} \R \pi_\ast (K)\simeq \R\pi^a_\ast (K)$. Ceci achève donc la preuve de \ref{pbcreal}.

Pour montrer \ref{pbc}, il reste seulement à montrer que $K\lmo \R\pi^a_\ast (K)$ est un isomorphisme dans $Ho(SSet)$ pour tout $K$ tronqué. Vu ce qui a été dit précédemment, c'est équivalent à dire que $hocolim_{\del^{op}} B_\bul^a \lmo \ast$ est une équivalence d'espaces. C'est donné par le lemme suivant. 

\begin{lem}\label{puresimp}
Soit $X$ un espace topologique non vide (resp. un ensemble simplicial non vide). Alors le nerf $X_\bul \lmo \ast$ du morphisme $p:X\lmo \ast$ induit une équivalence $hocolim_{\del^{op}} X_\bul \lmo \ast$ dans $Top$ (resp. dans $SSet$). 
\end{lem}

\begin{proof} de \ref{puresimp} --
L'assertion dans $SSet$ implique celle dans $Top$. En effet si $X\in Top$ on a vu dans la preuve de \ref{sque} que $hocolim_{\del^{op}} X_\bul$ et $hocolim_{\del^{op}} SX_\bul$ ont les mêmes groupes d'homotopie. 

Soit $X\in SSet$. Le morphisme $X_\bul \lmo \ast$ est une équivalence d'homotopie simpliciale dans $SSet^{\del^{op}}$. En effet soit $x:\ast \lmo X$ un point. On définit une homotopie $h:\del^1\times X_\bul\lmo X_\bul$ entre $id_{X_\bul}$ et $xp$. On définit $h_n:\del([n], [1]) \times X_n\lmo X_n$ par la formule suivante. Soit $a:[n]\lmo [1]$ un morphisme dans $\del$, il est essentiellement donné par un entier $0\leq m\leq n$. On pose $h_n(a, (x_0, \hdots, x_n))=(x_0, \hdots, x_m, x,\hdots, x)$. On a alors une homotopie qui vérifie $h(0,-)=id_{X_\bul}$ et $h(1,-)=xp$. 

On rappelle que le foncteur de réalisation
$$\re{-} : SSet^{\del^{op}} \lmo SSet$$
est définie par la formule suivante
$$\re{Y_\bul}:=coeq(\xymatrix{\bigsqcup_{n\in \del} \del^n\times Y_n& \bigsqcup_{p\mo q\in \del} \del^p\times Y_q \dar[l] } ).$$
Il a la propriété d'envoyer les équivalences d'homotopie simpliciale sur des équivalences d'homotopie simpliciale. On en déduit que $\re{X_\bul}$ est contractile. En outre on a un isomorphisme $hocolim_{\del^{op}} X_\bul \simeq \re{X_\bul}$ dans $Ho(SSet)$ (voir \cite{hirs}) et donc $hocolim_{\del^{op}} X_\bul$ est contractile. 
\end{proof}

Ceci achève la preuve du lemme \ref{pbc} et donc de la proposition \ref{hypdescpro}. 
\end{proof}
\end{proof}

\newpage

\section{K-théorie topologique des espaces non commutatifs}

\subsection{K-théorie semi-topologique et topologique}\label{kstktop}

En \ref{caracalg}, on a défini un préfaisceau d'anneaux en spectres symétriques 
$$\ka : \affc^{op}\lmo Sp$$
tel que pour tout $\spec(A)\in \affc$, on a un isomorphisme canonique $\ka(\spec(A))\simeq\kn(A)$ dans $Ho(Sp)$. Pour toute $\C$-dg-catégorie $T\in\dgcatc$, on a défini un préfaisceau de $\ka$-modules en spectres symétriques
$$\ukn(T) : \affc^{op} \lmo Sp$$
tel que pour tout $\spec(A)\in \affc$, on a un isomorphisme canonique $\ukn(T)(\spec(A))\simeq \kn(T\tel_\C A)$ dans $Ho(Sp)$. On a donc $\ukn(\unit)\simeq \ka$ où $\unit$ désigne la $\C$-dg-catégorie ponctuelle avec un seul objet et l'anneau $\C$ comme endomorphismes. On rappelle qu'on note $\ka-Mod_\s$ la catégorie des objets en $\ka$-modules dans $\spaf$. On a donc $\ukn(T)\in \ka-Mod_\s$. 

En \ref{realsp} on a vu que les foncteurs de réalisation spectrale (non dérivé et dérivé), 
$$\sspsp : \spaf\lmo Sp,$$
$$\resp{-} : Ho(\spaf)\lmo Ho(Sp)$$ 
sont des foncteurs monoïdaux. Par l'existence d'une catégorie homotopique des objets en monoïdes et en modules (voir \cite[Thm.3.3 et Thm.2.1]{hoveymon}), on en déduit que $\resp{\ka}$ est un anneau en spectres symétriques et qu'on obtient un foncteur de réalisation spectrale pour les $\ka$-modules à valeurs dans les spectres en $\resp{\ka}$-modules, 
$$\resp{-} : Ho(\ka-Mod_\s)\lmo Ho(\resp{\ka}-Mod_\s).$$
On a une version connective de ces spectres de K-théorie, noté $\tilde{\ka}$ et $\ukc(T)$. De même on obtient un foncteur de réalisation spectrale à valeurs pour les $\tilde{\ka}$-modules, 
$$\resp{-} : Ho(\tilde{\ka}-Mod_\s)\lmo Ho(\resp{\tilde{\ka}}-Mod_\s).$$

\begin{df}\label{defkst}
La \emph{K-théorie semi-topologique} (resp. la \emph{K-théorie semi-topologique connective}) d'une $\C$-dg-catégorie $T\in \dgcatc$ est par définition le spectre symétrique en $\resp{\ka}$-modules (resp. en $\resp{\tilde{\ka}}$-modules),
$$\kst(T):=\resp{\ukn (T)} \qquad (\textrm{resp.  } \kcst(T):=\resp{\ukc (T)}).$$
Par l'existence d'un foncteur de remplacement cofibrant fonctoriel dans les catégories de modèles concernées, on a donc deux foncteurs
$$\kst : \dgcatc\lmo \resp{\ka}-Mod_\s,$$
$$\kcst : \dgcatc\lmo \resp{\tilde{\ka}}-Mod_\s.$$
On notera $\kst_i(T):=\pi_i\kst(T)$ pour tout $i\in \Z$. 
\end{df}

\begin{rema}
Il nous semble nécessaire d'inclure la définition de $\kcst$ étant donné qu'il nous est a priori inconnu que $\kcst(T)$ est le revêtement connectif de $\kst(T)$. 
\end{rema}

\begin{rema}\label{fibtop}
On remarque l'existence d'un morphisme reliant la K-théorie algébrique et la K-théorie semi-topologique. Si $T\in \dgcatc$, on a le morphisme d'unité de la paire de foncteurs adjoints $(\resp{-}, \hbs)$,
$$\ukn(T)\lmo \hbs(\resp{\ukn(T)}).$$
En prenant les sections globales, c'est-à-dire la valeur sur $\spec(\C)$, on obtient un morphisme dans $Ho(Sp)$, 
$$\eta_T:\kn(T)\lmo \rhomi_{Ho(Sp)}(\sinf(\spec(\C))_+,\kst(T))\simeq \kst(T).$$
Ceci définit un morphisme $\eta:\kn\lmo \kst$ dans $Ho(Sp^{\dgcatc})$. On a une version connective $\tilde{\eta}:\kc\lmo \kcst$. Pour un schéma $X\in \schc$, en prenant le $\pi_0$ on a un morphisme 
$$\kc_0(X)\lmo \kcst_0(X).$$
Vu la formule \ref{pi0cor}, ce morphisme est le passage au quotient par la relation d'équivalence sur les fibrés vectoriels algébriques sur $X$ qui identifie deux fibrés quand on peut les relier par une courbe algébrique connexe. On en déduit un morphisme $\kcst_0(X)\lmo \ktopu^0(sp(X))$ vers le groupe de Grothendieck des fibrés vectoriels topologiques sur l'espace $sp(X)$. Nous allons voir plus bas que ce dernier morphisme est une bijection dans le cas d'un schéma propre lisse de type fini sur $\C$. 
\end{rema}

\begin{rema}\label{restlisskst}
En appliquant la version spectrale du théorème \ref{restliss}, on voit qu'on peut calculer les réalisations topologiques de la définition \ref{defkst} en restreignant d'abord les préfaisceaux de K-théorie algébrique aux schémas lisses.
\end{rema}

Les deux résultats suivant sont décisifs pour la définition de la K-théorie topologique des espaces non commutatifs.

On désigne par $bu$ le spectre de K-théorie topologique complexe connective usuel. C'est-à-dire que pour tout espace topologique compact $X$, si $\ktopu^0(X)$ est le groupe de Grothendieck des fibrés vectoriels topologiques complexes sur $X$, on a un isomorphisme 
$\ktopu^0(X)\simeq\pi_0 \map_{Ho(Sp)} (\sinf SX_+, bu)$. Un modèle de $bu$ en tant que spectre symétrique sera donné plus bas.

\begin{theo}\label{bu}
Il existe un isomorphisme canonique $\kcst(\unit)\simeq bu$ dans $Ho(Sp)$. 
\end{theo}

Une preuve de ce théorème est donnée plus bas en \ref{ktoppoint}. Pour le moment nous l'utilisons pour donner la définition de la K-théorie topologique.

\begin{theo}\label{annupoint}
Pour toute algèbre commutative lisse $B\in \calgc$, le morphisme naturel 
$$\kcst(B)\lmo \kst(B)$$ 
est un isomorphisme dans $Ho(Sp)$. En particulier par le théorème \ref{bu} on a un isomorphisme $\kst(\unit)\simeq \bu$ dans $Ho(Sp)$. 
\end{theo}

\begin{proof} de \ref{annupoint} --- C'est un fait connu que la K-théorie algébrique négative d'une algèbre commutative lisse est nulle (cf. \cite[Rem.7]{schl}). Par conséquent, le morphisme de préfaisceau en spectres $\ukc(B)\lmo \ukn(B)$ est une équivalence sur les schémas affines lisses. Par le théorème \ref{restliss}, on en déduit que ce morphisme induit une équivalence sur les réalisations spectrales $\kcst(B)\simeq \kst(B)$. 
\end{proof}

\begin{nota} 
Ces deux théorèmes se reformule en disant que l'on a des isomorphismes 
$$\resp{\tilde{\ka}}\simeq \resp{\ka}\simeq bu$$
dans $Ho(Sp)$. On notera donc $\bu$ l'anneau en spectres symétriques $\resp{\ka}$. On a donc des foncteurs
$$\kcst, \kst : \dgcatc\lmo \bu-Mod_\s.$$
\end{nota}

\begin{rema}\label{uniqbu}
C'est un fait classique que le spectre symétrique $bu$ qui représente la K-théorie topologique admet un modèle dans les spectres symétriques en anneaux commutatifs (associatifs et unitaires), avec l'addition correspondant à la somme des fibrés et la multiplication au produit tensoriel. Un moyen d'exprimer cela est de considérer $bu$ comme le spectre associé au $\gam$-espace spécial qui est la réalisation topologique du champ des fibrés vectoriels (voir plus en \ref{ktoppoint}) et de voir ce $\gam$-objet comme un $\gam$-anneau commutatif dont la loi est induite par le produit tensoriel. Cette structure d'anneaux commutatif sur $bu$ est de plus \emph{unique} par \cite[Cor.1.4]{bakerrichter} au sens que pour tout spectre en anneau commutatif $A$ et tout morphisme de spectres en anneaux $f:A\lmo bu$ qui induit un isomorphisme sur tous les groupes d'homotopie, alors il existe un morphisme $g : A\lmo bu$ dans la catégorie homotopique des spectres en anneaux commutatifs tel que $g$ est isomorphe à $f$.  

Ici nous n'avons pas la possibilité de considérer $\bu$ comme un spectre symétriques en anneau commutatif, ni $\ka$ comme un anneau commutatif car la structure d'anneau nous est donnée par celle des endomorphismes de l'unité dans $\mbb{M}_{loc} (\C)$ (voir la partie \ref{caracalg}). Nous ne parlerons donc que d'anneaux associatifs (unitaires), tout en sachant que notre spectre $\bu$ est équivalent dans $Sp$ à un spectre symétrique en anneau commutatif et que cette structure est unique. Remarquons que l'on ressent ici la limitation imposée par le language des catégories de modèles et des structures algébriques strictes, et qu'il serait beaucoup plus comfortable de travailler avec les $\infty$-catégories monoïdales de Lurie. 
\end{rema}

\begin{nota}\label{choixbeta}
 Par le théorème de périodicité de Bott, on sait que le groupe abélien $\pi_2\bu$ est libre de rang $1$. Pour définir la K-théorie topologique, on choisit un générateur de Bott $\beta\in \kst_2(\unit)=\pi_2 \bu$ parmi les deux possibles. Il sera plus commode de choisir d'abord un générateur $\alpha$ de la partie non triviale de $\kn_0(\po)$ et de prendre $\beta=\eta(\alpha)$ où $\eta:\kn_0(\po)\lmo \kst_0(\po)$ est le morphisme canonique. Alors $\beta$ nous donne un générateur de $\kst_2(\unit)$ par le morphisme canonique $\kst_0(\po)\lmo \ktopu(S^2) \simeq \ktopu^0(\ast)\oplus \beta \ktopu^{-2}(\ast)$. 
\end{nota}

\begin{nota}
On rappelle qu'étant donné un spectre symétrique en anneau $A$, un entier $k$, et un élément $a\in\pi_k A$, on peut définir l'anneau en spectre localisé $A[a^{-1}]$ par rapport à l'élément $a$. Il est munit d'un morphisme d'anneau $i_a:A\lmo A[a^{-1}]$ et vérifie que pour tout anneau en spectre $B$, et tout morphisme d'anneau $A\lmo B$, l'ensemble simplicial $\map_{A-Alg} (A[a^{-1}], B)$ est non vide si et seulement si la multiplication par $a$ est inversible dans le $\pi_*(A)$-module $\pi_*(B)$. Cette propriété caractérise l'objet $A[a^{-1}]$ à équivalence près. Ceci est par exemple une conséquence de la version non-commutative de \cite[Prop.1.2.9.1]{hag2} appliqué à la catégorie de modèles monoïdale $Sp$ des spectres symétriques. On en déduit alors par \cite[Cor.1.2.9.3]{hag2} que le foncteur obtenue par composition avec $i_a$, 
$$i_a^* : Ho( A[a^{-1}]-Mod_\s) \lmo Ho(A-Mod_\s)$$
est pleinement fidèle et son image essentielle est formée des $A$-modules $M$ tel que la multiplication par $a$ est inversible dans le $\pi_*A$-module $\pi_*M$. Si $M$ est un $A$-module et $a\in\pi_kA$ un élément, le localisé $M[a^{-1}]$ est alors défini comme étant le $A[a^{-1}]$-module $M\sm_A A[a^{-1}]$. 
\end{nota}

\begin{rema}\label{remloccomass}
Soit $A$ un spectre symétrique en anneau commutatif  et $a\in \pi_k A$. Notons $A_{ass}[a^{-1}]$ la localisation de $A$ par rapport à $a$ au sens des anneaux associatifs et $A_{com}[a^{-1}]$ la localisation de $A$ par rapport à $a$ au sens des anneaux commutatifs. Alors par la propriété universelle il existe un morphisme d'anneaux $ c: A_{ass}[a^{-1}]\lmo A_{com}[a^{-1}]$. Ce morphisme est une équivalence car il existe un diagramme commutatif,
$$\xymatrix{Ho(A_{ass}[a^{-1}]-Mod_r) \ar@{^{(}->}[r]  \ar[d]^-{c^*} & Ho(A-Mod_r) \ar[d]^-{\wr} \\ Ho( A_{com}[a^{-1}]-Mod) \ar@{^{(}->}[r] & Ho(A-Mod)  }$$
où l'indice $r$ désigne la catégorie des modules à droite. Comme $A$ est commutatif le morphisme des $A$-modules à droite dans les $A$-modules est une équivalence. Alors les deux catégories de modules $Ho(A_{ass}[a^{-1}]-Mod_r)$ et $Ho( A_{com}[a^{-1}]-Mod)$ s'identifie toutes les deux à la sous-catégorie de $Ho(A-Mod)$ constituée des $A$-modules pour lesquels la multiplication par $a$ est une équivalence. On en déduit que $c^*$ est une équivalence de catégories et donc que $c$ est une équivalence. 
\end{rema}

\begin{rema}
On peut donc considérer le spectre en anneau $\bu$ et l'élément de Bott $\beta\in \pi_2\bu$ et sa localisation $\bu[\beta^{-1}]$. A priori cette localisation est calculé au sens des anneaux associatifs. Cependant par la remarque \ref{uniqbu}, on sait qu'il existe un morphisme d'anneaux $\bu\lmos{\sim} bu$ avec $bu$ un modèle dans les anneaux commutatifs, qui est une équivalence. On en déduit par la remarque \ref{remloccomass} qu'il n'y pas d'ambiguité sur la nature de l'anneau localisé $\bu[\beta^{-1}]$ : il est équivalent à la localisation commutative de $bu$ et est donc équivalent à la colimite standard
$$\bu[\beta^{-1}]\simeq \xymatrix{  colim (bu \ar[r]^-{\cup \beta}  & bu \ar[r]^-{\cup \beta} & \cdots )}$$
où $\cup \beta$ désigne le morphisme produit par $\beta$. La structure d'anneaux de $\bu[\beta^{-1}]$ est donc bien la structure usuelle, notamment on a un isomorphisme $\bu[\beta^{-1}]\sm_\s H\C \simeq H\cuu$ dans la catégorie homotopique des $H\C$-algèbres. On notera $\BU:=\bu[\beta^{-1}]$.

\end{rema}

\begin{df}\label{deftop}
Le \emph{spectre de K-théorie topologique d'une $\C$-dg-catégorie $T\in \dgcatc$} est par définition le spectre symétrique,
$$\ktop(T):=\kst(T)[\beta^{-1}].$$
Ceci définit donc un foncteur
$$\ktop : \dgcatc\lmo \BU-Mod_\s.$$
On note $\ktop_i(T):=\pi_i\ktop(T)$ pour tout $i\in\Z$. 
\end{df}

\begin{rema}
En suivant la remarque \ref{fibtop}, on compose le morphisme unité $\kn\lmo \kst$ avec le morphisme de structure $\kst\lmo \ktop$ et on obtient de cette manière un morphisme toujours noté $\eta : \kn\lmo \ktop$ de la K-théorie algébrique vers la K-théorie topologique. 
\end{rema}

La K-théorie topologique hérite des propriétés de la K-théorie algébrique de la proposition \ref{propkn}. 

\begin{prop} 
\begin{description}
\item[a.] La K-théorie topologique commute aux colimites homotopiques filtrantes de dg-catégories. 
\item[b.] La K-théorie topologique envoie les équivalences Morita dérivée sur des équivalences. 
\item[c.] Pour toute suite exacte $\xymatrix{T'\ar[r]^-i & T\ar[r]^-p & T''}$, la suite induite
$$\ktop(T')\lmo \ktop(T)\lmo \ktop(T'')$$
est un triangle distingué dans $Ho(Sp)$. 
\end{description}
\end{prop}

\begin{nota}\label{notaktopsch}
On peut définir la K-théorie topologique des $\C$-schémas en utilisant notre nouvelle définition. On utilise les notations \ref{notaksch}. Si $X\in\schc$ est un $\C$-schéma de type fini, la K-théorie topologique de $X$ au sens dg est par définition le spectre symétrique $\ktop(X):=\ktop(\lpe(X))$. Ceci définit un foncteur toujours noté
$$\ktop : \schc\lmo Sp.$$
\end{nota}

\subsection{K-théorie semi-topologique du point}\label{ktoppoint}

On donne un modèle particulier pour le spectre $bu$, en tant que spectre symétrique. Pour toute algèbre commutative $A\in \calgc$, on note $\proj(A)$ la catégorie de Waldhausen des $A$-modules projectif de type fini, c'est-à-dire des fibrés vectoriels de rang fini sur $\spec(A)$ (voir la remarque \ref{algebra}). Les équivalences sont les isomorphismes et les cofibrations sont les monomorphismes admissibles. Ceci définit un pseudo-foncteur $\affc^{op} \lmo WCat$, puisque étant donné un morphisme d'algèbres $A\lmo B$, on a un foncteur exact 
\begin{align*}
\proj(A) & \lmo \proj(B) \\
E & \longmapsto E\te_A B
\end{align*}
et ceux-ci vérifient l'associativité seulement à isomorphisme canonique près. On note $\proj$ la strictification canonique de ce pseudo-foncteur. On note $\vect=Nw\proj$ le préfaisceau simplicial nerf des isomorphismes dans $\proj(A)$. La somme directe des modules définit une structure de monoïdes commutatif 'à homotopie cohérente près' sur $\vect$. Plus précisément, en utilisant la construction $B_W$ donnée en \ref{delass}, on a un $\gam$-préfaisceau simplicial 
$$\vect_\bul :=NiB_W \vect \in \gam-\spr$$
tel que pour tout entier $n\geq 0$, on ait une équivalence niveau par niveau $\vect_{(n)}\simeq \vect^n$. On définit le spectre symétrique $bu$ comme étant le spectre symétrique connectif, 
$$bu:=\mcal{B}\regam{\vect_\bul}^+.$$

\begin{rema}
On utilise les notations de \ref{prestru} concernant le classifiant des objets en groupes. On note $Gl_n : \affc^{op}\lmo SSet$ le préfaisceau simplicial discret en groupes linéaires. On note $\coprod_{n\geq 0} \mrm{B}Gl_n$ le $\gam$-préfaisceau simplicial avec la structure de monoïdes donnée par la somme par blocs des matrices. Alors il est connu qu'il existe une équivalence étale locale de $\gam$-préfaisceau simpliciaux,
$$\vect_\bul \simeq \coprod_{n\geq 0} \mrm{B}Gl_n$$
Alors par le théorème \ref{di}, les propositions \ref{proprelder} et \ref{gp}, on en déduit des équivalences
$$\re{\vect_\bul}^+\simeq \re{\coprod_{n\geq 0} \mrm{B}Gl_n}^+ \simeq (\coprod_{n\geq 0} \re{\mrm{B}Gl_n})^+ \simeq (\coprod_{n\geq 0} \mrm{B}\re{Gl_n})^+\simeq  (\coprod_{n\geq 0} \mrm{B} Gl_n(\C))^+$$
où $Gl_n(\C)$ désigne maintenant le groupe topologique. Celui-ci a le même type d'homotopie que le groupe unitaire $U_n(\C)$, on a donc une équivalence 
$$\re{\vect_\bul}^+\simeq (\coprod_{n\geq 0} \mrm{B} U_n(\C))^+.$$
Mais la complétion en groupe de $\coprod_{n\geq 0} \mrm{B} U_n(\C)$ est bien connue, voir par exemple \cite[App.Q]{fmfilt}, on a 
$$(\coprod_{n\geq 0} \mrm{B} U_n(\C))^+\simeq \mrm{B}U_\infty \times \Z,$$ 
où $BU_\infty$ est la colimite des $\mrm{B}U_n(\C)$ par rapport aux inclusions naturelle $\mrm{B}U_n(\C)\hookrightarrow \mrm{B}U_{n+1}(\C)$, avec la structure de $\gam$-objets toujours donné par la somme par blocs des matrices et la structure usuelle pour la composante $\Z$. Par conséquent, par le théorème \ref{bf}, on en déduit une équivalence de spectres $\mcal{B}\regam{\vect_\bul}^+ \simeq \mcal{B} (BU_\infty \times \Z)$, ce qui est la définition usuelle du spectre $bu$. 
\end{rema}

\begin{proof} du théorème \ref{bu} --- On a une chaîne d'isomorphismes naturels dans $Ho(Sp)$, 
\begin{align*}
\kcst(\ast)& =\resp{\ukc(\ast)} \\
& \simeq \resp{\kc(\vect)} \qquad &&\textrm{(par la remarque \ref{algebra})} \\
& = \resp{\mcal{B} K^\gam(\vect)} \qquad & &\textrm{(par définition du spectre de K-théorie, cf. \ref{delass})} \\
&\simeq \mcal{B} \regam{ K^\gam(\vect)} \qquad &&\textrm{(par la proposition \ref{proprelb}) } \\
&\simeq \mcal{B} \regam{(\vect_\bul)^+} \qquad &&\textrm{(par le lemme \ref{cofsc}, vu que les cofibrations de $\vect$ sont scindées) } \\
&\simeq \mcal{B} \regam{\vect_\bul}^+ \qquad &&\textrm{(par la proposition \ref{proprelplus})} \\
& =bu. 
\end{align*}
\end{proof}

\subsection{K-théorie semi-topologique et les champs de dg-modules}

La K-théorie semi-topologique, comme initiée par Toën (voir \cite{sat}, \cite{kal}, \cite{kkp}) a été d'abord définie comme la réalisation topologique d'un certain champ des dg-modules pseudo-parfaits associé à une dg-catégorie. Il existe en fait deux champs associés à un dg-catégorie $T$. Il s'agit du champ $\M_T$ des dg-modules pseudo-parfaits et du champ $\M^T$ des dg-modules parfaits. Ce dernier est pertinent pour notre étude la K-théorie topologique vu que la K-théorie algébrique se calcule avec la catégorie des dg-modules parfaits. D'autre part, le champ $\M_T$ donne lieu à une théorie duale à la K-théorie topologique qui mérite le nom de K-cohomologie topologique. Le champ $\M_T$ est étudié par Toën--Vaquié dans \cite{modob}, où il est prouvé sa locale géométricité et ses propriétés de finitude. Dans cette partie on exprime le lien entre la K-théorie semi-topologique connective et le champ $\M^T$, ainsi que le lien entre la K-cohomologie topologique et le champ $\M_T$. Ce résultat est basé sur l'existence d'une $\ao$-équivalence entre le champ $\M^T$ et la $S$-construction de la catégorie des dg-modules parfaits. 
\\

Pour toute $\C$-dg-catégorie $T$ on définit des préfaisceaux de catégorie de Waldhausen. On note 
\begin{itemize}
\item $\uparf(T) : \spec(A)\longmapsto \parf(T\tel_\C A)=\parf(T,A)$. 
\item $\upspa(T) : \spec(A) \longmapsto \pspa(T\tel_\C A)=\pspa(T,A)$, où cette dernière est la catégorie des $T\tel_\C A$-dg-modules qui sont parfaits relativement à $A$, encore appellés dg-modules pseudo-parfaits (voir \cite[Déf.2.7]{modob}). 
\end{itemize}
Ces foncteurs stricts sont obtenus comme strictification canonique de pseudo-foncteurs pour lesquels la fonctorialité est induite par l'image directe. De plus pour obtenir des foncteurs stricts, rappellons qu'on a fixé un foncteur de remplacement cofibrant pour les dg-catégories à la notation \ref{notapref}. 

Ces deux préfaisceaux de catégories de Waldhausen donne lieu à des champs, 
\begin{itemize}
\item $\M^T =Nw\uparf(T) : \spec(A)\longmapsto Nw\parf(T,A)$. 
\item $\M_T=Nw\upspa(T) : \spec(A)\longmapsto Nw\pspa(T,A)$. 
\end{itemize}
où $Nw$ désigne toujours le nerf des équivalences (c'est-à-dire des quasi-isomorphismes) dans les catégories de Waldhausen concernées. La somme directe des dg-modules induit une structure de monoïde commutatif à homotopie cohérente près sur ces champs. On applique le foncteur $B_W$ défini en \ref{gamstru} et on obtient des $\gam$-objets spéciaux dans $\spr$, 
$$\M^T_\bul = NwB_W\uparf(T), \qquad
\M_T^\bul =NwB_W\upspa(T).$$

\subsubsection{Les cofibrations se scindent à $\ao$-homotopie près}

Soit $T\in \dgcatc$ une dg-catégorie sur $\C$. Tout ce qui va être dit dans ce paragraphe est en fait valable pour le champ des pseudo-parfaits $\M_T$ en remplaçant la K-théorie des parfaits par celle des pseudo-parfaits. Pour fixer les idées et parce qu'on s'intéresse dans ce travail à la K-théorie des parfaits on écrit ce qui suit avec le champ $\M^T$. 

En utilisant les notations de l'exemple \ref{exespk}, on a un $\del$-préfaisceau simplicial donné par 
$$\K^T_\bul:=\K_\bul (\uparf(T))=NwS_\bul \uparf(T)$$

On pratique l'abus de notation qui consiste à considérer $\M^T$ comme un $\del$-objet de façon naturelle, c'est à dire en appliquant de foncteur $\alpha^*$ défini en \ref{strucdel}. On rappelle qu'on a défini en \ref{delass}, un morphisme de $\del$-objet noté $\phi$ qui relie la construction $B_W$ à la $S$-construction. La même construction que pour $\phi$ nous donne un morphisme de $\del$-objets
$$\lambda_\bul  : \M^T_\bul\lmo \K^T_\bul.$$

\begin{prop}\label{mk} 
Le morphisme $\lambda_\bul  : \M^T_\bul \lmo \K^T_\bul$ est une $\ao$-équivalence dans $\del-\spretao$. 
\end{prop}

Ce résultat est central dans notre étude de la K-théorie topologique des dg-catégories. Il nous permet de montrer plus bas le lien entre $\kcst(T)$ et $\M^T$ pour une dg-catégorie $T$. Il peut se résumer en disant que bien que les cofibrations ne sont pas forcément scindées dans le préfaisceau de catégories de Waldhausen $\parf(T,-)$, si on s'autorise à le regarder à $\ao$-homotopie près, alors c'est le cas. Ceci fait le lien avec la K-théorie qui est précisément l'invariant pour lequel les cofibrations se scindent par le théorème d'additivité de Waldhausen (cf. \cite[Thm.1.4.2]{wald}). La preuve de la proposition \ref{mk} occupe le reste de cette partie. 

\begin{proof} 
On introduit d'abord quelques notations. Soit $n\geq 1$ un entier. On désigne par $[n]$ la catégorie associé l'ensemble ordonné $\{1<2<\cdots <n\}$. On pose 
\begin{align*}
M_n :\affc^{op}  & \lmo Cat \\
\spec(A) & \longmapsto M_n(A)=\parf(T,A)^{[n-1]}.
\end{align*}
C'est le préfaisceau en catégories des suites de $(n-1)$ morphismes composables dans $\parf(T,-)$. Un morphisme de $a_1\mo \cdots \mo a_n$ vers $b_1\mo \cdots \mo b_n$ dans $\parf(T,A)^{[n-1]}$ est par définition la donnée de carrés commutatifs dans $\parf(T,A)$, 
$$\xymatrix{a_1\ar[r] \ar[d]  & a_2 \ar[r] \ar[d] & \cdots \ar[r] & a_n \ar[d]  \\ b_1\ar[r]  & b_2 \ar[r] & \cdots \ar[r] & b_n}$$
On a $M_1=\parf(T,-)$. Pour tout $n\geq 1$ et tout $A\in \calgc$, la catégorie $M_n(A)$ est naturellement munie d'une structure de modèles projective. Soit $X_n=NwM_n$ le préfaisceau simplicial qui classifie les suites de $(n-1)$ morphismes composables à quasi-isomorphisme près. Pour tout $n\geq 1$, on a une inclusion naturelle, 
$$\K^T_n \hookrightarrow X_n.$$
Le fait que tout morphisme de $\parf(T,A)$ se factorise par une cofibration suivi d'un quasi-isomorphisme implique que ce dernier morphisme est une équivalence globale dans $\spr$. Il suffit donc de prouver que le morphisme composé noté encore $\lambda_n : \M^T_n \lmo X_n$ est une $\ao$-équivalence pour tout $n\geq 1$. L'avantage maintenant est que l'on se moque de savoir si les morphismes considérés sont des cofibrations ou non. On procède par récurrence sur $n$. On a besoin du niveau $2$ et $n-1$ pour traiter le niveau $n$. 
Au niveau $1$ on a un isomorphisme naturel $\M^T_1 =X_1=Nw\parf(T,-)$. Au niveau $2$ le morphisme $\lambda_2$ agit sur les $0$-simplexes par 
$$\lambda_2 (a,b)=(a\mo a\oplus b).$$
On définit alors un inverse $\ao$-homotopique explicite à $\lambda_2$ noté $\mu_2$ défini sur les $0$-simplexes par 
$$\mu_2(i:x\mo y):=(x,\cone(i)).$$
On définit $\mu_2$ sur les morphismes de manière évidente et ceci nous donne un morphisme de préfaisceaux simpliciaux $\mu_2 :  X_2\lmo \M^T_2$. On a alors
$$\mu_2 \circ \lambda_2 (a,b)=\mu_2(a\mo a\oplus b)=(a,\cone(a\mo a\oplus b))\simeq (a,b),$$
où le dernier morphisme est un quasi-isomorphisme. On a donc pour chaque $A\in \calgc$ une homotopie $\mu_2 \circ \lambda_2\Rightarrow id$ entre endomorphisme de $\M^T_2(A)$. Dans l'autre direction on a 
$$\lambda_2\circ \mu_2(i:x\mo y)=\lambda_2(x,\cone(i))=(x\mo x\oplus \cone(i)).$$
On définit alors une $\ao$-homotopie $h:\ao\times X_2\lmo X_2$ sur toute algèbre $A\in \calgc$ par
\begin{align*}
h_A :  A\times X_2(A)  & \lmo X_2(A) \\
(f,i:x\mo y) & \longmapsto (f i:x\mo y).
\end{align*}
Le morphisme $h$ est une $\ao$-homotopie entre $id_{X_2}$ et l'endomorphisme $Z$ de $X_2$ défini par 
$$Z(i:x\mo y)=(0:x\mo y).$$
L'endomorphisme $Z$ est en fait conjugué par une autoéquivalence (globale) de $X_2$ au morphisme $\lambda_2\circ \mu_2$. Cette auto-équivalence est donné par le décalage des morphismes
\begin{align*}
t : X_2 & \lmo X_2\\
 (i:x\mo y)  & \longmapsto (y\mo \cone(i)),
\end{align*}
Le morphisme $t$ vérifie $t^{(3)} (i)=i[1]$, c'est donc une autoéquivalence de $X_2$. L'inverse de $t$ est donné par
$$t^{-1} (i:x\mo y)=\cocone(i)\mo x, $$
où le dernier morphisme est donné par la définition du cocône. On a alors 
\begin{align*}
tZt^{-1} (f) & = tZ(\cocone(i) \mo x) \\
& = t(0: \cocone(i)\mo x )\\
&= x\mo \cone(0:\cocone(i) \mo x). \\
\end{align*}
Le dg-module $\cone(0:\cocone(i) \mo x)$ est canoniquement quasi-isomorphe à $x\oplus \cone(i)$ avec la différentielle somme. On en déduit l'existence d'un quasi-isomorphisme $tZt^{-1} \simeq \lambda_2\circ \mu_2$. Pour résumer, le morphisme $h$ est une $\ao$-homotopie $id_{X_2}\Rightarrow Z$, on a $tZt^{-1} \simeq \lambda_2\circ \mu_2$ et une homotopie $\mu_2 \circ \lambda_2\Rightarrow id$ ce qui implique que $\lambda_2$ est une $\ao$-équivalence. 

Soit maintenant $n\geq 2$. On utilise la notion de produit fibrés de catégories de modèles comme définie dans \cite{dhall}. On considère le foncteur 
\begin{align*} 
F:M^{(n)}  & \lmo M^{(n-1)} \underset{M^{(1)}}{\ph} M^{(2)} \\
(a_1\mo a_2\mo \cdots\mo a_n) & \longmapsto ((a_1\mo \cdots\mo a_{n-2} \mo a_n), (a_{n-1}/a_{n-2} \mo a_n/a_{n-2}), a_{n-1}/a_{n-2}, id, id)
\end{align*}
où $a_{n-2} \mo a_n$ est le morphisme composé $a_{n-2} \mo a_{n-1} \mo a_n$, et la notation quotient désigne la cofibre homotopique (ou le cone). 
\begin{lem}\label{lemrec}
Le foncteur $F$ vérifie les deux hypothèse de \cite[Lem. 4.2]{dhall}. On en déduit que le morphisme induit par $F$,  
$$q_n:X_n \lmo X_n \underset{X_1}{\ph} X_2$$
est une équivalence globale dans $\spr$ (le produit fibré homotopique étant calculé dans la structure de modèles globale $\spr$).  
\end{lem}

\begin{proof}  du lemme \ref{lemrec} --- La preuve par B. Toën du fait que $q_3$ est une équivalence (\cite{dhall} juste après la preuve du lemme 4.2) se généralise directement aux suites de morphismes de longueur arbitraire\footnote{On note qu'il y a un décalage d'indice entre nos notations et le papier \cite{dhall}.}. La différence majeure est que l'on travaille avec un préfaisceau de catégories d'objets parfaits et non pas avec une catégorie de modèles stable. La même preuve fait sens avec les objets parfaits. De plus, \cite[Lem. 4.2]{dhall} peut s'appliquer niveau par niveau, et dans la structure de modèles globale sur $\spr$, un carré est homotopiquement cartésien si et seulement si il est niveau par niveau homotopiquement cartésien dans $SSet$. 
\end{proof}

On a donc un carré de préfaisceaux simpliciaux,
$$\xymatrix@R=1cm @C=3cm { \M^T_n \ar[r]^-{\lambda_n } \ar[d]^-{p} & X_n \ar[d]^-{q_n} \\ \M^T_{n-1} \ph \M^T_{2} \ar[r]^-{\lambda_n \ph  \lambda_2 } & X_{n-1} \underset{X_1}{\ph} X_1 }$$
où $p$ est le morphisme $p(a_1, \hdots, a_n)=((a_1,\hdots, a_{n-2},a_{n-1}\oplus a_n), (a_{n-1}, a_n))$. Ce dernier est une équivalence par sa définition même. Ce carré est commutatif à homotopie (globale) près. Les morphismes $\lambda_{n-1}$ and $\lambda_2$ sont des $\ao$-équivalences d'homotopie par hypothèse de récurrence et par ce qui précède. Les $\ao$-équivalences d'homotopie étant stable par produit fibrés homotopiques, le morphisme $\lambda_2 \ph  \lambda_2$ est une équivalence d'homotopie. On conclut donc par la propriété '2 parmi 3' que le morphisme $\lambda_n$ est une $\ao$-équivalence d'homotopie. Ceci achève la preuve de la proposition \ref{mk}. 
\end{proof}

\subsubsection{Lien avec les champs de dg-modules parfaits et pseudo-parfaits}\label{kstchamp}

\begin{prop}\label{mttsp}
Soit $T\in \dgcatc$ une dg-catégorie sur $\C$. Alors le $\gam$-espace spécial $\regam{\M^T_\bul}$ est très spécial. 
\end{prop} 

\begin{proof} 
On utilise la formule \ref{pi0cor}. On a un isomorphisme d'ensembles $\pi_0\re{\M^T_1}\simeq \pi_0\M^T(\C)/\sim$ où deux classes de dg-modules parfaits $[E]$ et $[E']$ sont équivalentes s'il existe une courbe algébrique connexe qui les relie dans $\M^T(\C)$. Soit $E$ un $T^{op}$-dg-module parfait. Soit 
$$\delta : \ao\lmo \M^T$$
le morphisme tel que pour tout $A\in \calgc$, 
$$\delta_A(f)=\cone(E\lmos{\times f} E).$$
Alors on a $\delta_A(0)=\cone(0:E\lmo E)=E\oplus E[1]$ et $\delta_A(1)=\cone(id_E)$ qui canoniquement quasi-isomorphe à $0$. On a donc montré que l'identité $[E\oplus E[1]]=[0]$ a lieu dans le monoïde $\pi_0\re{\M^T_1}$, c'est donc un groupe. 
\end{proof}

\begin{theo}\label{kstmt}
Soit $T\in \dgcatc$. Alors il existe un isomorphisme naturel,
$$\kcst(T)\simeq \mcal{B} \regam{\M^T_\bul}$$ 
dans $Ho(Sp)$. 
\end{theo} 

\begin{proof} 
On a une chaîne d'équivalences 
$$\kcst(T) = \resp{\ukc(T)} =\resp{\mcal{B} K^\gam(\parf(T,-))} \simeq \mcal{B}\regam{K^\gam(\parf(T,-))}$$ 
où la dernière équivalence vient de la proposition \ref{proprelb}. On note $K^\gam(\parf(T,-))=:K^\gam (T,-)$. On considère le morphisme de $\gam$-préfaisceaux simpliciaux, 
$$\sigma : \M^T_\bul \lmo K^\gam(T,-)$$
qui est un cas particulier du morphisme général (\ref{morphgam}) défini en \ref{delass}. On veut alors montrer que $\sigma$ induit une équivalence sur les réalisations topologiques dérivées. Comme ce sont des $\gam$-objets spéciaux, il suffit de vérifier que c'est une équivalence au niveau $1$. On a alors un diagramme commutatif dans $Ho(SSet)$, où l'on fait l'abus de notation en omettant les indices $\del$ et $\gam$ du foncteur de réalisation,

$$\xymatrix{\re{\M^T_1} \ar[r]^-{\re{\sigma}_1} \ar[dd]_-\wr &\re{K^\gam(T,-)_1} \ar[rd]^-\sim \\ & &  \re{K(T,-)} \\ \re{\M^T_\bul}^+_1  \ar[r]^-{\re{\lambda}^+_1} & \re{\K^T_\bul}^+_1 \ar[ur]^-\sim}$$
où le morphisme de $\del$-objets $\lambda$ est celui de la proposition \ref{mk}. En vertu de cette même proposition et du théorème \ref{di}, le morphisme induit $\re{\lambda}^+_1$ est une équivalence dans $SSet$. Le morphisme vertical gauche est une équivalence par la proposition \ref{mttsp}. D'autre part les $\del$-objets $\re{K^\gam(T,-)}$ et $\re{\K^T_\bul}^+$ ont même niveau $1$, qui est équivalent à $\re{K(T,-)}$. On déduit que le morphisme $\re{\sigma}_1$ est une équivalence dans $SSet$ et donc que $\re{\sigma}$ est une équivalence. Par conséquent on a un isomorphisme naturel $\kcst(T)\simeq \mcal{B}\regam{\M^T_\bul }$ dans $Ho(Sp)$. 
\end{proof}

\begin{theo}\label{mtps}
Soit $T\in \dgcatc$. Alors le $\gam$-espace spécial $\regam{\M_T^\bul}$ est très spécial et il existe un isomorphisme naturel, 
$$\resp{\kc(\upspa(T))}\simeq \mcal{B} \regam{\M_T^\bul}.$$
\end{theo} 

\begin{proof} 
Les preuves de \ref{mk}, \ref{mttsp}, et \ref{kstmt} fonctionnent si l'on remplace les parfaits par les pseudo-parfaits. 
\end{proof}

\subsection{Caractère de Chern topologique}

Dans cette partie on donne la construction du caractère de Chern topologique. Soit $T\in\dgcatc$ une $\C$-dg-catégorie. 
On rappelle qu'on a défini en \ref{defchernlin}, pour toute $\C$-dg-catégorie $T$, un morphisme de Chern algébrique de préfaisceaux de spectres en $\ka$-modules, 
$$\ch_T : \ukn(T)\lmo \uhcn(T). $$
En composant ce morphisme avec le morphisme canonique de $\ka$-modules $\uhcn(T)\lmo \uhp(T)$, on obtient un morphisme de $\ka$-modules, 
$$\ukn(T)\lmo \uhp(T).$$
On applique alors le foncteur de réalisation topologique spectrale (qui est un foncteur monoïdal) pour obtenir un morphisme de $\bu$-modules, 
$$\kst(T)=\resp{\ukn(T)}\lmo \resp{\uhp(T)}.$$
On utilise maintenant une formule de type Künneth pour l'homologie périodique. Le préfaisceau $\uhp(T)$ est donné par $\spec(A)\longmapsto \hp(T\tel_\C A)$. Le théorème de Kassel \cite[Thm 2.3]{kasselcyclic} ainsi que \cite[Prop 2.4]{kasselcyclic} implique l'existence pour toute $\C$-algèbre commutative lisse $A$ d'un morphisme naturel, 
$$\hp(T)\sml_{H\cuu}\hp(A)\lmo \hp(T\tel_\C A), $$
qui est une équivalence dans $Sp$. Ceci implique que le morphisme de préfaisceaux de spectres 
$$\hp(T)\sml_{H\cuu} \uhp(\ast) \lmo \uhp(T)$$
est une équivalence sur les schémas affines lisses dans $\spaf$. Le théorème \ref{restliss} implique alors que le morphisme induit sur les réalisations topologiques spectrales, 
$$\hp(T)\sml\resp{\uhp(\ast)} \simeq\resp{\hp(T)\sml_{H\cuu} \uhp(\ast)} \lmo \resp{\uhp(T)}$$
est une équivalence dans $Sp$. On en déduit un isomorphisme 
$$\resp{\uhp(T)}\simeq \hp(T)\sml_{H\cuu} \resp{\uhp(\ast)}$$
dans $Ho(Sp)$. En composant on a donc un morphisme 
\begin{equation}\label{mor1}
\kst(T)\lmo \hp(T)\sml_{H\cuu}\resp{\uhp(\ast)}
\end{equation}
qui est une transformation naturelle entre objet de $Ho(Sp^{\dgcatc})$. Par conséquent pour obtenir un morphisme à valeurs dans $\hp(T)$, on doit choisir un morphisme $\resp{\uhp(\ast)}\lmo H\cuu$. Par adjonction cela revient à choisir un morphisme $\uhp(\ast)\lmo (H\cuu)_{\s,B}$. Le préfaisceau de spectres $(H\cuu)_{\s,B}$ est donné par 
$$X\longmapsto \rhomi_{Ho(Sp)} (\resp{X}, H\cuu) \simeq \rhomi_{Ho(Sp)} (\resp{X}, H\C)\sml_{H\C} H\cuu,$$
qui est la cohomologie de Betti $2$-périodique du schéma affine $X$. On note $\hbuu$ ce préfaisceau. D'autre part on note $\hb$ la cohomologie de Betti usuelle, c'est-à-dire la préfaisceau de spectres en $H\C$-modules
$$X\longmapsto \rhomi_{Ho(Sp)} (\resp{X}, H\C),$$
dont les groupes d'homotopie stables sont les espaces vectoriels de cohomologie de Betti. On note $\hpa$ le préfaisceau $\uhp(\ast) : \spec(A)\mapsto \hp(A)$. On cherche donc un morphisme $\hpa\lmo \hbuu$. On considère le morphisme d'antisymétrisation standard, 
$$\hpa\lmo \hdrna$$
qui va de l'homologie périodique vers la cohomologie de de Rham naïve. Par cohomologie de de Rham naïve on entend le préfaisceau en spectres $X\mapsto \hdrna(X)$, tel que $\hdrna(X)$ est le spectre associé au complexe de de Rham algébrique de $X$, c'est-à-dire le complexe de $\C$-espace vectoriels des formes différentielles algébriques partout définies sur le schéma affine $X$. On note $\hdran$ l'analogue de $\hdrna$ construit à partir des formes analytiques. L'inclusion des formes algébriques dans les formes analytiques induit un morphisme $\hdrna\lmo \hdran$. Le morphisme évident de $\C$ dans le complexe de de Rham analytique induit un morphisme $\hb\lmo \hdran$ qui est une équivalence sur les schémas affines lisses\footnote{Ceci par le fait classique que pour une variété complexe lisse, le complexe de faisceaux des formes analytiques est une résolution injective du faisceau constant $\und{\C}$.}. On a donc des morphismes de préfaisceaux en spectres,
$$\hpa\lmo \hdrna \lmo \hdran\longleftarrow \hb,$$
et le dernier est une équivalence sur les lisses\footnote{En fait les deux autres morphismes de la chaîne sont aussi des équivalences sur les lisses par le théorème HKR et le théorème de Grothendieck respectivement. Mais nous n'avons pas besoin de ce fait.}.

\begin{nota}
On note $\sppro$ la structure de modèles propre locale sur la catégorie $\spaf$ des préfaisceaux en spectres symétriques. 
\end{nota}

Par la remarque \ref{hiro}, le morphisme $\hb \lmo \hdran$ est donc une équivalence propre locale, c'est à dire un isomorphisme dans $Ho(\sppro)$. On obtient donc un morphisme bien défini $\hpa\lmo \hb$ dans la catégorie $Ho(\sppro)$ et donc un morphisme 
$$\hpa\lmo \hbuu$$
dans $Ho(\sppro)$. On rappelle alors la proposition \ref{quipro} selon laquelle le foncteur de réalisation topologique reste de Quillen à gauche pour la stucture propre locale. Elle se généralise directement à la réalisation spectrale et on a une paire de Quillen, 
$$\xymatrix{\sppro \ar@<3pt>[r]^-{\resp{-}} & Sp \ar@<3pt>[l]^-{H_{\s,B}}  }.$$
Le morphisme $\hpa\lmo \hbuu = \hbs (H\cuu)$ donne par adjonction le morphisme annoncé
$$\mcal{P} : \resp{\hpa}=\resp{\uhp(\ast)}\lmo H\cuu$$
dans $Ho(Sp)$. En composant le morphisme (\ref{mor1}) avec ce qu'on vient d'établir on obtient donc un morphisme
$$\ch^{\mrm{st}}_T : \kst(T)\lmo \hp(T)$$
défini comme le composé
$$\xymatrix{  \kst(T)\ar[d]_-{\resp{\ch_T}}  \ar[rr]^-{\ch^{\mrm{st}}_T} &&  \hp(T) \\
 \resp{\uhp(T)} \ar[r]^-\sim & \hp(T)\sml_{H\cuu} \resp{\uhp(\ast)} \ar[r]^-{id\sml \mcal{P}} & \hp(T)\sml_{H\cuu} H\cuu \ar[u]_-\wr }$$
Par les considérations faîtes au début de \ref{kstktop}, la réalisation topologique spectrale peut se restreindre en un foncteur de Quillen à gauche sur les catégories de modules, 
$$\resp{-} : \ka-Mod_\s\lmo \bu-Mod_\s.$$
On en déduit que tous les morphismes en jeu dans le rectangle précédent sont des morphismes de $\bu$-modules et qu'on obtient de cette manière un morphisme 
$$\chst : \kst\lmo \hp$$
dans la catégorie $Ho(\bu-Mod_\s^{\dgcatc})$. Par abus de notations on notera $\chst : \kst(T)\lmo \hp(T)$ en omettant l'indice $T$ dans la notation. On note que pour toute $\C$-dg-catégorie $T$ on a un carré commutatif 
\begin{equation}\label{carst}
\xymatrix{\kn(T) \ar[r]^-{\ch} \ar[d]_-\eta & \hp(T) \ar[d]^-{id} \\ \kst(T) \ar[r]^-{\chst}  & \hp(T)  }
\end{equation} 
dans $Ho(\bu-Mod_\s)$, où $\eta$ est le morphisme naturel défini à la remarque \ref{fibtop} et le composé 
$$\xymatrix{\hp(T)\ar[r] & \resp{\uhp(T)}\ar[r] & \hp(T)\sml_{H\cuu} \resp{\uhp(\ast)} \ar[r]^-{id\sml \mcal{P}} &\hp(T)}$$
est égal à l'identité dans $\End_{Ho(Sp)}(\hp(T))$. 

Maintenant il suffit de vérifier que l'image $\chst(\beta)$ de la classe de Bott est inversible dans l'anneau $\hp(\ast)=H\cuu$. On utilise les notations \ref{choixbeta}. On suit alors la classe de Bott dans la K-théorie de $\po$. Comme cas particulier du carré (\ref{carst}), on a un carré commutatif de groupes abéliens
$$\xymatrix{\kn_0(\po) \ar[r]^-{\ch} \ar[d]_-{\eta} & \hp_0(\po) \ar[d]^-{id} \\ \kst_0(\po) \ar[r]^-{\chst} & \hp_0(\po) }$$
On rappelle que $\eta(\alpha)=\beta$ et on choisit par exemple le générateur $u\in\hp_0(\po)$ tel que $\ch(\alpha)=u$. On vérifie alors $\chst(\beta)=\chst(\eta(\alpha))=\ch(\alpha)=u$. 
Par la propriété universelle du localisé, on obtient pour toute $\C$-dg-catégorie $T$ un morphisme 
$$\ch^{\mrm{top}}_T : \ktop(T)\lmo \hp(T).$$
Ceci définit un morphisme 
$$\chtop : \ktop\lmo\hp$$
dans $Ho(\BU-Mod_\s^{\dgcatc})$.

\begin{theo}\label{carac}
Il existe un morphisme $\chtop : \ktop\lmo\hp$ appelé le \emph{caractère de Chern topologique} tel que le carré
$$\xymatrix{ \kn \ar[r]^-{ \ch} \ar[d] & \hcn\ar[d] \\  \ktop \ar[r]^-{ \chtop} &\hp  }$$
est commutatif dans $Ho(Sp^{\dgcatc})$. 
\end{theo}

\begin{proof} 
Découle directement du fait que le carré (\ref{carst}) est commutatif. 
\end{proof}

\subsection{Conjectures}\label{conj}

Les conjectures suivantes sont toutes des analogues non commutatifs de résultats connus pour les variétés algébriques propres et lisses. Elles sont donc appuyées par le slogan de l'introduction selon lequel les dg-catégories propres et lisses se comportent de la même façon que les variétés propres et lisses. Ce slogan se voit attribuer lui-même un sens précis lorsque l'on s'intéresse à la théorie motivique de chaque côté du foncteur $\lpe$. 

On énonce d'abord la conjecture reliée à la partie Betti de la structure de Hodge non-commutative hypothétique sur $\hp(T)$ pour $T$ propre et lisse. On rappelle qu'une dg-catégorie $T$ est propre si les complexes de morphismes de $T$ sont des complexes parfaits et si la catégorie triangulée $[\m{T}]$ a un générateur compact. Une dg-catégorie $T$ est lisse si le $T^{op}\tel T$-dg-module $(x,y)\mapsto T(x,y)$ est parfait. Il est montré dans \cite[Cor.2.13]{modob} qu'une dg-catégorie propre et lisse est de type fini, et possède donc un générateur compact. Une dg-catégorie propre et lisse est donc quasi-équivalente à une dg-algèbre propre et lisse. Ils suffient donc de montrer les conjectures suivantes pour $T$ une dg-algèbre propre et lisse. 

\begin{conj}\label{conjres} \emph{(dite conjecture du réseau)} --- 
Soit $T$ une dg-catégorie propre et lisse sur $\C$. Alors le morphisme 
$$\chtop\sm_\s H\C : \ktop(T)\sm_{\s} H\C \lmo \hp(T)$$
est une équivalence. 
\end{conj} 

On remarquera que la classe des dg-catégories qui vérifie cette conjecture est stable par un certain nombre d'opérations : colimites filtrantes, quotient, extension et rétract dans la catégorie homotopique $Ho(\dgmorc)$. 

La conjecture suivante est une généralisation du théorème \ref{annupoint} pour les dg-catégories propres et lisses.

\begin{conj} Soit $T$ une dg-catégorie propre et lisse sur $\C$. Alors $\kst_i(T)=0$ pour tout $i<0$. 
\end{conj}

Denis--Charles Cisinski a récemment montré l'annulation de la K-théorie algébrique négative des dg-algèbres propres et lisses dont la cohomologie est concentrée en degrés positifs ou nul, ceci en se basant sur la preuve par Schlichting de l'annulation de la K-théorie négative d'une catégorie abélienne noethérienne. 

La conjecture suivante est inspiré du résultat de Thomason cité dans l'introduction \cite[Thm.4.11]{thomet}. On rappelle qu'étant donné un spectre $E$ et un entier $n$, on peut, vu que la catégorie homotopique $Ho(Sp)$ est additive donner un sens à la réduction modulo $n$ du spectre $E$ noté $E/n$. On énonce cette conjecture avec $T$ quelconque -- précisons que l'on s'attend au plus à des hypothèses très générales sur $T$ (comme il m'a été suggéré par DC Cisinski). Cette conjecture est reliée à la propriété de $\ao$-invariance de la K-théorie algébrique à coefficients de torsion et au théorème de rigidité de Gabber. 

\begin{conj}\label{coefffinis}
Soit $T$ une dg-catégorie sur $\C$ et soit $n>0$ un entier. Alors le morphisme $\kn(T)/n\lmo \kst(T)/n$ est une équivalence. 
\end{conj}

\newpage

\section{Exemples de validité de la conjecture du réseau}

\subsection{Cas d'un $\C$-schéma lisse de type fini}\label{comp}

Le but est de comparer la K-théorie topologique d'un $\C$-schéma de type fini au sens dg avec la K-théorie topologique usuelle de ses points complexes. On compare aussi les deux caractères de Chern associés. 

\begin{nota}
Pour tout espace topologique $Y$, on note $\ktopu(Y):=\rhomi_{Ho(Sp)} (\sinf SY_+, \BU)$ le spectre de K-théorie topologique usuelle (non-connective) de $Y$. Si $X\in\schc$ est un $\C$-schéma de type fini on peut alors considérer sa K-théorie topologique des points complexes $\ktopu(sp(X))$. 
\end{nota}

\begin{nota}
Pour tout schéma $X\in\schc$, on note $\hp(X):=\hp(\lpe(X))$ son homologie périodique et $\hb(X)=\rhomi_{Ho(Sp)} (\resp{X}, H\C)$ sa cohomologie de Betti. Pour $X$ lisse, on a un isomorphisme $\hp(X)\lmo \hb(X)$ donné par la composé du morphisme d'antisymétrisation suivi de l'isomorphisme usuel entre cohomologie de de Rham et cohomologie de Betti. 
\end{nota}

\begin{nota}
Pour tout espace topologique $Y$, le caractère de Chern topologique usuel 
$$\chutop : \ktopu(Y) \lmo \hb(Y)$$ 
est définie par $\chutop=\rhomi_{Ho(Sp)}(\sinf SY_+,\ch^{\mrm{top}}_\unit))$, où $\ch^{\mrm{top}}_\unit:\BU\lmo H\cuu$ est le caractère de Chern topologique du point (Thm. \ref{carac}) qui on le rappelle est un morphisme d'anneaux en spectres. Ceci est justifié parce qu'on sait que $\ch^{\mrm{top}}_\unit(\beta)=u$ et que cette propriété caractérise uniquement le caractère de Chern topologique classique dans l'ensemble des classes d'homotopie de morphismes d'anneaux $[\BU,H\cuu]$. En effet ce dernier ensemble est en bijection avec l'ensemble $[\BU\sm H\C, H\cuu]$. L'équivalence standard entre les anneaux en spectres et les dg-algèbres transforme l'anneau en spectre $\BU\sm H\C$ en la dg-algèbre polynomiale $\C[\beta,\beta^{-1}]$ avec $\beta$ de degré $2$. L'image de $\beta$ ne peut donc être qu'un multiple non nul de $u$ ; le caractère de Chern classique correspondant à $\chutop(\beta)=u$. 
\end{nota}

\begin{prop}\label{compvarlisse}
Soit $X$ un $\C$-schéma lisse de type fini. Alors il existe un isomorphisme canonique 
$$t:\ktop(X)\lmos{\sim}\ktopu(sp(X))$$
dans $Ho(Sp)$. De plus le carré, 
\begin{equation}\label{utop}
\xymatrix{ \ktop(X) \ar[r]^-{\chtop} \ar[d]_-t & \hp(X) \ar[d]  \\ \ktopu(sp(X)) \ar[r]^-{\chutop} & \hbuu(X) }
\end{equation}
est commutatif dans $Ho(Sp)$. Ceci a pour conséquence que la conjecture du réseau est vrai pour une dg-catégorie de la forme $T=\lpe(X)$ avec $X$ un $\C$-schéma lisse de type fini. 
\end{prop} 

La preuve occupe la suite de cette partie. On commence d'abord par quelques notations et sorites. On rappelle (notation \ref{notaksch}) qu'on a un foncteur de K-théorie des schémas
$$\kn : \schc\lmo Sp.$$
On définit alors pour tout schéma $X$ un préfaisceau de K-théorie par 
$$\ukn(X):=(\spec(A)\longmapsto \kn(X\times \spec(A)).$$
Ceci définit un foncteur 
$$\ukn : \schc\lmo \spaf.$$
On suppose qu'on a un foncteur complexes parfaits
$$\lpe : \schc\lmo \dgcatc$$
tel que pour tout $X\in \schc$ et tout $A\in \calgc$, il existe un morphisme
$$\lpe(X)\tel_\C A\lmo \lpe(X\times \spec(A))$$
dans $\dgcatc$ qui est une équivalence Morita dérivée. On en déduit (prop. \ref{propkn}) une équivalence 
$$\kn(X\times \spec(A))\simeq \kn(\lpe(X)\tel_\C A)$$ 
dans $Sp$. Ceci se traduit par un isomorphisme de préfaisceaux
$$\ukn(X)\simeq \ukn(\lpe(X))$$
dans $Ho(\spaf)$. Vu le remarque \ref{restlisskst}, comme nous sommes intéressé ici par la K-théorie topologique, on peut restreindre nos préfaisceaux aux schémas affines lisses $\afflissc\hookrightarrow \affc$. Dans la suite de cette partie, on considère $\ukn(X)$ comme un objet de $\spafliss$.

On note $\splissna$ la structure de modèles $\ao$-Nisnevish locale sur la catégorie $\spafliss$. C'est une catégorie de modèles monoïdale et on note $\sm$ son produit. On note $\shc$ la catégorie homotopique stable des $\C$-schémas lisse de Morel--Voevodsky. On peut la définir de la manière suivante. La lettre $T:=S^1\sm\gm$ désignant la sphère de Tate dans $\splissna$, la catégorie $\shc$ est définie ici comme étant la catégorie homotopique des $T$-spectres symétriques dans $\splissna$,
$$\shc:=Ho(Sp_T \splissna).$$
C'est une catégorie monoïdale symétrique fermée. Nous allons utiliser le résultat de Riou qui dit que tout objet de $\shc$ de la forme $\sinf_{T, S^1} X_+$ pour $X$ un $\C$-schéma lisse de type fini est un objet (fortement) dualisable dans la catégorie monoïdale $\shc$ (voir \cite{rioudual}).

On choisit une présentation du schéma $\gm=\C\te_{\Z}\Z[t, t^{-1}]$. La fonction inversible $t$ donne une classe $b$ dans le groupe $\kn_1(\gm)$ et on choisit $t$ de telle sorte que le morphisme canonique
$$\kn_1(\gm)\lmo \kn_0(\po)\simeq \Z\oplus \alpha\Z$$ 
envoie $b$ sur la classe $\alpha$  qu'on a choisit précédemment. 
On note $\kh$ l'objet de $\shc$ dont le $S^1$-spectre sous-jacent est le préfaisceau $\kn \in \spafliss$, muni de la structure de $T$-spectre induite par le morphisme $T=S^1\sm \gm \lmo \kn$ qui correspond à la classe $b\in\kn_1(\gm)$. Plus précisément, pour tout $n\geq 0$ la classe $b$ détermine un morphisme
$$\xymatrix{T\sml \kn \ar[r]^-{b\sml id} &  \kn\sml\kn \ar[r] & \kn}$$
où le dernier morphisme est donné par la structure d'anneau de $\kn$ dans $Ho(\spafliss)$. En prenant des remplacements cofibrants-fibrants et en appliquant \cite[Proposition 2.3]{cisdes} on obtient un $T$-spectre $\kh$ avec $\kh(Y)_n=\kn(Y)$ pour tout $n\geq 0$. Un modèle fibrant de $\kh$ dans $\shc$ représente la K-théorie invariante par homotopie de Weibel, comme montré par Cisinski (\cite[Théorème 2.20]{cisdes}). Pour tout $\C$-schéma lisse de type fini $X$ on note 
$$\ukh(X):=\rhomi_{\shc} (\Sigma^\infty_{T, S^1} X_+ , \kh).$$
Par Yoneda dans $\shc$, celui-ci est globalement équivalent au préfaisceau de $T$-spectres
$$Y\longmapsto \kh(X\times Y),$$
dans $\spafliss$. On a une équivalence $\ukh(\unit)\simeq\kh$. La réalisation topologique de $T$ est donné par 
$$ssp(T)\simeq ssp(S^1\sm \gm) \simeq S^1\sm ssp(\gm)\simeq S^1\sm S^1=S^2.$$
On précise notre notation des $S^1$-spectres symétriques par $Sp=\sps$ et on note $\spss\sps$ la catégorie de modèles des $S^2$-spectres symétriques dans les $S^1$-spectres symétriques. Le foncteur espace de lacets infini $\Omega^\infty_{S^2} : \spss\sps\lmo \sps$ est une équivalence de Quillen. Il existe un foncteur de réalisation topologique au niveau de la catégorie des $T$-spectres $\shc$ ; il est simplement défini niveau par niveau. On le note 
$$\resh{-} : \shc\lmo Ho(\spss\sps).$$
On a un carré 
$$\xymatrix{ \shc \ar[r]^-{\resh{-}} \ar[d]^-{\R\Omega_T^\infty} & Ho(\spss\sps) \ar[d]^-{\R\Omega_{S^2}^\infty} \\ Ho(\splissna) \ar[r]^-{\resp{-} } & Ho(Sp_{S^1} )  }$$
où $\R\Omega_T^\infty$ désigne le foncteur dérivé d'espace de lacets infini relativement aux $T$-spectres. Ce carré n'est pas commutatif. La structure de $S^2$-spectre de $\resh{\ukh(X)}$ correspond à la multiplication par la classe de Bott $\beta\in \pi_2\resp{\ukn(\unit)}$ dans le $S^1$-spectre $\resp{\ukn(X)}$. Ceci car la structure de $T$-spectre de $\ukh(X)$ est donné par la multiplication par la classe $b$ et qu'on a de plus une équivalence $ssp(T)\simeq S^2$ tel que le triangle
$$\xymatrix{  \sinf(S^2)_+\ar[rd]^-{\beta} \ar[d]^-{\wr} & \\  \resp{T} \ar[r]^-{\re{b}} & \resp{\ukn(\unit)}\simeq \bu  }$$
est commutatif par le choix de $b$. Par conséquent, appliquer le foncteur $\R\Omega_{S^2}^\infty$ à $\resh{\ukh(X)}$ correspond à inverser l'action de $\beta$ dans le niveau $0$ qui est égal à $\resp{\ukn(X)}$. On a donc un morphisme canonique 
$$\resp{\ukn(X)}[\beta^{-1}] \lmo \R\Omega_{S^2}^\infty\resh{\ukh(X)},$$
qui est un isomorphisme dans $Ho(Sp_{S^1})$ car le $S^1$-spectre sous-jacent à $\ukh(X)$ est équivalent à $\ukn(X)$.

\begin{proof} de la proposition \ref{compvarlisse} ---
On a des isomorphismes naturels dans $Ho(\sps)$, 
\begin{align*}
\ktop(X) & =\resp{\ukn(X)}[\beta^{-1}] \\
& \simeq \R\Omega_{S^2}^\infty\resh{\ukh(X)} \\
& = \R\Omega_{S^2}^\infty\resh{\rhomi_{\shc} (\Sigma^\infty_{T, S^1} X_+ , \kh) }
\end{align*}

 D'après Riou \cite[Théorème 1.4+Théorème 2.2]{rioudual}, l'objet $\Sigma^\infty_{T, S^1} X_+$ est fortement dualisable dans $\shc$. Le foncteur $\resh{-}$ étant monoïdal, et comme un foncteur monoïdal commute au foncteur de dualité, on en déduit des isomorphismes naturels dans $Ho(\sps)$, 
\begin{align*}
\ktop(X) & \simeq \R\Omega_{S^2}^\infty\rhomi_{Ho(\spss\sps)} (\resh{\Sigma^\infty_{T, S^1} X_+} , \resh{\kh}) \\
& \simeq \R\Omega_{S^2}^\infty\rhomi_{Ho(\spss\sps)} (\Sigma^\infty_{S^2, S^1} \re{X}_+ , \resh{\kh}) \\
& \simeq \rhomi_{Ho(\sps)} (\Sigma^\infty_{S^1} \re{X}_+ , \R\Omega_{S^2}^\infty\resh{\kh}) \\
\end{align*}
On a en outre par ce qui précède des isomorphismes dans $Ho(\sps)$,
$$\R\Omega_{S^2}^\infty\resh{\kh}\simeq \ktop(\unit)=\BU.$$
On en déduit l'isomorphisme naturel cherché 
$$\ktop(X)\simeq \rhomi_{Ho(\sps)} (\Sigma^\infty_{S^1} \re{X}_+ , \BU)=\ktopu(sp(X)).$$

Il reste à montrer que les deux caractères de Chern coïncident à homotopie près. Le carré (\ref{utop}) se décompose en 
$$\xymatrix@=5pc{\ktop(X)\simeq \resp{\rhomi (\resp{X}, \kn)}[\beta^{-1}] \ar[r]^-{\resp{\rhomi(X,\ch)}} \ar[d] & \resp{\rhomi (\resp{X}, \hpa)} \ar[d] \\
\ktopu(sp(X))=\rhomi (\resp{X}, \resp{\kn}[\beta^{-1}])  \ar[r]^-{\rhomi(\resp{X}, \resp{\ch})} \ar@/_1pc/[dr]_-{\rhomi(\resp{X}, \chtop)} & \rhomi(\resp{X}, \resp{\hpa}) \ar[d]^-{\rhomi(\resp{X}, \mcal{P})} \\
 & \rhomi(\resp{X}, H\cuu)=\hbuu(X)  }$$ 
Dans ce diagramme, le carré supérieur est commutatif par fonctorialité de $\resp{-}$. De plus le triangle inférieur est commutatif car il est commutatif pour $X=\spec(\C)$ par définition du caractère de Chern topologique $\chtop$. 

La conjecture du réseau est alors vrai pour $T=\lpe(X)$ avec $X$ lisse sur $\C$ car c'est un fait connu qu'elle est vrai pour le caractère de Chern usuel $\chutop$ pour un espace qui le type d'homotopie d'un CW-complexe, ce qui est vérifié par les espaces topologiques du type $sp(X)$ pour tout schéma (lisse) de type fini $X$.  
\end{proof}

\subsection{Algèbres de dimension finie et le champ des modules de dimension finie}

L'anneau de base est le corps $\C$ des complexes. Par défaut, toutes les algèbres considérées dans cette sous-section sont des $\C$-algèbres unitaires et associatives. L'homologie périodique d'une telle algèbre est par convention l'homologie périodique relativement au corps de base $\C$. On note $\kctop(T)$ le spectre de K-théorie topologique <<pseudo-connective>> d'une $\C$-dg-catégorie $T$, c'est-à-dire que 
$$\kctop(T):=\kcst(T)[\beta^{-1}].$$
On souhaite donner une preuve du fait suivant concernant la K-théorie topologique d'une $\C$-algèbre de dimension finie (en tant que $\C$-espace vectoriel). 

\begin{prop}\label{cralgass}
Soit $B$ une $\C$-algèbre associative de dimension finie sur $\C$. Alors la conjecture du réseau \ref{conjres} est vraie pour $B$. Plus précisemment le morphisme canonique 
$$\chctop\sm_\s H\C : \kctop(B)\sm_\s H\C \lmo \hp(B)$$ 
est un isomorphisme dans $Ho(Sp)$. 
\end{prop}

\begin{nota}
On travaille sur le site étale $\affc$ des $\C$-schémas affines de type fini. On a des foncteurs de localisations,
$$Ho(\spr)\lmo Ho(\spret) \lmo Ho(\spretao).$$
Le mot champ étale désigne un objet de la catégorie homotopique $Ho(\spret)$. Cette catégorie homotopique est appelée la catégorie homotopique des champs étales ou simplement catégorie homotopique des champs. En \ref{prestru}, on a défini le foncteur classifiant 
$$\mrm{B}:Gr(\spr)\lmo \spr,$$
des objets en groupes stricts dans $\spr$. Le foncteur champ classifiant est le foncteur 
$$\bcl : Gr(\spret) \lmo \spret,$$
défini par $\bcl(G)=\und{a} \mrm{B}G$ où $\und{a}$ est un foncteur de remplacement fibrant dans la catégorie de modèles $\spret$. Pour tout préfaisceau en groupe $G$ (constant dans la direction simpliciale) et tout $X\in \affc$, l'espace $\bcl G(X)$ est équivalent au nerf du groupoïde des $G$-torseurs sur $X$. Dans la suite on appelera champ algébrique un champ $1$-géométrique au sens de \cite[1.3.3.1]{hag2} pour le contexte standard du site étale $\affc$ et la classe de morphismes est celle des morphismes lisses. Il s'agit donc de champs algébriques d'Artin.  
On note toujours $\afflissc\hookrightarrow \affc$ l'inclusion des schémas lisses et $l^*:\spr\lmo \sprliss$ la restriction, de même que $l^*:\spaf\lmo \spafliss$ la restriction sur les préfaisceaux en spectres. 
\end{nota}

\paragraph{Différents champs associés à une algèbre de dimension finie.} 

Soit $B$ une $\C$-algèbre associative de dimension finie. On considère quatre champs associés à $B$ organisés dans le diagramme suivant. 

$$\xymatrix{ \vect^B \ar@{^{(}->}[r]^-{ } \ar@{^{(}->}[d]^-{ } & \M^B \ar@{^{(}->}[d]^-{ } \\ \vect_B \ar@{^{(}->}[r]^-{ } & \M_B }$$

Ces champs sont définis de la manière suivante. On définit d'abord quatre préfaisceaux à valeurs dans la catégorie $WCat$ des catégories de Waldhausen $\V$-petites. On obtient de tels foncteurs stricts en appliquant la strictification canonique aux pseudo-foncteurs définis. 
\begin{itemize}
\item $\uproj(B) : \spec(A)\longmapsto \proj(B\te_\C A)$, où $\proj(B\te_\C A)$ est la catégorie de Waldhausen des  $B\te_\C A$-modules à droite qui sont projectifs de type fini relativement à $B\te_\C A$.
\item $\upsproj(B) : \spec(A)\longmapsto \psproj(B\te_\C A)$, où $\psproj(B\te_\C A)$ est la catégorie de Waldhausen des $B\te_\C A$-modules à droite qui sont projectifs de type fini relativement à $A$ (c'est-à-dire en tant que $A$-modules).
\item $\uparf(B) : \spec(A)\longmapsto \parf(B\te_\C A)$, où $\parf(B\te_\C A)$ est la catégorie de Waldhausen des complexes cofibrants de $B\te_\C A$-modules à droite qui sont parfaits relativement à $B\te_\C A$. 
\item $\upspa(B) : \spec(A)\longmapsto \pspa(B\te_\C A)$, où $\pspa(B\te_\C A)$ est la catégorie de Waldhausen des complexes cofibrants de $B\te_\C A$-modules à droite qui sont parfaits relativement à $A$ (c'est-à-dire en tant que complexe de $A$-modules). 
\end{itemize}

Si $\mcal{C}$ est une catégorie avec une notion d'équivalence on note $NwC$ le nerf de la sous-catégorie des équivalences dans $\mcal{C}$. On a par définition, 
\begin{itemize}
\item $\vect^B=Nw\uproj(B) : \spec(A) \longmapsto Nw\proj(B\te_\C A)$. 
\item $\vect_B=Nw\upsproj(B) : \spec(A) \longmapsto Nw\psproj(B\te_\C A)$. 
\item $\M^B=Nw\uparf(B): \spec(A) \longmapsto Nw\parf(B\te_\C A)$. 
\item $\M_B=Nw\upspa(B) : \spec(A) \longmapsto Nw\pspa(B\te_\C A)$. 
\end{itemize}
Comme on travaille sur un corps, $B$ est localement cofibrante en tant que dg-catégorie et c'est pourquoi les produits tensoriels écrits précédemment sont en fait dérivés. Le champ $\M_B$ est le champ étudié par Toën--Vaquié dans \cite{modob}. On remarque que comme $B$ est de dimension finie, elle est propre en tant que dg-catégorie (au sens de \cite[Déf.2.4]{modob}). D'après \cite[Lem.2.8,1)]{modob}, un complexe parfait relativement à $B\te_\C A$ est donc parfait relativement à $A$. On en déduit les monomorphismes $\vect^B \hookrightarrow \vect_B$ et $\M^B\hookrightarrow \M_B$. Les monomorphismes $\vect^B\hookrightarrow \M^B$ et $\vect_B\hookrightarrow \M_B$ sont les inclusions des complexes concentrés en degré $0$. Ces champs ont une structure supplémentaire de monoïdes commutatifs à homotopie près donnée par la somme directe des dg-modules. On applique le foncteur $B_W$ (défini en \ref{gamstru}) et on obtient des $\gam$-objets spéciaux dans $\spr$, $\vect^B_\bul$, $\vect_B^\bul$, $\M^B_\bul$, $\M_B^\bul$ sont le niveau $1$ est respectivement $\vect^B,\vect_B,\M^B,\M_B$. 
\\

Le champ $\vect^B$ est un champ algébrique lisse localement de présentation finie (voir par exemple \cite[1.]{modob}). Il admet de plus la description suivante en termes de gerbes résiduelles aux points globaux ; il existe un isomorphisme dans $Ho(\spret)$, 
\begin{equation}\label{formvect}
\vect^B \simeq \coprod_{M\in \chi} \bcl \uaut(M)
\end{equation}
où $\chi=\pi_0\vect^B(\C)$ est l'ensemble des classes d'isomorphisme de $B$-modules projectifs de type fini, pour tout $M\in \chi$, $\uaut(M)$ est le schéma en groupe des $B$-automorphismes de $M$. Cette formule peut être dérivée d'une formule générale valable pour tout champ algébrique localement de présentation finie sur $\C$ dont le complexe tangent en chaque point vérifie une propriété de finitude. 

\begin{prop}\label{formgerb}
Soit $F$ un champ algébrique localement de présentation finie et lisse sur $\C$ tel que pour tout point global $E\in \pi_0F(\C)$ le complexe tangent $\tc_E F$ vérifie $\mrm{H}^0(\tc_E F)=0$. Alors le morphisme canonique de champs algébriques,
$$\coprod_{E\in \pi_0F(\C)} \mcal{G}_E\lmo F$$
est une équivalence de champs, où $\mcal{G}_E$ désigne la gerbe résiduelle de $F$ au point $E$. 
\end{prop} 

\begin{proof} 
On vérifie directement que ce morphisme est localement de présentation finie. C'est un monorphisme de champs algébriques, et donc un morphisme représentable. La condition imposée sur le complexe tangent implique que ce morphisme induit une équivalence sur les complexes tangents et est donc un morphisme étale de champs algébriques. C'est donc une immersion ouverte de champs algébriques. C'est un épimorphisme sur le $\pi_0$ étant donné qu'il est surjectif sur les points complexes du $\pi_0$.  C'est donc un épimorphisme de champs. On en déduit que c'est une équivalence. 
\end{proof}

Le champ $\vect^B$ vérifie l'hypothèse de la proposition \ref{formgerb} vu que pour tout $M\in \chi$, on a un quasi-isomorphisme $\tc_M \vect^B\simeq \uend(M)[1]$ où $\uend(M)$ désigne ici le $\C$-espace vectoriel des endomorphismes de $M$. Ceci nous donne donc la formule (\ref{formvect}). 

\paragraph{K-théorie topologique d'une algèbre de dimension finie.}

Rappelons qu'en \ref{kstchamp}, on a montré que la K-théorie semi-topologique connective de $T$ est décrite par la réalisation topologique du champ $\M^T$ muni de sa structure de $\gam$-objet. En d'autres termes on a une équivalences naturelle en $T$, 
$$\kcst(T)\simeq \mcal{B}\regam{\M^T_\bul}.$$
Si maintenant $T=B$ est une algèbre associative, d'après la remarque \ref{algebra} la K-théorie algébrique connective peut être calculée avec le préfaisceau $\uproj(B)$ muni de sa structure de $\gam$-objet spécial induit par la somme directe des modules. En suivant les même étapes que la preuve de la proposition \ref{bu}, on a des équivalences naturelles de spectres,
\begin{equation}\label{kvect}
\kcst(B)\simeq \resp{\ukc(B)} \simeq \resp{\mcal{B} K^\gam (\uproj(B))}\simeq \mcal{B} \regam{K^\gam (\uproj(B))}\simeq \mcal{B} \regam{\vect^B_\bul}^+.
\end{equation}
Et donc des équivalences naturelles de spectres et de $\gam$-espaces respectivement, 
\begin{equation}\label{kcsttot}
\kcst(B)\simeq \mcal{B}\regam{\M^B_\bul}\simeq \mcal{B} \regam{\vect^B_\bul}^+,
\end{equation}
\begin{equation}\label{mbvectb}
\regam{\M^B_\bul}\simeq \regam{\vect^B_\bul}^+.
\end{equation}

\begin{proof} de la proposition \ref{cralgass}. ---
Si l'algèbre $B$ est semi-simple, alors $B$ vérifie la proposition \ref{cralgass} pour les raisons suivantes. En effet, par la théorie des représentations des algèbres semi-simples, $B$ est Morita-équivalente à un produit de copies de l'algèbre unité $\unit$ (c'est-à-dire de copie de $\C$). Il existe donc une équivalence Morita $B\lmos{\sim} \prod_{J} \unit$ dans $\dgcatc$ avec $J$ est un ensemble fini. La K-théorie topologique pseudo-connective et l'homologie périodique commutent tous deux aux produits finis, et on a des isomorphismes $\kctop(B)\lmos{\sim} \prod_{J} \kctop(\unit) \simeq \prod_{J} \BU$ et $\hp(B)\lmos{\sim} \prod_J \hp(\unit)\simeq \prod_J H\cuu$ dans $Ho(Sp)$. On a alors un carré commutatif dans $Ho(Sp)$, 
$$\xymatrix{ \kctop(B) \ar[r]^-{\chctop} \ar[d]^-{\wr} & \hp(B) \ar[d]^-{\wr}  \\ \prod_J \BU \ar[r]^-{\prod_J \chctop} & \prod_J H\cuu }$$
Le morphisme du bas étant un isomorphisme après $\sm_\s H\C$ on en déduit que le morphisme du haut est un isomorphisme après $\sm_\s H\C$. 

Si maintenant $B$ est une algèbre qui n'est pas semi-simple, son radical $rad(B)$ est un idéal bilatère nilpotent de $B$. On note $B_0=B/rad(B)$ l'algèbre semi-simplifiée de $B$. Pour montrer que $B$ vérifie la proposition \ref{cralgass} il suffit donc de montrer que le morphisme induit par le changement de base $\vect^B\lmo \vect^{B_0}$ est une $\ao$-équivalence dans $\spretao$. En effet une telle $\ao$-équivalence implique que le morphisme de $\gam$-objet spéciaux $\vect^B_\bul\lmo \vect^{B_0}_\bul$ est un isomorphisme dans $Ho(\gam-\spretao)$. On en déduit en prenant la réalisation topologique dérivée que le morphisme $B\lmo B_0$ induit un isomorphisme $\kctop(B)\simeq \kctop(B_0)$ dans $Ho(Sp)$. D'autre part Goodwillie a montré l'invariance de l'homologie périodique par extension infinitésimale (\cite[Thm.2.5.1]{goodwilliecyclic}) ; le morphisme $B\lmo B_0$ induit donc un isomorphisme $\hp(B)\simeq \hp(B_0)$ dans $Ho(Sp)$. On se ramène alors au cas d'une algèbre semi-simple. 

Il suffit donc de montrer que pour toute $\C$-algèbre $B$ de dimension finie et tout idéal bilatère nilpotent $I\subseteq B$ le morphisme de projection $B\lmo B/I$ induit une $\ao$-équivalence $\vect^B\lmo \vect^{B/I}$. Par une récurrence classique sur le degré de nilpotence de $I$ on se ramène à montrer la dernière assertion pour un idéal bilatère de carré nul. En effet supposons que pour toute algèbre $C$ et tout idéal bilatère nilpotent $J$, le morphisme de projection $C\lmo C/J$ induit une $\ao$-équivalence $\vect^C\lmo \vect^{C/J}$. Si $I$ est un idéal bilatère de $B$ de degré de nilpotence $n\geq 2$, on considère le triangle commutatif d'algèbres suivant. 
$$\xymatrix{B\ar[r] \ar[rd] & B/I^2=C \ar[d] \\ & B/I\simeq C/I   }$$
Comme $I^2$ est un idéal bilatère de $B$ de degré de nilpotence $\leq n-1$ et que l'idéal bilatère $I.C\subseteq C$ engendré par l'image de $I$ dans $C$ vérifie $I.C^2=0$ les flèches $B\lmo C$ et $C\lmo C/I$ induisent des $\ao$-équivalences sur $\vect$. On en déduit que $B\lmo B/I$ induit aussi une $\ao$-équivalence sur $\vect$. 

Soit maintenant $I\subseteq B$ un idéal de carré nul dans une $\C$-algèbre $B$ de dimension finie. Il s'agit de montrer que $B\lmo B/I=B_0$ induit une $\ao$-équivalence sur $\vect$. En utilisant la formule (\ref{formvect}), le morphisme $\vect^B\lmo \vect^{B_0}$ est équivalent au morphisme,
$$\coprod_{M\in \chi} \bcl \uaut(M)\lmo \coprod_{M\in \chi_0} \bcl \uaut(M_0),$$
où $\chi=\pi_0\vect^B(\C)$, $\chi_0=\pi_0\vect^{B_0}(\C)$, $M_0=M\te_B B_0$ pour tout $M\in \chi$. 
Tout d'abord, on observe que les ensembles $\chi$ et $\chi_0$ sont en bijection comme montrer dans \cite[Prop.2.12]{bass} (les hypothèses de complétude sont vérifiées puisque $I^2=0$). Comme les $\ao$-équivalences sont stables par somme quelconque dans $\spretao$, il s'agit donc de montrer que pour tout $M\in \chi$, le morphisme $\bcl\uaut(M)\lmo \bcl\uaut(M_0)$ est une $\ao$-équivalence. Ce dernier morphisme est l'image par le foncteur champ classifiant $\bcl$ du morphisme de schémas en groupe $\uaut(M)\lmo \uaut(M_0)$ dont on note $K$ le noyau. On écrit $M$ comme facteur direct d'un $B$-module libre $B^r=M\oplus N$, où $N$ est un sous-$B$-module de $B^r$ et $r$ un entier. On a $B_0^r=M_0\oplus N_0$. Dans le cas libre, le groupe des automorphismes est le groupe des matrices inversibles et on note $\uaut(B^r)=\ugl_r(B)$. Le noyau du morphisme $\ugl_r(B)\lmo \ugl_r(B_0)$ est donné par le groupe multiplicatif $I_r+\uma_r(I)$ où $I_r$ est la matrice identité de rang $r$ et $\uma_r(I)$ le schéma en groupe additif des matrices à coefficients dans $I$. Le schéma $I_r+\uma_r(I)$ est isomorphe à un espace affine (de dimension $dim_\C(I).r^2$) et est donc $\ao$-contractile. Dans le cas projectif général, on a donc un diagramme dont les lignes sont exactes, 

$$\xymatrix{ 1\ar[r] & K\ar[r]\ar[d]_-k  &\uaut(M)\ar[r]^-g \ar[d]_-i &\uaut(M_0) \ar[r]\ar[d]_-j &1 \\ 1\ar[r] & I_r+\uma_r(I)\ar[r] & \ugl_r(B)\ar[r]^-f & \ugl_r(B_0)\ar[r] & 1 }$$

Le morphisme $i$ envoie un automorphisme de $M$ sur l'automorphisme de $B^r$ qui est égal à l'identité sur le sous-module $N$. Le morphisme $j$ est défini de la même façon. Ces deux morphismes sont des immersions fermées, et le carré correspondant est commutatif : $fi=jg$. On en déduit l'existence du morphisme $k$, qui est lui-même une immersion fermée. Le noyau $K$ s'identifie donc au sous-schéma en groupe des $\C$-automorphismes de $B^r$ qui sont $B$-linéaires, égaux à l'identité sur $N$, et qui appartiennent à $I_r+\uma_r(I)$. Ces trois conditions se traduisent par des équations affines dans l'espace affine $I_r+\uma_r(I)$, et on en déduit que $K$ est lui-même isomorphe à un espace affine d'un certaine dimension, et est donc $\ao$-contractile. Ceci implique que le champ classifiant $\bcl K$ est aussi $\ao$-contractile, c'est-à-dire isomorphe au préfaisceau ponctuel $\ast$ dans $Ho(\spretao)$.

Le foncteur champ classifiant $\bcl$ envoie les suites exactes de groupes sur des suites de fibrations de champs. On a donc une suite de fibration dans $Ho(\spret_\ast)$, 
$$\bcl K\lmo \bcl \uaut(M)\lmo \bcl \uaut(M_0).$$

Etant donné un champ en groupe $G$, il est connu qu'il existe une équivalence de Quillen entre la catégorie de modèles des champs $G$-équivariants et celles des champs au dessus de $\bcl G$ (voir \cite[Lem.3.20]{kptalg}). Ce dernier résultat implique l'existence d'un isomorphisme de champs au dessus de $\bcl \uaut(M_0)$, 
$$\xymatrix{\bcl \uaut(M)\ar[rr]^-\sim \ar[rd] &&  [\bcl K/\uaut(M_0)]\ar[ld] \\ & \bcl \uaut(M_0) & }$$
dans $Ho(\spret/\bcl \uaut(M_0))$, pour une certaine action de $\uaut(M_0)$ sur $\bcl K$, et la notation $[-/-]$ réfère au champ quotient. Etant donné que $\bcl K$ est contractile dans $Ho(\spretao)$ et que le foncteur de localisation $Ho(\spret)\lmo Ho(\spretao)$ commute à la formation des quotients, en passant dans la catégorie $Ho(\spretao)$ ce dernier triangle est isomorphe à 
$$\xymatrix{\bcl \uaut(M)\ar[rr]^-\sim \ar[rd] &&  [\ast/\uaut(M_0)]  = \bcl \uaut(M_0) \ar[ld]^-{id}  \\ & \bcl \uaut(M_0) & }$$
Le morphisme $\bcl \uaut(M)\lmo \bcl \uaut(M_0)$ est donc un isomorphisme dans $Ho(\spretao)$. Ceci achève la preuve du fait que $\vect^B\lmo \vect^{B/I}$ est une $\ao$-équivalence pour tout idéal bilatère de carré nul $I$ et donc la preuve de la proposition \ref{cralgass}. 
\end{proof}

Notons qu'au cours de la preuve de la proposition \ref{cralgass}, on a montré l'invariance de la K-théorie semi-topologique connective par extension infinitésimale exprimée dans la proposition suivante. 

\begin{prop} 
Soit $B$ une algèbre associative de dimension finie sur $\C$ et $I$ un idéal à droite nilpotent de $B$. Alors le morphisme de projection $B\lmo B/I$ induit une équivalence $\kcst(B)\simeq \kcst(B/I)$ dans $Ho(Sp)$. 
\end{prop}

\begin{rema}\label{stab}
La proposition \ref{cralgass} ainsi que la formule (\ref{kcsttot}) permet donc d'exprimer l'homologie périodique d'une algèbre de dimension finie en fonction de l'espace de lacets infini $(\regam{\vect^B_\bul}^+)_1=:\re{\vect^B}^+$. Cette complétion en groupe admet la description suivante. Le monoïde $\pi_0\re{\vect^B}$ n'étant pas en général isomorphe à $\N$ munie de l'addition (comme c'est le cas pour $B$ commutative), le calcul de cette complétion en groupe est un peu plus compliqué que dans le cas commutatif, vu qu'on ne peut pas appliquer le résultat de Quillen qui calcule l'homologie de la complétion en groupe. Comme tout $B$-module projectif de type fini est facteur direct d'un $B$-module libre de type fini, pour compléter la structure de <<monoïde>> définie par la somme directe des modules sur l'espace $\re{\vect^B}$, il suffit d'inverser l'action du $B$-module régulier $B$. Il existe donc un isomorphisme, 
$$\re{\vect^B}^+ \simeq \re{\vect^B}[-B]$$
dans $Ho(SSet)$, où le dernier objet est le niveau $1$ de la localisation au sens des $\gam$-espaces du $\gam$-espace spécial $\re{\vect^B}$ par rapport au $B$-module régulier $B$. On souhaite maintenant dire que cette localisation peut se calculer en effectuant la colimite homotopique standard 
$$\xymatrix{hocolim (\re{\vect^B}\ar[r]^-{\oplus B} & \re{\vect^B}\ar[r]^-{\oplus B} & \re{\vect^B} \ar[r]^-{\oplus B} & \cdots ) }=:\re{\vect^B}^{ST}$$
où le morphisme $\oplus B$ est induit par l'endomorphisme $\uparf(B)$ qui envoie un $B$-dg-module $E$ sur $E\oplus B$ et la notation <<stable>> fait référence à l'homologie stable des algèbres de Lie de matrices de Loday--Quillen. Pour cela il faut utiliser le language des $\infty$-groupoïdes monoïdaux au sens de Lurie plutôt que celui des $\gam$-espaces, c'est pourquoi nous n'entrerons pas dans les détails et nous contenterons de rester volontairement vague sur les définitions utilisées. On considère donc $\re{\vect^B}$ comme un $\infty$-groupoïde monoïdal symétrique (dont la loi est donnée par la somme des modules) et on souhaite inverser l'objet $B$ respectivement à la somme. On applique pour cela \cite[Cor.4.24]{marco1} (qui fonctionne plus généralement pour toutes les $\infty$-catégories monoïdales symétriques présentables, voir aussi \cite[Rem.4.7]{marco1} et \cite[Rem.4.26]{marco1} pour le cas particulier des $\infty$-groupoïdes). Il existe alors un isomorphisme canonique
$$\re{\vect^B}[-B]\simeq \re{\vect^B}^{ST}$$
dans $Ho(SSet)$, pourvu que l'on montre que l'objet $B$ est un objet symétrique au sens de \cite[Déf.4.18]{marco1}. On doit donc montrer qu'il existe une homotopie entre le morphisme 
$$\xymatrix{B\oplus B\oplus B \ar[r]^-{(123)} & B\oplus B\oplus B }$$
et l'identité de $B\oplus B\oplus B$ dans l'espace des points complexes de $\ugl_3(B)$, où $(123)$ est l'automorphisme induit par la permutation cyclique $(123)$. Une telle homotopie est obtenue en composant une homotopie $(123)\Rightarrow id$ dans $\re{\ugl_3(\C)}$ avec le morphisme canonique $\re{\ugl_3(\C)}\lmo \re{\ugl_3(B)}$ induit par le morphisme structural $\C\lmo B$. 

En conclusion, il existe un isomorphisme 
$$\re{\vect^B}^+ \simeq \re{\vect^B}^{ST}$$
dans $Ho(SSet)$ entre la complétion en groupe de $\re{\vect^B}$ et la stabilisation de $\re{\vect^B}$ (où la colimite est comprise au sens homotopique). En prenant les groupes d'homotopie de la formule donnée par la proposition \ref{cralgass}, on obtient le corollaire suivant. 
\end{rema}

\begin{cor}\label{formhp1}
Soit $B$ une $\C$-algèbre associative de dimension finie. Alors pour tout $i\geq 0$ le caractère de Chern $\kctop(B)\lmo \hp(B)$ induit un isomorphisme de $\C$-espaces vectoriels, 
$$ colim_{k\geq 0} \pi_{i+2k} \re{\vect^B}^{ST}\te_\Z \C\simeq\hp_i(B)$$
où la colimite est induite par l'action de l'élément de Bott $\beta$ sur les groupes d'homotopie, $\pi_i\re{\vect^B}^{ST} \lmos{\times \beta} \pi_{i+2} \re{\vect^B}^{ST}$.
\end{cor}

\paragraph{Conséquences dans le cas lisse.} Soit $B$ une $\C$-algèbre associative de dimension finie qui est de plus de dimension globale finie. Cette hypothèse supplémentaire signifie précisemment que $B$ est lisse en tant que dg-catégorie (\cite[Déf.2.4]{modob}). Si $\spec(A)\in \affc$ est supposé \emph{lisse}, alors par le théorème de résolution de Quillen \cite[§4.Cor.1]{qui}, la K-théorie algébrique connective de $\proj(B\te_\C A)$ est la même que la K-théorie algébrique connective de $\psproj(B\te_\C A)$ ou même que la K-théorie algébrique connective de tous les $B\te_\C A$-modules de type fini. En effet comme $B\te_\C A$ est lisse, tout $B\te_\C A$-module à droite de type fini possède une résolution projective finie relativement à $B\te_\C A$. On a donc une équivalence globale de préfaisceaux en spectres restreint à $\afflissc$, 
\begin{equation}\label{qresol}
l^*\kc(\uproj(B))\simeq l^*\kc(\upsproj(B)).
\end{equation}
Par la propriété de restriction aux lisses (théorème \ref{restliss}) on a donc un isomorphisme naturel en $B$ dans $Ho(Sp)$, 
\begin{equation}\label{kpsproj}
\kcst(B)\simeq \resp{\kc(\uproj(B))} \simeq \resp{\kc(\upsproj(B))}
\end{equation}

Par ce qui précède on en déduit la proposition suivante. 

\begin{prop}\label{cons1}
Soit $B$ une $\C$-algèbre associative de dimension finie et de dimension globale finie. Alors on a des isomorphismes dans $Ho(Sp)$, 
$$\kcst(B)\simeq \mcal{B}\regam{\vect^B_\bul}^+\simeq\mcal{B}\regam{\vect_B^\bul}^+\simeq \mcal{B}\regam{\M^B_\bul}\simeq \mcal{B}\regam{\M_B^\bul}.$$
Et donc des isomorphismes dans $Ho(SSet)$, 
$$\re{\vect^B}^+\simeq \re{\vect_B}^+\simeq \re{\M^B}\simeq \re{\M_B},$$
où les deux premiers objets sont les niveaux $1$ de la complétion en groupe. 
\end{prop} 

\begin{proof}   
On a déjà montré l'isomorphisme $\kcst(B)\simeq \mcal{B}\regam{\vect^B_\bul}^+$ dans $Ho(Sp)$ à la formule (\ref{kvect}). D'après la formule (\ref{kpsproj}), on a un isomorphisme  $\kcst(B)\simeq \resp{\kc(\upsproj(B))}$ dans $Ho(Sp)$. En procédant comme dans la preuve du théorème \ref{bu} et de la formule (\ref{kvect}) et en appliquant la proposition \ref{cofsc} à la catégorie de Waldhausen des pseudo-projectifs (laquelle a des cofibrations scindées par définition), on a un isomorphisme $\resp{\kc(\upsproj(B))} \simeq \mcal{B}\regam{\vect_B^\bul}^+$. On a donc un isomorphisme $\kcst(B)\simeq \regam{\vect_B^\bul}^+$. L'isomorphisme $\kcst(B)\simeq \mcal{B}\regam{\M^B_\bul}$ est le théorème \ref{kstmt}. Le théorème de Gillet--Waldhausen (\cite[Thm.1.11.7]{tt}, voir remarque \ref{algebra}) nous donne un isomorphisme au niveau des préfaisceaux de K-théorie algébrique $\kc(\upsproj(B))\simeq \kc(\upspa(B))$ dans $Ho(\spr)$. Par le théorème \ref{mtps}, on en déduit des isomorphismes $\kcst(B)\simeq \resp{\kc(\upsproj(B))} \simeq \resp{\kc(\upspa(B))} \simeq \mcal{B}\regam{\M_B^\bul}$. La deuxième partie de la proposition en découle directement car le foncteur $\mcal{B}$ est une équivalence entre la catégorie homotopique des $\gam$-espaces très spéciaux et la catégorie homotopique des spectres connectifs. 
\end{proof}

De même que dans le cas non nécessairement lisse on a un isomorphisme,
$$\xymatrix{ \re{\vect_B}^+ \simeq hocolim (\re{\vect_B}\ar[r]^-{\oplus B} & \re{\vect_B}\ar[r]^-{\oplus B} &\re{\vect_B}\ar[r]^-{\oplus B} & \cdots ) =: \re{\vect_B}^{ST} }$$
dans $Ho(SSet)$ (voir la remarque \ref{stab}). Nous avons donc le corollaire suivant de la proposition \ref{cralgass}. 

\begin{cor}
Soit $B$ une $\C$-algèbre associative de dimension finie et de dimension globale finie. Alors pour tout $i\geq 0$ le caractère de Chern $\kctop(B)\lmo \hp(B)$ induit un isomorphisme de $\C$-espaces vectoriels,
$$ colim_{k\geq 0} \pi_{i+2k} \re{\vect_B}^{ST}\te_\Z \C  \simeq \hp_i(B)$$
où la colimite est induite par l'action de l'élément de Bott $\beta$ sur les groupes d'homotopie, $\pi_i \re{\vect_B}^{ST}\lmos{\times \beta} \pi_{i+2} \re{\vect_B}^{ST}$. 
\end{cor}

 \newpage
\begin{small}
\bibliographystyle{amsalpha}
\bibliography{/Users/Anthony/Dropbox/MATH/Latex/Bibliographie/ref}

\newcommand{\etalchar}[1]{$^{#1}$}
\providecommand{\bysame}{\leavevmode\hbox to3em{\hrulefill}\thinspace}
\providecommand{\MR}{\relax\ifhmode\unskip\space\fi MR }
\providecommand{\MRhref}[2]{%
  \href{http://www.ams.org/mathscinet-getitem?mr=#1}{#2}
}
\providecommand{\href}[2]{#2}
\begin{thebibliography}{FOOO00}

\bibitem[Abo12]{abouwrap}
M.~Abouzaid, \emph{On the wrapped {F}ukaya category and based loops}, Journal
  of Symplectic Geometry \textbf{10} (2012), no.~1, 27--79.

\bibitem[Bar07]{barwenr}
C.~Barwick, \emph{On (enriched) left {B}ousfield localization of model
  categories}, Arxiv preprint arXiv:0708.2067 (2007).

\bibitem[Bas68]{bass}
Hyman Bass, \emph{Algebraic {K}-theory}, WA Benjamin New York, 1968.

\bibitem[BF78]{bf}
A.~Bousfield and E.~Friedlander, \emph{Homotopy theory of gamma-spaces,
  spectra, and bisimplicial sets}, Geometric applications of homotopy theory II
  (1978), 80--130.

\bibitem[Bla01]{blande}
Benjamin~A Blander, \emph{Local projective model structures on simplicial
  presheaves}, K-theory \textbf{24} (2001), no.~3, 283--301.

\bibitem[BO02]{bondalorlovicm}
A~Bondal and D~Orlov, \emph{Derived categories of coherent sheaves},
  International Congress of Mathematicians, 2002, p.~47.

\bibitem[BR08]{bakerrichter}
Andrew Baker and Birgit Richter, \emph{Uniqueness of {E} infinity structures
  for connective covers}, Proceedings of the American Mathematical Society
  \textbf{136} (2008), no.~2, 707--714.

\bibitem[BVdB03]{bvgen}
A.I. Bondal and M.~Van~den Bergh, \emph{Generators and representability of
  functors in commutative and noncommutative geometry}, Moscow Mathematical
  Journal \textbf{3} (2003), no.~1, 1--36.

\bibitem[Cis13]{cisdes}
Denis-Charles Cisinski, \emph{Descente par {\'e}clatements en {K}-th{\'e}orie
  invariante par homotopie}, Annals of Mathematics \textbf{177} (2013),
  425--448.

\bibitem[CT]{tabcisym}
D.C. Cisinski and G.~Tabuada, \emph{Symmetric monoidal structure on
  non-commutative motives}, Journal of K-theory: K-theory and its Applications
  to Algebra, Geometry, and Topology \textbf{1}, no.~1, 1--68.

\bibitem[CT11]{nck}
\bysame, \emph{Non-connective {K}-theory via universal invariants}, Compositio
  Mathematica \textbf{147} (2011), no.~04, 1281--1320.

\bibitem[Del74]{deligneIII}
P.~Deligne, \emph{Th{\'e}orie de {H}odge {III}}, Publications Math{\'e}matiques
  de l'IH{\'E}S \textbf{44} (1974), no.~1, 5--77.

\bibitem[DHI04]{dhi}
D.~Dugger, S.~Hollander, and D.C. Isaksen, \emph{Hypercovers and simplicial
  presheaves}, Mathematical Proceedings of the Cambridge Philosophical Society,
  vol. 136, Cambridge Univ Press, 2004, pp.~9--51.

\bibitem[DI01]{hyptop}
D.~Dugger and D.C. Isaksen, \emph{Hypercovers in topology}, Arxiv preprint
  math/0111287 (2001).

\bibitem[Dug01]{duguni}
D.~Dugger, \emph{Universal homotopy theories}, Advances in Mathematics
  \textbf{164} (2001), no.~1, 144--176.

\bibitem[Dyc11]{tobimf}
T.~Dyckerhoff, \emph{Compact generators in categories of matrix
  factorizations}, Duke Mathematical Journal \textbf{159} (2011), no.~2,
  223--274.

\bibitem[Efi12]{efimovmf}
Alexander~I Efimov, \emph{Cyclic homology of categories of matrix
  factorizations}, arXiv preprint arXiv:1212.2859 (2012).

\bibitem[Eis80]{eisenbudmf}
David Eisenbud, \emph{Homological algebra on a complete intersection, with an
  application to group representations}, Transactions of the American
  Mathematical Society \textbf{260} (1980), no.~1, 35--64.

\bibitem[EM06]{perm}
A.D. Elmendorf and M.A. Mandell, \emph{Rings, modules, and algebras in infinite
  loop space theory}, Advances in Mathematics \textbf{205} (2006), no.~1,
  163--228.

\bibitem[FM94]{fmfilt}
E.M. Friedlander and B.~Mazur, \emph{Filtrations on the homology of algebraic
  varieties}, no. 529, American Mathematical Soc., 1994.

\bibitem[FOOO00]{fooo}
Kenji Fukaya, Yong-Geun Oh, Hiroshi Ohta, and Kaoru Ono, \emph{Lagrangian
  intersection floer theory}, American Mathematical Soc., 2000.

\bibitem[Fre09]{freed}
D.S. Freed, \emph{Remarks on {C}hern-{S}imons theory}, Bull. Amer. Math. Soc
  \textbf{46} (2009), 221--254.

\bibitem[FW01]{fwcomp}
E.~Friedlander and M.~Walker, \emph{Comparing {K}-theories for complex
  varieties}, American Journal of Mathematics (2001), 779--810.

\bibitem[Goo85]{goodwilliecyclic}
Thomas~G Goodwillie, \emph{Cyclic homology, derivations, and the free
  loopspace}, Topology \textbf{24} (1985), no.~2, 187--215.

\bibitem[GR71]{sga1}
A.~Grothendieck and M.~Raynaud, \emph{Rev{\^e}tements {\'e}tales et groupe
  fondamental: S{\'e}minaire de g{\'e}om{\'e}trie alg{\'e}brique du {B}ois
  {M}arie 1960-61 {SGA} 1}, vol. 224, Springer, 1971.

\bibitem[Hir75]{hironakatri}
Heisuke Hironaka, \emph{Triangulations of algebraic sets}, Algebraic geometry
  (Proc. Sympos. Pure Math., Vol. 29, Humboldt State Univ., Arcata, Calif.,
  1974), no.~29, Amer. Math. Soc. Providence, RI, 1975, pp.~165--185.

\bibitem[Hir09]{hirs}
P.S. Hirschhorn, \emph{Model categories and their localizations}, AMS
  Bookstore, 2009.

\bibitem[Hov98]{hoveymon}
Mark Hovey, \emph{Monoidal model categories}, arXiv preprint math/9803002
  (1998).

\bibitem[Hov99]{hoveymod}
M.~Hovey, \emph{Model categories}, Mathematical Surveys and Monographs,
  vol.~63, American Mathematical Society, Providence, RI, 1999.

\bibitem[HRS96]{hrstilt}
Dieter Happel, Idun Reiten, and Sverre~O Smal{\o}, \emph{Tilting in abelian
  categories and quasitilted algebras}, no. 575, American Mathematical Soc.,
  1996.

\bibitem[HSS00]{hss}
M.~Hovey, B.~Shipley, and J.~Smith, \emph{{Symmetric spectra}}, Journal of the
  {AMS} \textbf{13} (2000), no.~1, 149--208.

\bibitem[Jar87]{jardsimp}
John~Frederick Jardine, \emph{Simplical presheaves}, Journal of Pure and
  Applied Algebra \textbf{47} (1987), no.~1, 35--87.

\bibitem[Kal10]{kal}
D.~Kaledin, \emph{Motivic structures in non-commutative geometry}, Proceeding
  of the ICM, vol. 901, 2010, pp.~461--496.

\bibitem[Kas87]{kasselcyclic}
C.~Kassel, \emph{Cyclic homology, comodules, and mixed complexes}, Journal of
  Algebra \textbf{107} (1987), no.~1, 195--216.

\bibitem[Kel98]{kelinv}
Bernhard Keller, \emph{Invariance and localization for cyclic homology of dg
  algebras}, Journal of pure and applied Algebra \textbf{123} (1998), no.~1,
  223--273.

\bibitem[Kel99]{kelcyclic}
B.~Keller, \emph{On the cyclic homology of exact categories}, Journal of Pure
  and Applied Algebra \textbf{136} (1999), no.~1, 1--56.

\bibitem[Kel06]{kedg}
\bysame, \emph{On differential graded categories}, International {C}ongress of
  {M}athematicians. {V}ol. {II}, Eur. Math. Soc., Z\"urich, 2006, pp.~151--190.

\bibitem[Kel08]{kellerclus}
Bernhard Keller, \emph{Cluster algebras, quiver representations and
  triangulated categories}, arXiv preprint arXiv:0807.1960 (2008).

\bibitem[KKP08]{kkp}
L.~Katzarkov, M.~Kontsevich, and T.~Pantev, \emph{Hodge theoretic aspects of
  mirror symmetry}, Arxiv preprint arxiv:0806.0107 (2008).

\bibitem[Kon]{kontssymp}
M.~Kontsevich, \emph{Symplectic geometry of homological algebra}, available at
  the author's webpage.

\bibitem[Kon94]{kontshom}
Maxim Kontsevich, \emph{Homological algebra of mirror symmetry}, arXiv preprint
  alg-geom/9411018 (1994).

\bibitem[Kon01]{kontsdefq}
\bysame, \emph{Deformation quantization of algebraic varieties}, Letters in
  Mathematical Physics \textbf{56} (2001), no.~3, 271--294.

\bibitem[KPT{\etalchar{+}}09]{kptalg}
Ludmil Katzarkov, Tony Pantev, Bertrand To{\"e}n, et~al., \emph{Algebraic and
  topological aspects of the schematization functor}, Compos. Math \textbf{145}
  (2009), no.~3, 633--686.

\bibitem[Kra08]{krauserep}
Henning Krause, \emph{Representations of quivers via reflection functors},
  arXiv preprint arXiv:0804.1428 (2008).

\bibitem[KS06]{kontssoi}
M.~Kontsevich and Y.~Soibelman, \emph{Notes on {A}-infinity algebras,
  {A}-infinity categories and non-commutative geometry. i}, Arxiv preprint
  math/0606241 (2006).

\bibitem[Lur09]{htt}
J.~Lurie, \emph{Higher topos theory}.

\bibitem[MT12]{tabmarjac}
Matilde Marcolli and Goncalo Tabuada, \emph{Jacobians of noncommutative
  motives}, arXiv preprint arXiv:1212.1118 (2012).

\bibitem[MV99]{mv}
F.~Morel and V.~Voevodsky, \emph{A1-homotopy theory of schemes}, Publications
  Math{\'e}matiques de l'IHES \textbf{90} (1999), no.~1, 45--143.

\bibitem[Orl03]{orlovdereq}
Dmitri~Olegovich Orlov, \emph{Derived categories of coherent sheaves and
  equivalences between them}, Russian Mathematical Surveys \textbf{58} (2003),
  no.~3, 511.

\bibitem[Orl05]{orlovderivedcatsing}
Dmitri Orlov, \emph{Derived categories of coherent sheaves and triangulated
  categories of singularities}, arXiv preprint math/0503632 (2005).

\bibitem[Pan11]{panditth}
P.~Pandit, \emph{Moduli problems in derived noncommutative geometry}, PhD
  thesis, University of Pennsylvania.

\bibitem[Pet12]{petitdgaff}
Fran{\c{c}}ois Petit, \emph{{DG} affinity of {DQ}-modules}, International
  Mathematics Research Notices \textbf{2012} (2012), no.~6, 1414--1438.

\bibitem[PS04]{poleshapidefq}
Pietro Polesello and Pierre Schapira, \emph{Stacks of quantization-deformation
  modules on complex symplectic manifolds}, International Mathematics Research
  Notices \textbf{2004} (2004), no.~49, 2637--2664.

\bibitem[Qui]{qui}
D.~Quillen, \emph{{Higher algebraic {K}-theory. I. Algebraic {K}-theory, I:
  Higher {K}-theories (Proc. Conf., Battelle Memorial Inst., Seattle, Wash.,
  1972)}}, Lecture Notes in Math \textbf{341}, 85--147.

\bibitem[Ric03]{richter}
Birgit Richter, \emph{Symmetry properties of the {D}old-{K}an correspondence},
  Mathematical Proceedings of the Cambridge Philosophical Society, vol. 134,
  Cambridge Univ Press, 2003, pp.~95--102.

\bibitem[Rio05]{rioudual}
Jo{\"e}l Riou, \emph{Dualit{\'e} de {S}panier--{W}hitehead en g{\'e}om{\'e}trie
  alg{\'e}brique}, Comptes Rendus Mathematique \textbf{340} (2005), no.~6,
  431--436.

\bibitem[Rob12]{marco1}
M.~Robalo, \emph{Noncommutative {M}otives {I}: {A} universal characterization
  of the motivic stable homotopy theory of schemes}, arXiv preprint
  arXiv:1206.3645 (2012).

\bibitem[Ros94]{rosenbergk}
Jonathan Rosenberg, \emph{Algebraic {K}-theory and its applications}, vol. 147,
  Springer, 1994.

\bibitem[Sch]{schw-sym}
S.~Schwede, \emph{An untitled book project about symmetric spectra}, available
  on the author's web page.

\bibitem[Sch06]{schl}
M.~Schlichting, \emph{Negative {K}-theory of derived categories}, Mathematische
  Zeitschrift \textbf{253} (2006), no.~1, 97--134.

\bibitem[Sch08]{shapiradef}
Pierre Schapira, \emph{Deformation quantization modules on complex symplectic
  manifolds}, Poisson geometry in mathematics and physics, Contemp. Math., vol.
  450, Amer. Math. Soc., Providence, RI, 2008, pp.~259--271. \MR{2397629
  (2009f:53150)}

\bibitem[Seg74]{seg}
G.~Segal, \emph{Categories and cohomology theories}, Topology \textbf{13}
  (1974), no.~3, 293--312.

\bibitem[Sei08]{seidelfuk}
P~Seidel, \emph{Fukaya categories and {P}icard-{L}efschetz theory, zurich
  lect}, Adv. Math., European Math. Soc., Z{\"u}rich (2008).

\bibitem[sga72]{sga4-1}
\emph{Th\'eorie des topos et cohomologie \'etale des sch\'emas. {T}ome 1:
  {T}h\'eorie des topos}, Lecture Notes in Mathematics, Vol. 269,
  Springer-Verlag, Berlin, 1972, S{\'e}minaire de G{\'e}om{\'e}trie
  Alg{\'e}brique du Bois-Marie 1963--1964 (SGA 4), Dirig{\'e} par M. Artin, A.
  Grothendieck, et J. L. Verdier. Avec la collaboration de N. Bourbaki, P.
  Deligne et B. Saint-Donat.

\bibitem[Shk07]{shkly}
D.~Shklyarov, \emph{Hirzebruch-{R}iemann-{R}och theorem for dg-algebras}, arXiv
  preprint arXiv:0710.1937 (2007).

\bibitem[Sim96]{simpsontop}
Carlos Simpson, \emph{The topological realization of a simplicial presheaf},
  arXiv preprint q-alg/9609004 (1996).

\bibitem[SS00]{ss-alg}
S.~Schwede and B.E. Shipley, \emph{Algebras and modules in monoidal model
  categories}, Proceedings of the London Mathematical Society \textbf{80}
  (2000), no.~2, 491.

\bibitem[SS03]{shischcat}
Stefan Schwede and Brooke Shipley, \emph{Stable model categories are categories
  of modules}, Topology \textbf{42} (2003), no.~1, 103--153.

\bibitem[SV96]{sv2}
A.~Suslin and V.~Voevodsky, \emph{Singular homology of abstract algebraic
  varieties}, Inventiones mathematicae \textbf{123} (1996), no.~1, 61--94.

\bibitem[SV00]{sv1}
\bysame, \emph{Bloch-{K}ato conjecture and motivic cohomology with finite
  coefficients}, Nato ASI Series C Mathematical and Physical Sciences
  \textbf{548} (2000), 117--192.

\bibitem[Tab07]{tabth}
G.~Tabuada, \emph{Th{\'e}orie homotopique des dg-cat{\'e}gories}, Ph.D. thesis,
  arXiv:0710.4303v1 [math.KT], 2007.

\bibitem[Tab08]{tabhkt}
\bysame, \emph{Higher {$K$}-theory via universal invariants}, Duke Math. J.
  \textbf{145} (2008), no.~1, 121--206.

\bibitem[Tab11]{tabgui}
Goncalo Tabuada, \emph{A guided tour through the garden of noncommutative
  motives}, arXiv preprint arXiv:1108.3787 (2011).

\bibitem[Tab12]{tabpro}
Gon{\c{c}}alo Tabuada, \emph{Products, multiplicative {C}hern characters, and
  finite coefficients via noncommutative motives}, Journal of Pure and Applied
  Algebra (2012).

\bibitem[Tho85]{thomet}
R.W. Thomason, \emph{Algebraic {K}-theory and {\'e}tale cohomology}, Ann. Sci.
  {\'E}cole Norm. Sup.(4) \textbf{18} (1985), no.~3, 437--552.

\bibitem[To{\"e}02]{galhom}
B.~To{\"e}n, \emph{Vers une interpr{\'e}tation galoisienne de la th{\'e}orie de
  l'homotopie}, Cahiers de topologie et g{\'e}om{\'e}trie diff{\'e}rentielle
  cat{\'e}goriques \textbf{43} (2002), no.~4, 257--312.

\bibitem[To{\"e}06]{dhall}
\bysame, \emph{Derived {H}all algebras}, Duke Mathematical Journal \textbf{135}
  (2006), no.~3, 587--615.

\bibitem[To{\"e}07]{dgmor}
\bysame, \emph{The homotopy theory of {$dg$}-categories and derived {M}orita
  theory}, Invent. Math. \textbf{167} (2007), no.~3, 615--667.

\bibitem[To{\"e}10a]{ldgcat}
\bysame, \emph{Lectures on dg-categories}, Topics in Algebraic and Topological
  K-theory, Lectures notes in math., vol. 2008, Springer, 2010.

\bibitem[To{\"e}10b]{sat}
\bysame, \emph{Lectures on saturated $dg$-categories},
  \url{http://www-math.mit.edu/~auroux/frg/miami10-notes}, January 2010,
  Hand-written notes by D. Auroux.

\bibitem[Tsy83]{tsyhom}
B.~Tsygan, \emph{The homology of matrix {L}ie algebras over rings and the
  {H}ochschild homology}, Russian Mathematical Surveys \textbf{38} (1983),
  no.~2, 198--199.

\bibitem[Tsy07]{tsygm}
\bysame, \emph{On the {G}auss-{M}anin connection in cyclic homology}, Methods
  Funct. Anal. Topology \textbf{13} (2007), no.~1, 83--94.

\bibitem[TT90]{tt}
R.~W. Thomason and T.~Trobaugh, \emph{Higher algebraic {$K$}-theory of schemes
  and of derived categories}, The {G}rothendieck {F}estschrift, {V}ol.\ {III},
  Progr. Math., vol.~88, Birkh\"auser Boston, Boston, MA, 1990, pp.~247--435.

\bibitem[TV05]{hag1}
B.~To{\"e}n and G.~Vezzosi, \emph{Homotopical algebraic geometry {I}: topos
  theory}, Advances in mathematics \textbf{193} (2005), no.~2, 257--372.

\bibitem[TV07]{modob}
B.~To{\"e}n and M.~Vaqui{\'e}, \emph{Moduli of objects in dg-categories}, Ann.
  Sci. \'Ecole Norm. Sup. (4) \textbf{40} (2007), no.~3, 387--444.

\bibitem[TV08]{hag2}
B.~To{\"e}n and G.~Vezzosi, \emph{Homotopical algebraic geometry. {II}.
  {G}eometric stacks and applications.}, Memoirs of the American Mathematical
  Society \textbf{193} (2008), no.~902.

\bibitem[TV09]{caract}
\bysame, \emph{Caract{\`e}res de {C}hern, traces {\'e}quivariantes et
  g{\'e}om{\'e}trie alg{\'e}brique d{\'e}riv{\'e}e}, Arxiv preprint
  arXiv:0903.3292 (2009).

\bibitem[Wal85]{wald}
F.~Waldhausen, \emph{Algebraic {$K$}-theory of spaces}, Algebraic and geometric
  topology ({N}ew {B}runswick, {N}.{J}., 1983), Lecture Notes in Math., vol.
  1126, Springer, Berlin, 1985, pp.~318--419.

\end{thebibliography}
\end{small}

\end{document}